\pdfoutput=1
\documentclass[11pt]{article}

\usepackage[utf8]{inputenc}
\usepackage[pdftex]{graphicx,xcolor}
\usepackage{array,bookmark,fullpage,enumitem,microtype}
\usepackage[leqno]{amsmath}
\usepackage{amssymb,amsthm,mathtools,bm,stmaryrd,tikz-cd,extdash}
\usepackage[nottoc]{tocbibind}
\usepackage{hyperref}
\hypersetup{pdftex,colorlinks=true,allcolors=blue,unicode,bookmarksnumbered,psdextra}
\usepackage[capitalize,nosort]{cleveref}

\usepackage{mymacros}

\let\Dia\Diamond
\newcommand*\graph{\mathrm{graph}}
\newcommand*\core{\mathrm{core}}
\newcommand*\Sh{\!{Sh}}
\newcommand*\Hom{\mathrm{Hom}}
\newcommand*\Iso{\mathrm{Iso}}
\newcommand*\End{\mathrm{End}}
\newcommand*\Mod{\mathrm{Mod}}

\DeclarePairedDelimiter\sqsqbr{\llbracket}{\rrbracket}
\newcommand\relmiddle[1]{\nonscript\;\middle#1\nonscript\;}

\let\defn\textbf
\let\olddisplaystyle\displaystyle
\let\displaystyle\textstyle

\begin{document}

\title{Representing Polish groupoids via metric structures}
\author{Ruiyuan Chen}
\date{}
\maketitle

\begin{abstract}
We prove that every open $\sigma$-locally Polish groupoid $G$ is Borel equivalent to the groupoid of models on the Urysohn sphere $\#U$ of an $\@L_{\omega_1\omega}$-sentence in continuous logic.  In particular, the orbit equivalence relations of such groupoids are up to Borel bireducibility precisely those of Polish group actions, answering a question of Lupini.  Analogously, every non-Archimedean (i.e., every unit morphism has a neighborhood basis of open subgroupoids) open quasi-Polish groupoid is Borel equivalent to the groupoid of models on $\#N$ of an $\@L_{\omega_1\omega}$-sentence in discrete logic.

The proof in fact gives a topological representation of $G$ as the groupoid of isomorphisms between a ``continuously varying'' family of structures over the space of objects of $G$, constructed via a topological Yoneda-like lemma of Moerdijk for localic groupoids and its metric analog.  Other ingredients in our proof include the Lopez-Escobar theorem for continuous logic, a uniformization result for full Borel functors between open quasi-Polish groupoids, a uniform Borel version of Katětov's construction of $\#U$, groupoid versions of the Pettis and Birkhoff--Kakutani theorems, and a development of the theory of non-Hausdorff topometric spaces and their quotients.
\end{abstract}

\tableofcontents

\section{Introduction}
\label{sec:intro}

This paper is a contribution to the classification theory of definable equivalence relations and group actions.
Our main result is a representation theorem for certain topological groupoids which generalize continuous actions of \defn{Polish groups} (separable, completely metrizable topological groups); our result shows that in fact, these groupoids are no more general, by representing them as groupoids of isomorphisms between metric structures, hence induced by the logic action.

It is by now well-known that Polish group actions correspond in a precise sense to isomorphism problems for definable classes of countable first-order structures and their separable metric analogs.  Classically, this holds for \defn{non-Archimedean} Polish groups $G$ (i.e., $G$ admits a countable neighborhood basis of open subgroups at the identity, or equivalently $G$ embeds as a closed subgroup of the infinite symmetric group $\#S_\infty$).  Such $G$ are, up to isomorphism, precisely the automorphism groups of countable first-order structures (e.g., countable groups, graphs, linear orders), while the orbit equivalence relations of their actions on Polish spaces are precisely the isomorphism relations between classes of countable structures which are Borel-definable in the sense of being axiomatized by a countable theory in the countably infinitary first-order logic $\@L_{\omega_1\omega}$.  Moreover, natural topological and descriptive properties of the action itself correspond to natural model-theoretic properties of the class of structures.  See e.g., \cite{Gao}, \cite{BK}, \cite{HK}, \cite{Bec} for details.

The extension of this correspondence to general Polish groups is naturally phrased in terms of \defn{continuous logic for metric structures}, systematically formalized and studied by Ben~Yaacov, Henson, and others over the past decade as the ``correct'' analog to classical first-order logic suitable for studying structures based on an underlying complete metric space (e.g., Banach spaces).  See \cite{BBHU} for an introduction to and survey of first-order continuous logic, and \cite{BDNT}, \cite{CL} for the $\@L_{\omega_1\omega}$ theory.  As noted implicitly already by Gao--Kechris \cite{GK}, and explicated in more detail by Melleray \cite{Me2}, Ivanov--Majcher-Iwanow \cite{IMI}, and others, the above correspondence between non-Archimedean Polish group actions and classical model theory extends largely as expected to arbitrary Polish groups and continuous model theory.  In particular, every Polish group is the automorphism group of a separable metric structure; and arbitrary Polish group actions correspond to isomorphism problems for $\@L_{\omega_1\omega}$-definable classes of metric structures.

More precisely, for a countable first-order language $\@L$, isomorphism on classes of $\@L$-structures can be induced by a Polish group action as follows.  For a fixed countable set (resp., Polish metric space) $X$, the space $\Mod_X(\@L)$ of (resp., metric) $\@L$-structures on underlying universe $X$ has a natural Polish topology, on which the symmetric group $\#S_X$ (resp., the isometry group $\Iso(X)$) acts continuously via pushforward of structures; this is called the \defn{logic action}.  The orbit equivalence relation is exactly the isomorphism relation between $\@L$-structures on $X$.  Similarly, for a countable $\@L_{\omega_1\omega}$-theory (which we may assume is a single sentence) $\phi$, the action of $\#S_X$ (resp., $\Iso(X)$) on the Borel subspace $\Mod_X(\@L, \phi) \subseteq \Mod_X(\@L)$ of models of $\phi$ induces the isomorphism relation between models of $\phi$ on $X$.  By the universality of the logic action (see \cite[2.7.3]{BK}, \cite[1.3]{IMI}) and the Lopez-Escobar theorem (\cite{LE}, \cite{CL}), every non-Archimedean (resp., general) Polish group action arises in this way.

It is natural to consider what happens when we allow the underlying universe $X$ of the structures to vary.  This leads to the setting of \defn{groupoids}, generalizations of groups and group actions consisting of a set of objects and a set of abstract ``morphisms'' between each pair of objects together with multiplication, unit, and inverse operations satisfying the usual axioms.  A \defn{topological groupoid} is a groupoid whose sets of objects and (all) morphisms are topological spaces and whose groupoid operations are continuous; it is common to restrict attention to topological groupoids which are \defn{open}, meaning the multiplication operation is an open map.  An example of an open Polish groupoid is given by the \defn{action groupoid} $G \ltimes X$ of a Polish group action $G \curvearrowright X$, with object space $X$ and morphisms $g : x -> y$ consisting of all $g \in G$ such that $g \cdot x = y$.
Open Polish groupoids were introduced by Ramsay \cite{Ram} and extensively studied by Lupini \cite{Lup}, who proved groupoid analogs of many descriptive set-theoretic results for Polish group actions, and asked:

\begin{question}[Lupini {\cite[9.1]{Lup}}]
\label{qu:lupini}
Does there exist an open Polish groupoid $G$ whose induced orbit equivalence relation $\#E_G$ is not Borel bireducible with that of a Polish group action?
\end{question}

Here by the \defn{orbit equivalence relation} $\#E_G$ of a groupoid, we mean the connectedness relation between its objects.  In invariant descriptive set theory, the natural ranking of complexity between definable equivalence relations $E, F$ on standard Borel spaces $X, Y$ respectively is given by \defn{Borel reducibility} (see e.g., \cite{Gao}): we say that $E$ \defn{Borel reduces} to $F$ if there is a Borel map $X -> Y$ descending to an injection $X/E -> Y/F$; and $E, F$ are \defn{Borel bireducible} if each Borel reduces to the other.  Note that Borel bireducibility between orbit equivalence relations of groupoids is implied by Borel isomorphism (or more generally equivalence) of the groupoids.

For example, for any countable language $\@L$, there is a natural open Polish groupoid $\!{Mod}_{\le\#N}(\@L)$ consisting of $\@L$-structures on an initial segment of $\#N$ and all isomorphisms between them, which captures isomorphisms between all countable $\@L$-structures (including finite ones) and contains the action groupoid of the logic action $\#S_\infty \curvearrowright \Mod_\#N(\@L)$.  Since a standard procedure of adding constants replaces every finite structure with a countably infinite one, while preserving (non-)isomorphism, it is easily seen that $\!{Mod}_{\le\#N}(\@L)$ is in fact Borel equivalent to $\#S_\infty \ltimes \Mod_\#N(\@L', \phi)$ for some sentence $\phi$.  Thus, allowing varying universes for models does not change the possible complexities of the isomorphism relation; in particular, no $G = \!{Mod}_{\le\#N}(\@L)$ can positively answer \cref{qu:lupini}.

In a similar vein, there is a natural open Polish groupoid $\!{Mod}_{\le\#U}(\@L)$ parametrizing all metric $\@L$-structures (of bounded diameter $\le 1$, say) and isomorphisms between them.  The role of $\#N$ in the discrete setting is here played by the \defn{Urysohn sphere} $\#U$, the unique-up-to-isometry ultrahomogeneous Polish metric space of diameter $\le 1$ containing an isometric copy of every other Polish metric space of diameter $\le 1$; see \cite{Me1}, \cite{Usp}, and \cref{sec:urysohn} for details.  Objects of $\!{Mod}_{\le\#U}(\@L)$ consist of a closed subspace $X \subseteq \#U$ together with a metric $\@L$-structure on $X$;%
\footnote{The topology on $\!{Mod}_{\le\#U}(\@L)$ can be defined via the embedding $F |-> d(F, -)$ from the Effros Borel space of $\#U$ to $\#R^\#U$; we omit the details, which are irrelevant for our purposes.}
by universality of $\#U$, these include an isomorphic copy of every separable metric $\@L$-structure.
Here it is less clear how every structure in $\!{Mod}_{\le\#U}(\@L)$ can be uniformly replaced with a structure on some fixed metric space in a Borel way.  That this can be done was shown by Elliott--Farah--Paulsen--Rosendal--Toms--Törnquist \cite{EFPRTT}, who (essentially) used Katětov's \cite{Kat} construction of $\#U$ to extend every separable metric $\@L$-structure to a structure on $\#U$ (see \cref{sec:katetov}).  As a result, just like in the discrete setting, considering groupoids of isomorphisms between metric structures does not in fact confer any extra generality over the simpler context of the logic action for a fixed universe, at least as far as the orbit equivalence relations are concerned.

The main results of this paper show that the above reductions to group actions hold abstractly for any open Polish groupoid, and not just for the orbit equivalence relation but for the groupoid itself.  In particular, we answer \cref{qu:lupini} in the negative.

\begin{theorem}[\cref{thm:locpolgpd-rep}]
\label{thm:intro-locpolgpd-rep}
For any open Polish groupoid $G$, there is a countable language $\@L$ and a metric $\@L_{\omega_1\omega}$-sentence $\phi$ such that $G$ is Borel equivalent to the groupoid $\Iso(\#U) \ltimes \Mod_\#U(\@L, \phi)$ of models of $\phi$ on $\#U$.  In particular, $\#E_G$ is Borel bireducible with the orbit equivalence relation of the logic action $\Iso(\#U) \curvearrowright \Mod_\#U(\@L, \phi)$.
\end{theorem}

We also have a version of this result in the setting of non-Archimedean groupoids/discrete structures.  We say that a topological groupoid $G$ is \defn{non-Archimedean} if every unit morphism in $G$ has a neighborhood basis (in the space of morphisms) consisting of open subgroupoids; this includes the action groupoid of every continuous action of a non-Archimedean group.

\begin{theorem}[\cref{thm:nonarchgpd-rep}]
\label{thm:intro-nonarchgpd-rep}
For any open non-Archimedean Polish groupoid $G$, there is a countable language $\@L$ and a (discrete) $\@L_{\omega_1\omega}$-sentence $\phi$ such that $G$ is Borel equivalent to the groupoid $\#S_\infty \ltimes \Mod_\#N(\@L, \phi)$ of models of $\phi$ on $\#N$.  In particular, $\#E_G$ is Borel bireducible with the orbit equivalence relation of the logic action $\#S_\infty \curvearrowright \Mod_\#N(\@L, \phi)$.
\end{theorem}

In fact, we will prove these results for slightly more general classes of groupoids: \cref{thm:intro-locpolgpd-rep} holds for open topological groupoids which are \defn{$\sigma$-locally Polish} (i.e., admit a countable cover by open Polish subspaces),%
\footnote{Note that Lupini \cite{Lup} uses ``Polish groupoid'' to refer to a notion intermediate between what we are calling ``open Polish groupoid'' and ``open $\sigma$-locally Polish groupoid''.}
while \cref{thm:intro-nonarchgpd-rep} holds for de~Brecht's \cite{deB} even more general (not necessarily even $T_1$) \defn{quasi-Polish spaces} (see \cref{sec:qpol} for the definition).

We also note that the proofs of \cref{thm:intro-locpolgpd-rep,thm:intro-nonarchgpd-rep} will show that the Borel equivalence can be taken to be an injective-on-objects functor from $G$ in each case (see \cref{rmk:locpolgpd-rep-io,rmk:nonarchgpd-rep-io}), whence the reductions from $\#E_G$ to the logic actions can be taken to be embeddings (i.e., injective).

\subsection{Topological representation}
\label{sec:intro-toprep}

In the rest of this Introduction, we outline the proofs of \cref{thm:intro-locpolgpd-rep,thm:intro-nonarchgpd-rep}, which use several intermediate results which we believe to be interesting in their own right.  Both proofs naturally factor into two parts: a \emph{topological representation} theorem of the groupoid in question as the groupoid of isomorphisms between a ``continuously varying family'' of structures (whose universes also ``vary continuously''), followed by a \emph{Borel uniformization} argument to turn the groupoid into the logic action for a fixed universe.

The topological representation theorems are inspired by ideas from topos theory, in particular the theory of localic groupoid representations (see \cite[C5]{Jeleph}, \cite{Mo1}).  In the non-Archimedean/discrete case (\cref{thm:intro-nonarchgpd-rep}), the notion of ``continuously varying family'' of discrete $\@L$-structures over a topological space $X$ is well-known (under the name ``sheaf of structures'') in topos theory and categorical logic.  We prefer one of several equivalent definitions, that of an \defn{étale $\@L$-structure} over $X$, consisting informally of a topological $\@L$-structure $\@M$ together with a projection map $p : \@M -> X$ such that the assignment of fibers $x \in X |-> \@M_x := p^{-1}(x)$ is ``continuous''; see \cref{sec:etale} for details.  Given an étale $\@L$-structure $\@M$ over $X$, one can define its \defn{isomorphism groupoid} $\Iso_X(\@M)$, with object space $X$ and morphisms $x -> y$ consisting of isomorphisms $\@M_x \cong \@M_y$, and with topology imitating the usual definition of the pointwise convergence topology on the automorphism group $\Aut(\@M)$ of a single $\@L$-structure $\@M$.  All of these notions are essentially standard in topos theory,%
\footnote{We will relegate discussion of connections with topos theory to the footnotes.}
although we give a concrete, self-contained exposition in \crefrange{sec:etalesp}{sec:etalestr}, partly to make this paper more accessible, and partly to emphasize the analogy with the metric setting.

In the metric setting, as there does not appear to be a well-developed theory of continuous categorical logic (see \cite{AH}, \cite{Cho} for some recent work in this area), we define in \cref{sec:metale} an \emph{ad hoc} analog of étale $\@L$-structures, called \defn{metric-étale $\@L$-structures}, to capture the idea of a ``continuously varying family'' of metric $\@L$-structures over a space $X$.  The definition is naturally phrased in the language of \defn{grey sets} and \defn{topometric spaces} due to Ben~Yaacov, Melleray, and others (see \cite{BYM}), commonly used in continuous logic; we need a non-Hausdorff version of the theory, which we develop in \cref{sec:greytopmet}.  As before, we then define for each metric-étale $\@L$-structure over $X$ its \defn{isomorphism groupoid} $\Iso_X(\@M)$, imitating the usual definition of $\Aut(\@M)$ for a metric structure $\@M$.

We now have the following topological representation theorems, forming the first halves of the proofs of \cref{thm:intro-nonarchgpd-rep,thm:intro-locpolgpd-rep} respectively:

\pagebreak

\begin{theorem}[\cref{thm:nonarchgpd-yoneda-iso}]
\label{thm:intro-nonarchgpd-yoneda-iso}
For any open non-Archimedean (quasi-)Polish groupoid $G$, there is a countable language $\@L$ and a (second-countable) étale $\@L$-structure $\@M$ over the space of objects $G^0$ of $G$ such that we have an identity-on-objects isomorphism of topological groupoids $G \cong \Iso_{G^0}(\@M)$.
\end{theorem}

\begin{theorem}[\cref{thm:locpolgpd-yoneda-iso}]
\label{thm:intro-locpolgpd-yoneda-iso}
For any open ($\sigma$-locally) Polish groupoid $G$, there is a countable language $\@L$ and a (second-countable) metric-étale $\@L$-structure $\@M$ over the space of objects $G^0$ of $G$ such that we have an identity-on-objects isomorphism of topological groupoids $G \cong \Iso_{G^0}(\@M)$.
\end{theorem}

In the rest of this subsection, we briefly discuss the proofs of these results.  In some sense, both are ultimately generalizations of \defn{Cayley's theorem} from elementary group theory, which in its strong form asserts that the left multiplication action of a group $G$ on itself gives an isomorphism $G \cong \Aut(G \text{ as a right $G$-set})$.  The generalization of Cayley's theorem to groupoids is the fundamental \defn{Yoneda lemma}.  For a groupoid $G$, a \defn{right $G$-set} can be defined as a multi-sorted structure $\@M = (M^y)_{y \in G^0}$ (where $G^0 := \{$objects of $G\}$) with action maps $(-) \cdot g : M^y -> M^z$ for each $g : z -> y \in G$ satisfying the usual associativity and unitality laws.  To each $x \in G^0$, we may associate the right $G$-set $\@M_x = (M_x^y)_{y \in G^0}$ given by the hom-sets
\begin{align*}
M_x^y := G(y, x) := \{g : y -> x \in G\}
\end{align*}
with the right multiplication $G$-action.  Then $\@M := (\@M_x)_{x \in G^0}$ is a family of right $G$-sets over $G^0$; and the Yoneda lemma (phrased in terms of right $G$-sets) says that left multiplication gives a bijection between $G(x, y)$ and the set of isomorphisms $\@M_x \cong \@M_y$ for each $x, y \in G^0$, whence we have represented $G$ as the isomorphism groupoid of $\@M$.

In a different direction, Cayley's theorem is generalized to Polish groups by the representations as automorphism groups mentioned in the first few paragraphs of this Introduction:

\begin{theorem}[classical]
\label{thm:intro-nonarchgrp-rep}
For every non-Archimedean Polish group $G$, there is a countable language $\@L$ and countable $\@L$-structure $\@M$ such that $G \cong \Aut(\@M)$.
\end{theorem}

\begin{theorem}[essentially {\cite[3.1]{GK}}; see also {\cite[5]{Me2}}]
\label{thm:intro-polgrp-rep}
For every Polish group $G$, there is a countable language $\@L$ and a separable metric $\@L$-structure $\@M$ such that $G \cong \Aut(\@M)$.
\end{theorem}

For comparison, we now sketch a proof of \cref{thm:intro-nonarchgrp-rep}, a variation of the usual proof in e.g., \cite[2.4.4]{Gao}, which is more analogous to the usual proof of Cayley--Yoneda for discrete group(oid)s:

\begin{proof}[Proof sketch of \cref{thm:intro-nonarchgrp-rep}]
Let $\@U$ be a countable basis of open subgroups of $G$.  Consider the multi-sorted \defn{canonical structure} $\@M = (G/U)_{U \in \@U}$, where $G/U$ is the set of left cosets of $U$, equipped with the right multiplication maps
\begin{align*}
f_{U,V,S} := (-) \cdot S : G/U -> G/V
\end{align*}
for each $U, V \in \@U$ and $S \in G/V$ satisfying $U \subseteq S \cdot S^{-1}$ (which precisely guarantees that $f_{U,V,S}$ is well-defined).  By a topological version of the usual associativity/naturality argument in the proof of Cayley--Yoneda (see e.g., \cite[III~\S2, p.~61]{Mac}), one computes that every homomorphism $h : \@M -> \@M$ is uniquely determined by the coherent system of left cosets $(h(U))_{U \in \@U} \in \projlim_{U \in \@U} G/U$; moreover, the pointwise convergence topology on the endomorphism monoid $\End(\@M)$ coincides with the inverse limit topology on $\projlim_{U \in \@U} G/U$.  In other words, $\End(\@M)$ is the \defn{left completion} of $G$, whence the subgroup of invertible elements $\Aut(\@M) \subseteq \End(\@M)$ is isomorphic to $G$.
\end{proof}

The heart of the proof of \cref{thm:intro-nonarchgpd-yoneda-iso} (see \cref{sec:nonarchgpd-hom}) is a groupoid version of the ``topological naturality argument'' in the proof sketch above, first given by Moerdijk \cite[\S3.5]{Mo2} in the more abstract context of open localic groupoids, which when applied to an étale version of the canonical structure $\@M$ for a non-Archimedean groupoid $G$ identifies the \defn{homomorphism category} $\Hom_{G^0}(\@M)$ of $\@M$ (whose subgroupoid of invertible elements is $\Iso_{G^0}(\@M)$) as the ``left completion'' of $G$.  The proof is then completed by a Pettis-type Baire category argument (\cref{sec:pettis}) which shows that density of $G$ in $\Hom_{G^0}(\@M)$ implies $G \cong \Iso_{G^0}(\@M)$, analogously to the last sentence in the above proof sketch.  Thus, the proof of \cref{thm:intro-nonarchgpd-yoneda-iso} can be seen as the amalgamation of the proofs of \cref{thm:intro-nonarchgrp-rep} and the Yoneda lemma, as depicted in the following table:

\begin{center}
\renewcommand\arraystretch{1.5}
\begin{tabular}{c|ccc}
& discrete & non-Archimedean & Polish \\
\hline
group & Cayley's theorem & \cref{thm:intro-nonarchgrp-rep} & \cref{thm:intro-polgrp-rep} \\
groupoid & Yoneda lemma & \cref{thm:intro-nonarchgpd-yoneda-iso} & \cref{thm:intro-locpolgpd-yoneda-iso}
\end{tabular}
\end{center}

The proofs of \cref{thm:intro-polgrp-rep,thm:intro-locpolgpd-yoneda-iso} are conceptually similar to those of \cref{thm:intro-nonarchgrp-rep,thm:intro-nonarchgpd-yoneda-iso}, with quotients $G/U$ by open subgroup(oid)s $U$ replaced by left completions with respect to left-invariant metrics (or \defn{grey subgroup(oid)s}, in the terminology of \cite{BYM}) and the right multiplication maps $f_{U,V,S}$ replaced by approximate versions; see \cref{sec:locpolgpd-hom} for details.  For \cref{thm:intro-locpolgpd-yoneda-iso}, a key role is played by a groupoid version (\cref{thm:birkhoff-kakutani}, essentially noted by Buneci \cite{Bun}) of the Birkhoff--Kakutani metrization theorem guaranteeing the existence of enough left-invariant metrics; this depends on a nontrivial lemma (\cref{thm:ramsay}) of Ramsay \cite{Ram} on halving neighborhoods of unit morphisms in paracompact groupoids.

\subsection{Borel uniformization}
\label{sec:intro-unif}

Given \cref{thm:intro-nonarchgpd-yoneda-iso,thm:intro-locpolgpd-yoneda-iso}, we obtain \cref{thm:intro-nonarchgpd-rep,thm:intro-locpolgpd-rep} respectively via two steps, both of which take place in the Borel rather than topological context.  First, we replace the (metric\Hyphdash*)étale structure $\@M$ with a ``Borel family'' of structures with underlying universe $\#N$ (resp., $\#U$); see \cref{sec:cbstr-unif,sec:sbmstr-unif}.  This is done via a uniformly Borel version of adding constants, resp., Katětov's construction of $\#U$, essentially as described in the paragraphs preceding \cref{thm:intro-locpolgpd-rep}.  This technique is well-known in the literature, although there does not appear to be either a detailed exposition or a prior application with the level of uniformity we need.  We therefore give a careful statement and proof:

\begin{theorem}[\cref{thm:cbstr-isogpd-ff}]
\label{thm:intro-cbstr-isogpd-ff}
For any second-countable étale structure $\@M$ over a quasi-Polish space $X$, there is a countable (single-sorted) language $\@L$ (not necessarily that of $\@M$) and a full and faithful Borel functor $\Iso_X(\@M) -> \#S_\infty \ltimes \Mod_\#N(\@L)$.
\end{theorem}

\begin{theorem}[\cref{thm:sbmstr-isogpd-ff}]
\label{thm:intro-sbmstr-isogpd-ff}
For any second-countable metric-étale structure $\@M$ over a quasi-Polish space $X$, there is a countable (single-sorted) language $\@L$ (not necessarily that of $\@M$) and a full and faithful Borel functor $\Iso_X(\@M) -> \Iso(\#U) \ltimes \Mod_\#U(\@L)$.
\end{theorem}

Recall that a functor $F : G -> H$ between categories is \defn{full and faithful} if it is bijective on each hom-set $G(x, y)$.  \Cref{thm:intro-nonarchgpd-yoneda-iso,thm:intro-locpolgpd-yoneda-iso,thm:intro-cbstr-isogpd-ff,thm:intro-sbmstr-isogpd-ff} yield a full and faithful Borel functor $F : G -> \#S_\infty \ltimes \Mod_\#N(\@L)$ (resp., $F : G -> \Iso(\#U) \ltimes \Mod_\#U(\@L)$).  By abstract category theory, such a functor is an equivalence between $G$ and its \defn{essential image} (saturation of its image); however, the proof of this fact requires the axiom of choice to construct the inverse.  We show that for a full (i.e., surjective on hom-sets) Borel functor whose domain is an open quasi-Polish groupoid, the inverse may be constructed in a Borel way (using a large section uniformization argument):

\begin{theorem}[\cref{thm:functor-borel-full}]
\label{thm:intro-functor-borel-full}
For any full Borel functor $F : G -> H$ from an open quasi-Polish groupoid $G$ to a standard Borel groupoid $H$, the essential image of $F$ is a Borel subset of $H^0$, and there is a Borel map assigning to each $y$ in the essential image some $x_y \in G^0$ and $h_y : F(x_y) -> y$.  In particular, if $F$ is full and faithful, then $F$ is a Borel equivalence to its essential image.
\end{theorem}

Finally, the Lopez-Escobar theorem, due in the metric case to Coskey--Lupini \cite{CL}, yields an $\@L_{\omega_1\omega}$-sentence $\phi$ axiomatizing the essential image of $F$, whence we have a Borel equivalence $F : G -> \#S_\infty \ltimes \Mod_\#N(\@L, \phi)$ (resp., $F : G -> \Iso(\#U) \ltimes \Mod_\#U(\@L, \phi)$).

\subsection{Connection with strong conceptual completeness}
\label{sec:intro-scc}

One natural motivation for \cref{thm:intro-nonarchgpd-rep} comes from the strong conceptual completeness theorem we proved in \cite{Cscc}, which shows that every countable discrete $\@L_{\omega_1\omega}$-theory $\@T$ may be ``completely recovered'' from its standard Borel groupoid $\!{Mod}_{\le\#N}(\@L, \@T)$ of countable models on initial segments of $\#N$, in the form of its \defn{syntactic Boolean $\omega_1$-pretopos} $\-{\ang{\@L \mid \@T}}^B_{\omega_1}$, a categorical ``Lindenbaum--Tarski algebra'' capturing all logically relevant aspects of the syntax of $\@T$; see \cite{Cscc} for the precise definitions.  In other words, the syntax-to-semantics map $\!{Mod}_{\le\#N} : (\@L, \@T) |-> \!{Mod}_{\le\#N}(\@L, \@T)$ is ``injective'' (in the suitable 2-categorical sense, i.e., a full and faithful 2-functor).  In this context, \cref{thm:intro-nonarchgpd-rep} can be seen as characterizing the (essential) image of $\!{Mod}_{\le\#N}$, yielding in combination with \cite{Cscc}

\begin{theorem}[\cref{thm:2interp-equiv}]
\label{thm:intro-2interp-equiv}
$\!{Mod}_{\le\#N}$ is a contravariant equivalence between the 2-category of countable (discrete) $\@L_{\omega_1\omega}$-theories and the 2-category of open non-Archimedean (quasi-)Polishable standard Borel groupoids.
\end{theorem}

This gives a precise formulation of the idea that the model theory of discrete $\@L_{\omega_1\omega}$ ``completely corresponds'' to the Borel theory of open non-Archimedean Polish groupoids.

Since \cref{thm:intro-locpolgpd-rep} provides the analog to \cref{thm:intro-nonarchgpd-rep} in the metric setting, an obvious question is whether there is a strong conceptual completeness theorem for metric $\@L_{\omega_1\omega}$, analogous to \cite{Cscc}, which together with \cref{thm:intro-locpolgpd-rep} would yield a complete correspondence between metric $\@L_{\omega_1\omega}$ and open Polish groupoids analogous to \cref{thm:intro-2interp-equiv}.  The main obstacle to such a result is not any particular added technical difficulty in the metric setting, but rather (as mentioned above) the lack of a well-developed theory of categorical logic for metric $\@L_{\omega_1\omega}$.  In particular, one would hope for a good notion of \defn{syntactic category} for a metric $\@L_{\omega_1\omega}$-theory (especially for the $\omega_1$-coherent, or $\Sigma_1$, fragment, as in \cite{Cscc}), along with analogs of key results in the discrete setting such as the Joyal--Tierney representation theorem.  This is a subject of our ongoing work.

\subsection{Contents of paper}

This paper is split into two roughly parallel parts, dealing respectively with the discrete and metric contexts.  The second part refers to the first in some places where definitions and proofs are similar.

In \cref{sec:sp+gpd}, we review basic definitions and facts involving topological spaces, groupoids, and actions which are needed in the rest of the paper, including \cref{thm:intro-functor-borel-full}.

In \cref{sec:disclog}, we review definitions and conventions regarding discrete $\@L_{\omega_1\omega}$, including the space of structures on $\#N$, the logic action, and the Lopez-Escobar theorem.

In \cref{sec:etale}, we develop the theory of étale spaces and structures and their homomorphism categories and isomorphism groupoids; the main technical results here are that everything remains in the quasi-Polish context under suitable assumptions.  We then define the Borel-context analogs of these notions, and prove \cref{thm:intro-cbstr-isogpd-ff} (more generally for a Borel family of structures $\@M$).

In \cref{sec:nonarchgpd}, we define the canonical étale structure $\@M$ of an open non-Archimedean quasi-Polish groupoid and prove the topological representation \cref{thm:intro-nonarchgpd-yoneda-iso}, thereby completing the proof of \cref{thm:intro-nonarchgpd-rep}.  In combination with the results of \cite{Cscc}, we then deduce \cref{thm:intro-2interp-equiv}.

In \cref{sec:greytopmet}, we review the theory of grey sets and topometric spaces, including a proposed definition of the latter in the non-Hausdorff case.  The goal of introducing this terminology is to help make evident the analogy between the discrete and metric contexts.

In \cref{sec:ctslog}, we review definitions and conventions regarding continuous $\@L_{\omega_1\omega}$.

In \cref{sec:metale}, we develop the theory of metric-étale spaces and structures and their Borel-context analogs, similar to the étale setting.  We then review Katětov's construction of $\#U$, and use it to prove \cref{thm:intro-sbmstr-isogpd-ff} (in the uniformly Borel context).

Finally, in \cref{sec:locpolgpd}, we define the canonical metric-étale structure of an open $\sigma$-locally Polish groupoid and prove \cref{thm:intro-locpolgpd-yoneda-iso}, thereby proving \cref{thm:intro-locpolgpd-rep}.

\paragraph*{Acknowledgments}

I would like to thank Martino Lupini for several helpful discussions and references, and Alexander Kechris and Anush Tserunyan for providing some comments on an earlier draft of this paper.

%
%
%
%
%
%

\section{Spaces and groupoids}
\label{sec:sp+gpd}

This section contains some basic definitions and facts related to topological spaces and groupoids which are needed in the rest of the paper.

\subsection{Quasi-Polish and $\sigma$-locally Polish spaces}
\label{sec:qpol}

We find it convenient to work in the category of \defn{quasi-Polish spaces} \cite{deB},%
\footnote{Quasi-Polish spaces are equivalent to locales whose frames of opens are countably presented \cite{Hec}; this largely explains why concepts and results in topos theory have descriptive set-theoretic analogs.}
which are homeomorphic copies of $\*\Pi^0_2$ subspaces of $\#S^\#N$, where $\#S$ is the \defn{Sierpinski space} $\#S = \{0 < 1\}$ with $\{1\}$ open but not closed, and where a $\*\Pi^0_2$ set is one of the form
\begin{align*}
\bigcap_{i \in \#N} (F_i \cup U_i) \qquad\text{for $F_i$ closed, $U_i$ open}.
\end{align*}
This coincides with the usual definition of $\*\Pi^0_2$ if every closed set is $G_\delta$ (e.g., in a metrizable space), but is more general otherwise.  We will use the following basic properties of quasi-Polish spaces without mention (see \cite{deB} or \cite{Cqpol}):
\begin{itemize}

\item  Polish spaces are quasi-Polish.

\item  Quasi-Polish spaces are standard Borel, and can be made Polish by adjoining countably many closed sets to their topology.

\item  A topological space is Polish iff it is quasi-Polish and regular.

\item  Quasi-Polish spaces are Baire.

\item  Countable products and disjoint unions of quasi-Polish spaces are quasi-Polish.

\item  $\*\Pi^0_2$ subspaces of quasi-Polish spaces are quasi-Polish.

\item  Continuous open $T_0$ images of quasi-Polish spaces are quasi-Polish.

\end{itemize}

A topological space is \defn{$\sigma$-locally (quasi-)Polish} if it has a countable cover by open (quasi-)Polish subspaces.  Every $\sigma$-locally quasi-Polish space is quasi-Polish (see \cite[3.6]{Cqpol}).

\subsection{Fiber spaces}
\label{sec:fiber}

For a set $X$, a \defn{set over $X$} (also variously known as a \defn{bundle}, \defn{fiber space}, etc.)\ is simply a set $A$ equipped with a map $p : A -> X$ (called the \defn{projection map}), thought of as the family of sets $(p^{-1}(x))_{x \in X}$ indexed by $X$.  The \defn{($p$-)fiber of $A$ over $x \in X$} is then
\begin{align*}
A_x := p^{-1}(x).
\end{align*}
We often refer to $A$ itself as a set over $X$, with the projection map $p$ clear from context (and explicitly specified if not).  For some property $\Phi$, we say that $A$ is \defn{($p$-)fiberwise $\Phi$} if each $A_x$ is $\Phi$.

Given sets $p : A -> X$ and $q : B -> X$ over $X$, a \defn{map over $X$} between them is a map $f : A -> B$ commuting with the projections:
\begin{equation*}
\begin{tikzcd}
A \drar["p"'] \ar[rr, "f"] && B \dlar["q"] \\
& X
\end{tikzcd}
\end{equation*}
Equivalently, for each $x \in X$, $f$ restricts to a map between fibers
\begin{align*}
f_x := f|A_x : A_x -> B_x.
\end{align*}

Given $p : A -> X$ and $q : B -> X$, their \defn{fiber product} is
\begin{align*}
A \times_X B := \{(a, b) \in A \times B \mid p(a) = q(b)\},
\end{align*}
equipped with the projection map
\begin{align*}
p \times_X q := p \circ (A \times_X q) = q \circ (p \times_X B) : A \times_X B -> X,
\end{align*}
where $A \times_X q : A \times_X B -> A$ and $p \times_X B : A \times_X B -> B$ are the product projections, also called \defn{pullbacks} of $q, p$ (resp.)\ across $p, q$ (resp.).
\begin{equation*}
\begin{tikzcd}
A \times_X B \dar["A \times_X q"'] \rar["p \times_X B"] \drar["p \times_X q"] &[1em] B \dar["q"] \\
A \rar["p"'] & X
\end{tikzcd}
\end{equation*}
Again, the maps $p, q$ in the notation $A \times_X B$ must be inferred from context.

The \defn{$n$-fold fiber product} of $p : A -> X$ is
\begin{align*}
A^n_X := A \times_X A \times_X \dotsb \times_X A = \{(a_0, \dotsc, a_{n-1}) \in A^n \mid p(a_0) = \dotsb = p(a_{n-1})\},
\end{align*}
with projection $p^n_X := p \times_X \dotsb \times_X p : A^n_X -> X$.  When $n = 0$, we take $p^0_X : A^0_X -> X$ to be the identity $1_X : X -> X$.

By a \defn{(partial) section} of $p : A -> X$, we mean a map $s : U -> A$ from some $U \subseteq X$ such that $p \circ s = 1_U$.  Such $s$ is determined by its image $S := \im(s) \subseteq A$, which is a subset such that $p|S : S -> X$ is injective; we often refer to $S$ also as a \defn{section}.

For a topological space $X$, a \defn{topological space over $X$} is a topological space $A$ equipped with a continuous map $p : A -> X$.  We say that $p : A -> X$ is \defn{quasi-Polish}, \defn{$\sigma$-locally Polish}, etc., if $A$ is, and \defn{open} if $p$ is.  For topological spaces $p : A -> X$ and $q : B -> X$ over $X$, we give $A \times_X B$ the subspace topology from $A \times B$, with basic open sets $U \times_X V$ for open $U \subseteq A$ and $V \subseteq B$.  Then $A \times_X q$ and $p \times_X B$ are continuous.
If $p$ is open, then so is $p \times_X B$, with
\begin{align*}
(p \times_X B)(S \times_X T) = q^{-1}(p(S)) \cap T
\end{align*}
for open $S \subseteq A$ and $T \subseteq B$.

Similarly, for a standard Borel space $X$, a \defn{standard Borel space over $X$} is a standard Borel space $A$ equipped with a Borel map $p : A -> X$.  The fiber product of standard Borel spaces over $X$ is standard Borel, and the projection maps are Borel.

\subsection{Lower powerspaces}
\label{sec:lowpow}

Let $X$ be a topological space.  The \defn{lower powerspace} $\@F(X)$ is the space of closed subsets of $X$, equipped with the topology generated by the subbasic open sets
\begin{align*}
\Dia U := \{F \in \@F(X) \mid F \cap U \ne \emptyset\} \qquad\text{for open $U \subseteq X$}.
\end{align*}

\begin{proposition}[de~Brecht--Kawai \cite{dBK}]
If $X$ is quasi-Polish, then so is $\@F(X)$.
\end{proposition}

We will also need the following generalization of $\@F(X)$.  Let $p : A -> X$ be a continuous map between topological spaces.  The \defn{fiberwise lower powerspace} of $A$ over $X$ is
\begin{align*}
\@F_X(A) := \{(x, F) \mid x \in X \AND F \in \@F(A_x)\},
\end{align*}
equipped with the first projection map $\@F_X(p) : \@F_X(A) -> X$, and with subbasic open sets
\begin{align*}
\begin{aligned}
\@F_X(p)^{-1}(U) &= \{(x, F) \in \@F_X(A) \mid x \in U\} &&\text{for (subbasic) open $U \subseteq X$}, \\
\Dia_X S &:= \{(x, F) \in \@F_X(A) \mid F \in \Dia S_x\} &&\text{for (basic) open $S \subseteq A$}
\end{aligned}
\end{align*}
(where $S_x := S \cap A_x = S \cap p^{-1}(x) = (p|S)^{-1}(x)$).

\begin{proposition}
If $X, A$ are quasi-Polish, then so is $\@F_X(A)$.
\end{proposition}
\begin{proof}
We claim that we have a homeomorphism
\begin{align*}
e : \@F_X(A) -->{}& \{(x, F') \in X \times \@F(A) \mid F' = \-{F' \cap A_x}\} \tag{$*$} \\
(x, F) |-->{}& (x, \-F).
\end{align*}
$e$ is a bijection, with inverse $(x, F') |-> (x, F' \cap A_x)$.
Subbasic open sets in $X \times \@F(A)$ are $U \times \@F(A)$ for open $U \subseteq X$ and $X \times \Dia S$ for open $S \subseteq A$; we have $e^{-1}(U \times \@F(A)) = \@F_X(p)^{-1}(U)$ and $e^{-1}(X \times \Dia S) = \Dia_X S$, so $e$ is a homeomorphism.

So it is enough to show that the right-hand side of ($*$) is $\*\Pi^0_2$ in $X \times \@F(A)$.
Let $\@S, \@U$ be countable bases of open sets in $A, X$ respectively.  For $(x, F') \in X \times \@F(A)$, we have
\begin{align*}
A_x = p^{-1}(x) = p^{-1}(\-{\{x\}} \cap \bigcap_{x \in U \in \@U} U)
= p^{-1}(\-{\{x\}}) \cap \bigcap_{x \in U \in \@U} p^{-1}(U);
\end{align*}
thus $F' = \-{F' \cap A_x}$ iff
\begin{enumerate}
\item[(i)]  $F' \subseteq p^{-1}(\-{\{x\}})$, and
\item[(ii)]  for each $x \in U \in \@U$, $p^{-1}(U)$ is dense in $F'$ (whence by Baire category, so is $\bigcap_{x \in U \in \@U} p^{-1}(U)$).
\end{enumerate}
Condition (i) can be expressed as $\-{p(F')} \subseteq \-{\{x\}}$, which is $\*\Pi^0_2$ since $\-{p(-)} : \@F(A) -> \@F(X)$ and $\-{\{-\}} : X -> \@F(X)$ are continuous and the specialization order $\subseteq$ on $\@F(X)$ is $\*\Pi^0_2$ (see e.g., \cite[\S3]{dBK}, \cite[\S9]{Cqpol}).  Condition (ii) can be expressed as
\begin{align*}
\forall U \in \@U\, \forall S \in \@S\, (x \in U \AND F' \in \Dia S \implies F' \in \Dia(S \cap p^{-1}(U))).
&\qedhere
\end{align*}
\end{proof}

\subsection{Categories and groupoids}
\label{sec:gpd}

A \defn{(small) category} $G$ consists of sets $G^0$ (\emph{objects}) and $G^1$ (\emph{morphisms}) together with maps
\begin{itemize}
\item  $\sigma, \tau : G^1 -> G^0$ (\emph{source} and \emph{target}); usually $g \in G^1$ with $\sigma(g) = x$ and $\tau(g) = y$ is denoted $g : x -> y$, and the set of all such $g$ denoted
\begin{align*}
G(x, y) := \{g : x -> y \in G\} = \sigma^{-1}(x) \cap \tau^{-1}(y) \subseteq G^1;
\end{align*}
\item  $\iota : G^0 -> G^1$ (\emph{unit});
\item  $\mu : G^1 \times_{G^0} G^1 := \{(g, h) \in G^1 \times G^1 \mid \sigma(g) = \tau(h)\} -> G^1$ (\emph{multiplication}), usually denoted $g \cdot h := \mu(g, h)$;
\end{itemize}
satisfying the usual axioms (e.g., $\sigma(g \cdot h) = \sigma(h)$, $g \cdot \iota(\sigma(g)) = g$, etc.).

We adopt the convention that objects $x \in G^0$ in a category are identified with their unit morphisms $\iota(x) \in G^1$, so that we may simply regard $G = G^1$ as the underlying set of the category, and $\iota : G^0 -> G$ as the inclusion of objects (i.e., unit morphisms).

Unless otherwise specified, we always regard a category $G$ (= $G^1$) as a set over $G^0$ via the target map $\tau$.  Thus, for example,
\begin{align*}
G_x = \tau^{-1}(x) = \{g : y -> x \in G\}
\end{align*}
refers to the set of morphisms to $x$.

A morphism $g : x -> y$ in a category $G$ is an \defn{isomorphism} if it has an \defn{inverse} $g^{-1} : y -> x \in G$ satisfying $g \cdot g^{-1} = y$ and $g^{-1} \cdot g = x$.

A \defn{(small) groupoid} is a (small) category $G$ equipped with an additional map $\nu : G -> G$ (\emph{inverse}, usually denoted $g^{-1} := \nu(g)$) taking each morphism to its inverse.  Thus, a category can be made into a groupoid iff every morphism in it is an isomorphism.  The \defn{core} of a category $G$ is the groupoid $\core(G)$ of isomorphisms in $G$.

We identify a group $G$ with the corresponding one-object groupoid (where, under our convention, the object is the identity $1 \in G$).

The \defn{orbit} of an object $x \in G^0$ in a groupoid $G$ is its connected component
\begin{align*}
[x]_G := \{y \in G^0 \mid G(x, y) \ne \emptyset\}.
\end{align*}
The \defn{orbit equivalence relation} $\#E_G \subseteq G^0 \times G^0$ of a groupoid $G$ is given by
\begin{align*}
x \mathrel{\#E_G} y \coloniff y \in [x]_G \iff G(x, y) \ne \emptyset.
\end{align*}
For a subset $S \subseteq G^0$, we let $[S]_G := [S]_{\#E_G} = \bigcup_{x \in S} [x]_G$ denote its $\#E_G$-saturation.

A subset $A \subseteq G$ of a groupoid $G$ is \defn{symmetric} if $A = A^{-1}$, \defn{unital} if $\sigma(A), \tau(A) \subseteq A$, and a \defn{subgroupoid} if it is symmetric and closed under multiplication (hence unital, since $\sigma(g) = g^{-1} \cdot g$ and $\tau(g) = g \cdot g^{-1}$).  Every subset $A \subseteq G$ contains a largest symmetric unital subset, namely
\begin{align*}
\{g \mid g, g^{-1}, \sigma(g), \tau(g) \in A\},
\end{align*}
and is contained in a smallest subgroupoid, namely
\begin{align*}
\ang{A} := \bigcup_{n \ge 1} (A \cup A^{-1})^n
\end{align*}
(where $B^n := B \cdot B \cdot \dotsb \cdot B$), the \defn{subgroupoid generated by $A$}.

A \defn{topological category} (resp., \defn{topological groupoid}) is a small category (groupoid) $G$ such that $G^0, G^1$ are topological spaces and the maps $\sigma, \tau, \iota, \mu$ ($, \nu$) are continuous.  Note that then $\sigma, \iota$ exhibit $G^0$ as a retract of $G^1$, whence the topology on $G^0$ must be the subspace topology, and so we may continue to regard $G = G^1$ as the underlying space of the groupoid.

\begin{remark}
\label{rmk:gpd-top-noninv}
It is possible for a topological category to be a groupoid without being a \emph{topological} groupoid, i.e., for the inverse map to be discontinuous.  For example, consider the group $(\#Z, +)$ where $\#Z$ has the topology consisting of all upward-closed sets, regarded as a one-object groupoid.  In particular, the core of a topological category is not necessarily a topological groupoid in the subspace topology.
\end{remark}

For a class of spaces $\Phi$, we say that a topological category or groupoid $G$ is $\Phi$ if $G^0, G^1 \in \Phi$.  When $\Phi$ is closed under retracts, it is enough to require $G = G^1 \in \Phi$.  Thus, we have the notions of \defn{Polish category}, \defn{quasi-Polish groupoid}, etc.

A topological groupoid $G$ is \defn{open} if $\sigma : G -> G^0$ (equivalently $\tau : G -> G^0$ or $\mu : G \times_{G^0} G -> G$) is an open map, or equivalently for any open $U, V \subseteq G$, the product set $U \cdot V = \mu(U \times_{G^0} V)$ is open.  Note that in an open topological groupoid, every open set $U \subseteq G$ generates an open subgroupoid $\ang{U}$ (since $\mu, \nu$ are open).

A topological groupoid $G$ is \defn{non-Archimedean} if every unit morphism $x \in G^0$ has a neighborhood basis (in $G$) consisting of open subgroupoids $U \subseteq G$.  A \defn{basis of open subgroupoids} in $G$ is a family $\@U$ of open subgroupoids $U \subseteq G$ from which every $x \in G^0$ has a neighborhood basis.

%

\begin{lemma}
\label{thm:nonarchgpd-basis-ctble}
If $G$ is second-countable and non-Archimedean, then $G$ has a countable basis of open subgroupoids.
\end{lemma}
\begin{proof}
Let $\@U$ be a basis of open subgroupoids in $G$ and $\@V$ be a countable basis of open sets in $G$.  Take $\{\ang{V} \mid V \in \@V \AND \exists U \in \@U\, (V \subseteq U)\}$.
\end{proof}

A \defn{standard Borel category} (resp., \defn{standard Borel groupoid}) is a small category (groupoid) $G$ such that $G^0, G^1$ are standard Borel spaces and the maps $\sigma, \tau, \iota, \mu$ are Borel.  Every quasi-Polish category (groupoid) has an underlying standard Borel category (groupoid).  Unlike in the topological context (\cref{rmk:gpd-top-noninv}), if a standard Borel category $G$ is a groupoid, then the inverse map $\nu : G -> G$ is automatically Borel (since its graph is Borel).

\subsection{Pettis's theorem}
\label{sec:pettis}

The following are based on the well-known corresponding results for Polish groups (see e.g., \cite[2.2.1, 2.3.2]{Gao}) as well as the localic ``closed subgroupoid theorem'' of Johnstone \cite{Jclsubgpd}.

\begin{proposition}[Pettis's theorem]
\label{thm:pettis}
Let $G$ be a $\sigma$-fiberwise Baire topological groupoid, and let $A, B, U, V \subseteq G$ with $U$ $\tau$-fiberwise open, $V$ $\sigma$-fiberwise open, $U \setminus A$ $\tau$-fiberwise meager, and $V \setminus B$ $\sigma$-fiberwise meager.  Then $U \cdot V \subseteq A \cdot B$.
\end{proposition}
\begin{proof}
Let $g : x -> y \in G$.  We have
\begin{align*}
g \in U \cdot V
&\iff (U^{-1} \cdot g) \cap V \ne \emptyset \\
&\implies (U^{-1} \cdot g) \cap V \text{ non-meager in $\sigma^{-1}(x)$}
&&\text{since $\sigma^{-1}(x)$ is Baire} \\
&\implies (A^{-1} \cdot g) \cap B \text{ non-meager in $\sigma^{-1}(x)$}
&&\text{since $(U^{-1} \cdot g) \setminus (A^{-1} \cdot g), V \setminus B$ are} \\[-.4ex] &&&\text{meager in $\sigma^{-1}(x)$} \\
&\implies (A^{-1} \cdot g) \cap B \ne \emptyset \\
&\iff g \in A \cdot B.
&&\qedhere
\end{align*}
\end{proof}

\begin{corollary}
\label{thm:subgpd-dense}
Let $G$ be a quasi-Polish groupoid and $H \subseteq G$ be a quasi-Polish subgroupoid which is $\tau$-fiberwise (equivalently, $\sigma$-fiberwise) dense.  Then $H = G$.
\end{corollary}
\begin{proof}
Since $H \subseteq G$ is quasi-Polish, it is $\*\Pi^0_2$, hence (being fiberwise dense) fiberwise comeager.  Take $A = B = H$ and $U = V = G$ in \cref{thm:pettis}.
\end{proof}

\subsection{Functors and equivalences}
\label{sec:functor}

A \defn{functor} $F : G -> H$ between two categories $G, H$ is a map preserving the category structure maps $\sigma, \tau, \iota, \mu$.  A functor $F : G -> H$ is
\begin{itemize}
\item  an \defn{embedding} if $F : G -> H$ is injective;
\item  \defn{faithful} if $F|G(x, y) : G(x, y) -> H(F(x), F(y))$ is injective for all $x, y \in G^0$;
\item  \defn{full} if $F|G(x, y) : G(x, y) -> H(F(x), F(y))$ is surjective for all $x, y \in G^0$;
\item  \defn{essentially surjective} if its \defn{essential image} $[F(G^0)]_{\core(H)}$ is all of $H^0$, i.e., for every $y \in H^0$, there is $x \in G^0$ such that $F(x) \cong y$.
\end{itemize}
A subcategory $G \subseteq H$ is \defn{full} if the inclusion $G -> H$ is full, i.e., $G$ consists of all morphisms in $H$ between some subset of objects $G^0 \subseteq H^0$; we denote this by $G = H|G^0$.

A functor $F : G -> H$ restricted to the objects gives a homomorphism $F : \#E_G -> \#E_H$ between the corresponding orbit equivalence relations, which is a \defn{reduction} (descends to an injection between the quotient spaces $G^0/G `-> H^0/H$) if $F$ is full.

Given two functors $F, F' : G -> H$, a \defn{natural transformation} $\alpha : F -> F'$ is a map $\alpha : G^0 -> H$ such that $\alpha(x) : F(x) -> F'(x)$ for each $x \in G^0$, and $\alpha(y) \cdot F(g) = F'(g) \cdot \alpha(x)$ for each $g : x -> y \in G$.  A natural transformation $\alpha : F -> F'$ is a \defn{natural isomorphism} if each $\alpha(x)$ is an isomorphism, in which case $\alpha^{-1} : F' -> F$ defined by $\alpha^{-1}(x) := \alpha(x)^{-1}$ is also a natural transformation.  If $H$ is a groupoid, then clearly every natural transformation is a natural isomorphism.  Two functors $F, F' : G -> H$ are \defn{(naturally) isomorphic}, written $F \cong F'$, if there is a natural isomorphism $\alpha : F -> F'$.

A functor $F : G -> H$ is an \defn{equivalence of categories} if it has an \defn{inverse equivalence} $F' : H -> G$ such that $F \circ F'$ and $F' \circ F$ are naturally isomorphic to the respective identity functors.  It is a standard fact in category theory (see e.g., \cite[IV~4.1]{Mac}) that

\begin{proposition}
\label{thm:functor-equiv}
A functor $F : G -> H$ is an equivalence iff it is full, faithful, and essentially surjective.
\end{proposition}
\begin{proof}[Proof of $\Longleftarrow$]
Choose (using essential surjectivity) for each $y \in H^0$ some $x_y \in G^0$ and isomorphism $h_y : F(x_y) -> y \in \core(H)$.  Define the inverse equivalence
\begin{align*}
F' : H &--> G \\
(h : y -> y') &|--> ((F|G(x_y, x_{y'}))^{-1}(h_{y'}^{-1} \cdot h \cdot h_y) : x_y -> x_{y'}).
\end{align*}
The natural isomorphism $F \circ F' \cong 1_H$ is given by $h_y : F(F'(y)) = F(x_y) -> y$ for each $y \in H^0$, while the natural isomorphism $F' \circ F \cong 1_G$ is given by $(F|G(x_{F(x)}, x))^{-1}(h_{F(x)}) : F'(F(x)) = x_{F(x)} -> x$ for each $x \in G^0$.
\end{proof}

The notions of \defn{Borel functor} (between standard Borel categories)
and \defn{Borel natural transformation}
are defined in the obvious way.  By a \defn{Borel equivalence of categories}, we mean a Borel functor with a Borel inverse equivalence witnessed by Borel natural isomorphisms.  The use of the axiom of choice in the proof of \cref{thm:functor-equiv} means that this is a stronger notion than that of a full, faithful, and essentially surjective Borel functor.  However, in the case of open quasi-Polish groupoids, we recover the equivalence:

\begin{proposition}
\label{thm:functor-borel-full}
Let $G$ be an open quasi-Polish groupoid, $H$ be a standard Borel groupoid, and $F : G -> H$ be a full Borel functor.  Then $F$ has Borel essential image, and the map
\begin{align*}
\tau \circ (H \times_{H^0} F) : H \times_{H^0} G^0 := \{(h, x) \in H \times G^0 \mid \sigma(h) = F(x)\} &--> H^0 \\
(h, x) &|--> \tau(h)
\end{align*}
(whose image is the essential image $[F(G^0)]_H$ of $F$)
has a Borel section
\begin{align*}
(h_{(-)}, x_{(-)}) : [F(G^0)]_H -> H \times_{H^0} G^0,
\end{align*}
i.e., there is a Borel assignment to each $y \in [F(G^0)]_H$ of a $x_y \in G^0$ and $h_y : F(x_y) -> y \in H$.
\end{proposition}
\begin{proof}
We will apply the large section uniformization theorem in the form \cite[18.6*]{Kcdst'} to the inverse graph of $\tau \circ (H \times_{H^0} F)$:
\begin{align*}
P &:= H^0 \times_{H^0} (H \times_{H^0} G^0) \\
&= \{(y, (h, x)) \in H^0 \times (H \times_{H^0} G^0) \mid y = (\tau \circ (H \times_{H^0} F))(h, x) = \tau(h)\}.
\end{align*}

Assign to each $y \in [F(G^0)]_H$ the $\sigma$-ideal
\begin{align*}
I_y := \left\{R \subseteq P_y \relmiddle|
\forall (h, x) \in P_y\;
(\{g : z -> x \in G \mid (h \cdot F(g), z) \in R\} \text{ is meager in } \tau^{-1}(x))
\right\}
\end{align*}
(where $P_y := \{(h, x) \mid (y, (h, x)) \in P\}$ is the fiber of $P$).  We claim that also
\begin{align*}
I_y = \left\{R \subseteq P_y \relmiddle|
\exists (h, x) \in P_y\;
(\{g : z -> x \in G \mid (h \cdot F(g), z) \in R\} \text{ is meager in } \tau^{-1}(x))
\right\}.
\end{align*}
To see this, let $x_1, x_2 \in G^0$ and $h_i : F(x_i) -> y$ ($i = 1, 2$).  Then $h_2^{-1} \cdot h_1 : F(x_1) -> F(x_2)$, so since $F$ is full, there is $k : x_1 -> x_2 \in G$ such that $F(k) = h_2^{-1} \cdot h_1$.  Then we have a homeomorphism $k \cdot (-) : \tau^{-1}(x_1) -> \tau^{-1}(x_2)$ which exchanges the two sets $\{g : z -> x_i \in G \mid (h_i \cdot F(g), z) \in R\}$ ($i = 1, 2$), so one is meager iff the other is.  This proves the claim.

Now let $R \subseteq P$ be Borel.  The set
\begin{align*}
J_R := \left\{(y, (h, x)) \in P \relmiddle| \{g : z -> x \in G \mid (h \cdot F(g), z) \in R_y\} \text{ is meager in } \tau^{-1}(x)\right\}
\end{align*}
is Borel, since its complement is the Baire category quantifier $\exists^*$ (see \cite[\S7]{Cqpol}) for the continuous open map $H^0 \times_{H^0} (H \times_{H^0} \tau) : H^0 \times_{H^0} (H \times_{H^0} G) -> H^0 \times_{H^0} (H \times_{H^0} G^0)$ applied to the Borel set
\begin{align*}
\{(y, (h, g)) \mid (h \cdot F(g), \sigma(g)) \in R_y\}
&\subseteq H^0 \times_{H^0} (H \times_{H^0} G) \\
&= \{(y, (h, g)) \in H^0 \times (H \times G) \mid y = \tau(h) \AND \sigma(h) = F(\tau(g))\}.
\end{align*}
It follows that for $y \in [F(G^0)]_H$, the condition
\begin{align*}
R_y \in I_y
\iff \forall (h, x) \in P_y\, ((y, (h, x)) \in J_R)
\iff \exists (h, x) \in P_y\, ((y, (h, x)) \in J_R)
\end{align*}
is $\*\Delta^1_1([F(G^0)]_H)$.  Also, each $P_y \not\in I_y$, since $J_P = \emptyset$.  Thus by \cite[18.6*]{Kcdst'}, $P$ has a Borel uniformization, which is the graph of the desired Borel section.
\end{proof}

\begin{corollary}
\label{thm:functor-borel-equiv}
Let $G$ be an open quasi-Polish groupoid, $H$ be a standard Borel groupoid, and $F : G -> H$ be a full and faithful Borel functor.  Then $F$ is a Borel equivalence $G -> H|[F(G^0)]_H$.
\end{corollary}
\begin{proof}
The proof of \cref{thm:functor-equiv} goes through, using $h_{(-)}, x_{(-)}$ provided by \cref{thm:functor-borel-full}.
\end{proof}

\subsection{Actions}
\label{sec:action}

Let $G$ be a category and $p : A -> G^0$ be a set over $G^0$.  An \defn{action} of $G$ on $A$ (more precisely, $(A, p)$) is a map
\begin{align*}
\alpha : G \times_{G^0} A = \{(g, a) \in G \times A \mid \sigma(g) = p(a)\} -> A,
\end{align*}
usually denoted $g \cdot a := \alpha(g, a)$, satisfying the usual axioms ($p(g \cdot a) = \tau(g)$, $p(a) \cdot a = a$, and $(g \cdot h) \cdot a = g \cdot (h \cdot a)$).  Thus, for $g : x -> y \in G$, the action of $g$ yields a map $A_x -> A_y$.  For $G$ a topological category, the action is \defn{continuous} if $A$ is a topological space and $p, \alpha$ are continuous.  Similarly, for $G$ a standard Borel category, the action is \defn{Borel} if $A$ is a standard Borel space and $p, \alpha$ are Borel.

Let $G$ be a category acting on $p : A -> G^0$.  The \defn{action category} (also known as \defn{category of elements}) $G \ltimes A$ has, informally, objects consisting of elements $a \in A$ and morphisms $g : a -> b$ consisting of morphisms $g : p(a) -> p(b) \in G$ such that $g \cdot a = b$.  Formally, we put
\begin{align*}
(G \ltimes A)^0 &:= A, \\
(G \ltimes A)^1 &:= G \times_{G^0} A = \{(g, a) \in G \times A \mid \sigma(g) = p(a)\},
\end{align*}
with $(g, a) \in (G \ltimes A)^1$ thought of as the morphism $g : a -> g \cdot a$; hence we define
\begin{align*}
\sigma(g, a) := a, &&
\tau(g, a) := g \cdot a, &&
\iota(a) := (\iota(p(a)), a), &&
(h, g \cdot a) \cdot (g, a) := (h \cdot g, a).
\end{align*}
Under our convention that objects are identified with unit morphisms, $a \in A = (G \ltimes A)^0$ is identified with $\iota(a) = (p(a), a) \in (G \ltimes A)^1$, and so $\sigma, \tau$ become
\begin{align*}
\sigma(g, a) = (p(a), a) = (\sigma(g), a), &&
\tau(g, a) = (\tau(g), g \cdot a).
\end{align*}

When $G$ is a groupoid acting on $p : A -> G^0$, then $G \ltimes A$ is also a groupoid, with
\begin{align*}
(g, a)^{-1} = (g^{-1}, g \cdot a);
\end{align*}
we call it the \defn{action groupoid}.  The \defn{orbit equivalence relation} of the action
\begin{align*}
\#E_G^A := \#E_{G \ltimes A} \subseteq A \times A
\end{align*}
is that of the action groupoid, i.e.,
\begin{align*}
a \mathrel{\#E_G^A} b \iff \exists g \in G(p(a), p(b))\, (g \cdot a = b).
\end{align*}

When $G$ is a topological category (groupoid) acting continuously on $p : A -> G^0$, then $G \ltimes A$ is also a topological category (groupoid).
If $G$ is an open topological groupoid, then so is $G \ltimes A$ (since $\sigma : G \ltimes A -> A$ is the pullback of $\sigma : G -> G^0$ across $p$).
If $G$ is non-Archimedean, then so is $G \ltimes A$ (since for a basic open subgroupoid $U \subseteq G$ and open $S \subseteq A$, $U \times_{G^0} S \subseteq G \ltimes A$ is a basic open subgroupoid).
If $G, A$ are (quasi-)Polish, then so is $G \ltimes A$.

Similarly, for a standard Borel category (groupoid) $G$ with a Borel action on $p : A -> G^0$, $G \ltimes A$ is also a standard Borel category (groupoid).

\section{Discrete infinitary logic}
\label{sec:disclog}

In this section, we review some concepts and conventions regarding discrete infinitary first-order logic.

We will be working with multi-sorted structures.  A \defn{(multi-sorted) first-order language} $\@L$ consists of a set $\@L_s$ of \defn{sorts}, a set $\@L_r$ of \defn{relation symbols}, each with an associated \defn{arity} which is a finite tuple of sorts $(P_0, \dotsc, P_{n-1}) \in \@L_s^n$ for some $n$, and a set $\@L_f$ of \defn{function symbols}, each with an associated \defn{arity} $(P_0, \dotsc, P_{n-1}) \in \@L_s^n$ as well as a \defn{value sort} $Q \in \@L_s$.  We let
\begin{gather*}
\@L_r(P_0, \dotsc, P_{n-1}) \subseteq \@L_r, \\
\@L_f(P_0, \dotsc, P_{n-1}; Q) \subseteq \@L_f
\end{gather*}
denote respectively the $(P_0, \dotsc, P_{n-1})$-ary relation symbols and the $(P_0, \dotsc, P_{n-1})$-ary $Q$-valued function symbols.  In the single-sorted case ($\@L_s = 1$), we instead write $\@L_r(n), \@L_f(n)$ respectively.

Given a language $\@L$, an \defn{$\@L$-structure} $\@M$ consists of:
\begin{itemize}
\item  for each sort $P \in \@L_s$, an underlying set $P^\@M$ (when $\@L$ is one-sorted, we denote $P^\@M$ for the unique sort $P$ by $M$);
\item  for each $(P_0, \dotsc, P_{n-1})$-ary relation symbol $R \in \@L_r$, a subset $R^\@M \subseteq P_0^\@M \times \dotsb \times P_{n-1}^\@M$;
\item  for each $(P_0, \dotsc, P_{n-1})$-ary $Q$-valued function symbol $f \in \@L_f$, a function $f^\@M : P_0^\@M \times \dotsb \times P_{n-1}^\@M -> Q^\@M$.
\end{itemize}

Given a (single-sorted relational, for simplicity) $\@L$-structure $\@M$ and a bijection $f : M \cong N$ for some other set $N$, we have the \defn{pushforward structure}
\begin{align*}
f_*(\@M) := (N, f(R^\@M))_{R \in \@L_r}.
\end{align*}
In other words, $f_*(\@M)$ is the unique $\@L$-structure on $N$ such that $f : M \cong N$ is an $\@L$-isomorphism $f : \@M \cong f_*(\@M)$.

For a countable single-sorted relational language $\@L$, we have the \defn{Polish space of $\@L$-structures on $\#N$}
\begin{align*}
\Mod_\#N(\@L) &\cong \prod_{n \in \#N; R \in \@L_r(n)} 2^{\#N^n} \\
\@M = (\#N, R^\@M)_{R \in \@L_r} &|-> (R^\@M)_{R \in \@L_r}
\end{align*}
where we are identifying $2^{\#N^n}$ with the power set of $\#N^n$.  Pushforward of structures gives the \defn{logic action} $\#S_\infty \curvearrowright \Mod_\#N(\@L)$ of the infinite symmetric group $\#S_\infty$ on $\Mod_\#N(\@L)$, a continuous action of a Polish group.  The action groupoid
\begin{align*}
\#S_\infty \ltimes \Mod_\#N(\@L)
\end{align*}
is the \defn{Polish groupoid of $\@L$-structures on $\#N$}.  Its objects are $\@L$-structures on $\#N$, while its morphisms are isomorphisms between structures.  Since $\#S_\infty$ is a group, hence open as a groupoid, $\#S_\infty \ltimes \Mod_\#N(\@L)$ is open; since $\#S_\infty$ is non-Archimedean, so is $\#S_\infty \ltimes \Mod_\#N(\@L)$.

The logic $\@L_{\omega_1\omega}$ is the extension of the usual finitary first-order logic with countably infinite conjunction $\bigwedge$ and disjunction $\bigvee$; see \cite[\S11.2]{Gao}.  We will only need to deal with $\@L_{\omega_1\omega}$ (as opposed to structures) for single-sorted languages $\@L$, and even then, only the results that follow in this section.
For an $\@L$-structure $\@M$ and $\@L_{\omega_1\omega}$-sentence (i.e., formula with no free variables) $\phi$, we write as usual $\@M \models \phi$ if $\@M$ satisfies $\phi$.

%

For an $\@L_{\omega_1\omega}$-sentence $\phi$, an easy induction (see \cite[16.7]{Kcdst}) shows that
\begin{align*}
\Mod_\#N(\@L, \phi) := \{\@M \in \Mod_\#N(\@L) \mid \@M \models \phi\} \subseteq \Mod_\#N(\@L)
\end{align*}
is Borel; it is clearly also $\#S_\infty$-invariant (under the logic action).  The converse is the well-known theorem of Lopez-Escobar \cite{LE} (see also \cite[16.8]{Kcdst}, \cite[11.3.6]{Gao}):

\begin{theorem}[Lopez-Escobar]
\label{thm:lopez-escobar}
For any $\#S_\infty$-invariant Borel $B \subseteq \Mod_\#N(\@L)$, there is an $\@L_{\omega_1\omega}$-sentence $\phi$ such that $B = \Mod_\#N(\@L, \phi)$.
\end{theorem}

For an $\@L_{\omega_1\omega}$-sentence $\phi$, the action groupoid
\begin{align*}
\#S_\infty \ltimes \Mod_\#N(\@L, \phi)
\end{align*}
is the \defn{standard Borel groupoid of models of $\phi$ on $\#N$}.

\begin{remark}
\label{rmk:isogpd-nonarchopen}
We may refine the topology on $\Mod_\#N(\@L)$ to make $\Mod_\#N(\@L, \phi)$ into a $\*\Pi^0_2$ subset, hence $\#S_\infty \ltimes \Mod_\#N(\@L, \phi)$ into an open non-Archimedean Polish groupoid, by using the Becker--Kechris theorem \cite[5.2.1]{BK}, or more simply by taking the topology generated by a countable fragment of $\@L_{\omega_1\omega}$ containing $\phi$ (see e.g., \cite[\S11.4]{Gao}).
\end{remark}


\section{Étale spaces and structures}
\label{sec:etale}

In this section, we define and study étale spaces, étale structures, and their topological categories of homomorphisms.\footnote{%
An \emph{étale space} over $X$ corresponds to a sheaf on $X$.
An \emph{étale $\@L$-structure} over $X$ is an $\@L$-structure in the topos $\Sh(X)$ of sheaves on $X$ (see \cite[D1.2]{Jeleph}).
The \emph{isomorphism groupoid} (resp., \emph{homomorphism category}) of an étale $\@L$-structure $\@M$ over $X$ is given by the pullback (resp., comma object) of the corresponding geometric morphism $X -> \!{Set}[\@L]$ to the classifying topos of $\@L$-structures (see \cite[D3]{Jeleph}) with itself.
}
The contents of this section will form half of the proof of \cref{thm:intro-nonarchgpd-rep}, which will be completed in \cref{sec:nonarchgpd}.

Informally, an \emph{étale space} $A$ over a topological space $X$ is a ``continuous'' assignment of a set (or discrete space) $A_x$ to each $x \in X$.  For a first-order language $\@L$, an \emph{étale $\@L$-structure} $\@M$ over $X$ is a ``continuous'' assignment of an $\@L$-structure $\@M_x$ to each $x \in X$, and consists of an underlying étale space $M$ over $X$ (more generally, one for each sort of $\@L$) together with ``continuous'' interpretations of the function and relation symbols in $\@L$.  For an étale $\@L$-structure $\@M$ over $X$, its \emph{homomorphism category} $\Hom_X(\@M)$ is a canonical topological category whose objects are points $x \in X$ and whose morphisms $x -> y$ are $\@L$-homomorphisms $\@M_x -> \@M_y$; its \emph{isomorphism groupoid} $\Iso_X(\@M)$ is defined similarly but with $\@L$-isomorphisms.  These are analogous to the endomorphism monoid $\End(\@M)$ and automorphism group $\Aut(\@M)$ of a single $\@L$-structure $\@M$.

The main results of this section are: first, that the homomorphism category and isomorphism groupoid of a quasi-Polish étale $\@L$-structure are quasi-Polish (\cref{thm:etalestr-isogpd-qpol}); and second, that the isomorphism groupoid admits a full and faithful Borel functor to a groupoid of structures on $\#N$ (\cref{thm:cbstr-isogpd-ff}).

\subsection{Étale spaces}
\label{sec:etalesp}

Let $X$ be a topological space.  An \defn{étale space over $X$} is a topological space $p : A -> X$ over $X$ such that $A$ admits a cover by \defn{open sections}, which are open sets $S \subseteq A$ to which $p$ restricts to an open embedding $p|S : S -> X$ (so that $s := (p|S)^{-1} : p(S) `-> A$ is a continuous partial section of $p$ with image $S$ and open domain $p(S) \subseteq X$).  Intuitively, the open sections uniformly witness that $A$ is fiberwise discrete, so that $A$ forms a ``continuous'' assignment of a set $A_x$ to each $x \in X$.

The following facts about étale spaces are standard:

\begin{proposition}
\label{thm:etale-props}
\leavevmode
\begin{enumerate}[label=(\alph*)]
\item  If $p : A -> X$ is étale, then $p$ is open.
\item  If $p : A -> X$ and $q : B -> A$ are étale, then so is $p \circ q : B -> X$.
\item  If $p : A -> X$ is étale, $q : B -> A$ is continuous, and $p \circ q : B -> A$ is étale, then $q$ is étale (hence in particular open).
\item  If $p : A -> X$ is étale and $f : Y -> X$ is continuous, then $Y \times_X p : Y \times_X A -> Y$ is étale.  Hence, fiber products of étale spaces over $X$ are étale over $X$.
\end{enumerate}
\end{proposition}
\begin{proof}
See e.g., \cite[\S2.3]{Ten}, \cite[2.1 (arXiv version)]{Cscc}.
\end{proof}

\begin{proposition}
\label{thm:etale-diagopen}
An open topological space $p : A -> X$ over $X$ is étale iff the diagonal $\Delta_A \subseteq A \times_X A$ is open.
\end{proposition}
\begin{proof}
For open $S \subseteq A$, $p|S$ is injective, i.e., $S$ is an open section, iff $S \times_X S \subseteq \Delta_A$.  If $A$ is étale over $X$, then $\Delta_A \subseteq A \times_X A$ is a union of $S \times_X S$ for open sections $S \subseteq A$, hence open.  Conversely, if $\Delta_A \subseteq A \times_X A$ is open, then for any $a \in A$, we have $(a, a) \in S \times_X S \subseteq \Delta_A$ for some basic open square $S \times_X S \subseteq \Delta_A$ with $S \subseteq A$ an open section containing $a$.
\end{proof}

\begin{lemma}
\label{thm:relopen}
Let $X, Y$ be two topological spaces and $R \subseteq X \times Y$ be a binary relation.  The second projection $q : R -> Y$ is open iff for every open $U \subseteq X$, the \defn{relational image}
\begin{align*}
R[U] := \{y \in Y \mid \exists x \in X\, ((x, y) \in R)\} \subseteq Y
\end{align*}
is open.
\end{lemma}
If the above conditions hold, we say $R$ is \defn{relationally open}.
\begin{proof}
For a basic open rectangle $U \times V \subseteq X \times Y$, with $U \subseteq X$ and $V \subseteq Y$ open, we have $q(R \cap (U \times V)) = q(R \cap (U \times Y)) \cap V = R[U] \cap V$, which gives $\Longleftarrow$; taking $V := Y$ gives $\Longrightarrow$.
\end{proof}

When $R$ in \cref{thm:relopen} is an equivalence relation on $X$, we have $R[U] = [U]_R = \pi^{-1}(\pi(U))$ where $\pi : X ->> X/R$ is the quotient map; thus in this case, $R$ is relationally open iff the quotient map is open.  In the étale context, this gives

\begin{proposition}
\label{thm:etale-quotient}
Let $p : A -> X$ be an open topological space over $X$, and let $E \subseteq A \times_X A$ be an open (in $A \times_X A$) equivalence relation.  Then $p$ descends to $p' : A/E -> X$ which is étale over $X$, and the quotient map $\pi : A ->> A/E$ is open.
\end{proposition}
\begin{proof}
Since $p$ is open, the second projection $q : A \times_X A -> A$ is open, whence $q|E : E -> A$ is open, i.e., $E \subseteq A \times A$ is relationally open.  Thus $\pi$ is open.  We have $p'(S) = p(\pi^{-1}(S))$, whence $p'$ is open.  The diagonal $\Delta_{A/E} \subseteq A/E \times_X A/E$ is the projection of $E \subseteq A \times_X A$, hence open; so $p'$ is étale.
%
%
\end{proof}

In the situation of \cref{thm:etale-quotient}, we call open $S \subseteq A$ \defn{$E$-small over $X$} if $S \times_X S \subseteq E$.  Open sections $\pi(S) \subseteq A/E$ (with $S$ $E$-invariant open) are in 1--1 correspondence with their $E$-invariant open $E$-small lifts $S \subseteq A$.

We will primarily be interested in étale spaces over quasi-Polish spaces, which are furthermore ``uniformly fiberwise countable'' in the following equivalent senses:

\begin{proposition}
\label{thm:etale-ctble}
Let $p : A -> X$ be an étale space over a second-countable space $X$.  The following are equivalent:
\begin{enumerate}
\item[(i)]  $A$ is second-countable;
\item[(ii)]  $A$ has a countable cover by open sections;
\item[(iii)]  $A$ has a countable basis of open sections.
\end{enumerate}
If in addition $X$ is quasi-Polish, these are equivalent to
\begin{enumerate}
\item[(iv)]  $A$ is quasi-Polish.
\end{enumerate}
\end{proposition}
\begin{proof}
(i)$\implies$(iii):  Let $\@S$ be a countable basis of open sets in $A$.  Then $\@S' := \{S \in \@S \mid S \text{ is an open section}\}$ is a countable basis of open sections.
Indeed, for open $T \subseteq A$, $T$ is a union of open sections, each of which is a union of $S \in \@S$ which (being contained in an open section) are open sections.

(ii)$\implies$(iii):  Let $\@S$ be a countable cover of $A$ by open sections, and let $\@U$ be a countable basis of open sets in $X$.  Then $\@S' := \{S \cap p^{-1}(U) \mid S \in \@S \AND U \in \@U\}$ is a countable basis of open sections.  Indeed, for open $T \subseteq A$, we have $T = \bigcup_{S \in \@S} (S \cap T) = \bigcup_{S \in \@S} (S \cap p^{-1}(p(S \cap T))) = \bigcup_{S \in \@S} \bigcup_{\@U \ni U \subseteq p(S \cap T)} (S \cap p^{-1}(U))$.

(iii)$\implies$(ii,i): obvious.

(iv)$\implies$(i): obvious.

(ii)$\implies$(iv), when $X$ is quasi-Polish:  clearly (ii) implies that $A$ is $\sigma$-locally quasi-Polish (see \cref{sec:qpol}), hence quasi-Polish.
\end{proof}


\subsection{Symmetric groupoids}
\label{sec:symgpd}

Let $p : A -> X$ be an étale space over a topological space $X$.  Define
\begin{align*}
\Hom_X(A) &:= \{(x, y, f) \mid x, y \in X \AND f : A_x -> A_y\}, \\
\Iso_X(A) &:= \{(x, y, f) \in \Hom_X(A) \mid f : A_x \cong A_y\},
\end{align*}
\vspace{-2em}
\begin{align*}
\sigma : \Hom_X(A) &--> X &
\tau : \Hom_X(A) &--> Y \\
(x, y, f) &|--> x, &
(x, y, f) &|--> y, \\[1em]
\iota : X &--> \Hom_X(A) &
\nu : \Iso_X(A) &--> \Iso_X(A) \\
x &|--> (x, x, 1_{A_x}), &
(x, y, f) &|--> (y, x, f^{-1}),
\end{align*}
\begin{align*}
\mu : \Hom_X(A) \times_X \Hom_X(A) &--> \Hom_X(A) \\
((y, z, g), (x, y, f)) &|--> (x, z, g \circ f)
\end{align*}
(where the fiber product is with respect to $\sigma$ on the left, $\tau$ on the right).  These maps equip $\Hom_X(A)$ with the structure of a category, the \defn{transformation category} of $A$, with objects $x \in X$ and morphisms
\begin{align*}
\Hom_X(A)(x, y) = \{(x, y, f) \mid f : A_x -> A_y\}.
\end{align*}
We will usually refer to morphisms as $f : A_x -> A_y$, or simply $f$, instead of $(x, y, f)$.  The core $\Iso_X(A) \subseteq \Hom_X(A)$ is the \defn{symmetric groupoid} of $A$, generalizing the symmetric group (when $X = 1$).

Equip $\Hom_X(A)$ with the topology generated by $\sigma, \tau$ and the subbasic open sets
\begin{align*}
\sqsqbr{S |-> T} := \{f : A_x -> A_y \in \Hom_X(A) \mid f(S_x) \cap T_y \ne \emptyset\}
\end{align*}
for open $S, T \subseteq A$.  Clearly the map $(S, T) |-> \sqsqbr{S |-> T}$ preserves unions in each variable separately; thus it suffices to consider $\sqsqbr{S |-> T}$ for $S, T$ in some basis $\@S$ of open sets in $A$.  Note that when $S, T$ are (basic) open sections, we have
\begin{align*}
\sqsqbr{S |-> T} = \{f : A_x -> A_y \in \Hom_X(A) \mid x \in p(S) \AND f(S_x) = T_y\}.
\end{align*}
Thus the topology on $\Hom_X(A)$ generalizes the pointwise convergence topology on $A^A$ for a set $A$ (when $X = 1$).  Equip $\Iso_X(A) \subseteq \Hom_X(A)$ with the subspace topology.

\begin{proposition}
\label{thm:symgpd-cts}
The maps $\sigma, \tau, \iota, \nu, \mu$ are continuous, hence equip $\Hom_X(A)$ (resp., $\Iso_X(A)$) with the structure of a topological category (resp., topological groupoid).
\end{proposition}
\begin{proof}
$\sigma, \tau$ are continuous by definition.

Continuity of $\iota$: we have $\iota^{-1}(\sigma^{-1}(U)) = U = \iota^{-1}(\tau^{-1}(U))$ for open $U \subseteq X$, and $\iota^{-1}(\sqsqbr{S |-> T}) = p(S \cap T)$ for open $S, T \subseteq A$.

Continuity of $\nu$: we have $\nu^{-1}(\sigma^{-1}(U)) = \tau^{-1}(U)$, $\nu^{-1}(\tau^{-1}(U)) = \sigma^{-1}(U)$, and $\nu^{-1}(\sqsqbr{S |-> T}) = \sqsqbr{T |-> S}$.

Continuity of $\mu$: we have $\mu^{-1}(\sigma^{-1}(U)) = \Hom_X(A) \times_X \sigma^{-1}(U)$ and $\mu^{-1}(\tau^{-1}(U)) = \tau^{-1}(U) \times_X \Hom_X(A)$.  For open sections $S, T \subseteq A$, to check that $\mu^{-1}(\sqsqbr{S |-> T})$ is open: let $((y, z, g), (x, y, f)) \in \mu^{-1}(\sqsqbr{S |-> T})$, i.e., $g(f(S_x)) = T_z$, and let $R \subseteq A$ be any open section containing $f(S_x)$; then $((y, z, g), (x, y, f)) \in \sqsqbr{R |-> T} \times_X \sqsqbr{S |-> R} \subseteq \mu^{-1}(\sqsqbr{S |-> T})$.
\end{proof}

\begin{proposition}
\label{thm:symgpd-qpol}
If $X, A$ are quasi-Polish, then so are $\Hom_X(A), \Iso_X(A)$.
\end{proposition}
\begin{proof}
Recall the fiberwise lower powerspace construction from \cref{sec:lowpow}.  Let
\begin{align*}
e : \Hom_X(A) &--> \@F_{X \times X}(A \times A) \\
(x, y, f) &|--> (x, y, \graph(f))
\end{align*}
(where $A \times A$ is equipped with the projection map $p \times p : A \times A -> X \times X$).  Clearly $e$ is injective.  Subbasic open sets in $\@F_{X \times X}(A \times A)$ are $\@F_{X \times X}(p \times p)^{-1}(U \times V)$ for open $U, V \subseteq X$, and $\Dia_{X \times X}(S \times T)$ for open $S, T \subseteq A$; these have $e$-preimages $\sigma^{-1}(U) \cap \tau^{-1}(V)$ and $\sqsqbr{S |-> T}$ respectively.  So $e$ is an embedding.

Hence, to show that $\Hom_X(A)$ is quasi-Polish, it suffices to show that $\im(e)$ is $\*\Pi^0_2$.  Let $\@S$ be a countable basis of open sections in $A$.  We claim that for $(x, y, F) \in \@F_{X \times X}(A \times A)$,
\begin{align*}
(x, y, F) \in \im(e) \iff \left(\begin{aligned}
&\forall S \in \@S\, (x \in p(S) \implies \exists T \in \@S\, ((x, y, F) \in \Dia_{X \times X}(S \times T))) \AND \\
&\forall S \in \@S, T_1, T_2 \in \@S \left(\begin{aligned}
&(x, y, F) \in \Dia_{X \times X}(S \times T_1) \cap \Dia_{X \times X}(S \times T_2) \implies \\
&\exists \@S \ni T_3 \subseteq T_1 \cap T_2\, ((x, y, F) \in \Dia_{X \times X}(S \times T_3))
\end{aligned}\right)
\end{aligned}\right).
\end{align*}
The two conditions on the right-hand side are easily seen to say respectively that the fiber of $F \subseteq A_x \times A_y$ over each $a \in A_x$ is nonempty and has at most one element.

To show that $\Iso_X(A) \subseteq \Hom_X(A)$ is quasi-Polish, add the ``converses'' of the above conditions (with the roles of $x, y$ swapped).
\end{proof}

\subsection{Étale structures}
\label{sec:etalestr}

Let $\@L$ be a (multi-sorted) first-order language and $X$ be a topological space.  An \defn{étale $\@L$-structure over $X$} is, succinctly, an $\@L$-structure in the category of étale spaces over $X$ (see \cite[\S D1.2]{Jeleph}).  Explicitly, an étale $\@L$-structure $\@M = (P^\@M, R^\@M, f^\@M)_{P, R, f \in \@L}$ over $X$ consists of:
\begin{itemize}
\item  for each sort $P \in \@L_s$, an underlying étale space $p : P^\@M -> X$ over $X$ (when $\@L$ is one-sorted, we denote $P^\@M$ for the unique sort $P$ by $M$);
\item  for each $(P_0, \dotsc, P_{n-1})$-ary relation symbol $R \in \@L_r$, an open subset $R^\@M \subseteq P_0^\@M \times_X \dotsb \times_X P_{n-1}^\@M$;
\item  for each $(P_0, \dotsc, P_{n-1})$-ary $Q$-valued function symbol $f \in \@L_f$, a continuous map $f^\@M : P_0^\@M \times_X \dotsb \times_X P_{n-1}^\@M -> Q^\@M$ over $X$.
\end{itemize}
An étale $\@L$-structure $\@M$ over $X$ gives rise to an $\@L$-structure
\begin{align*}
\@M_x := (P^\@M_x, R^\@M_x, f^\@M_x)_{P,R,f \in \@L}
\end{align*}
for each $x \in X$, which varies ``continuously'' in $x$.

We say that an étale $\@L$-structure $\@M$ over $X$ is \defn{second-countable} if $\@L$ is countable, $X$ is second-countable, and the interpretation $P^\@M$ of each sort $P \in \@L_s$ is second-countable.

Let $\@M$ be an étale $\@L$-structure over $X$.  Its \defn{homomorphism category} is
\begin{align*}
\Hom_X(\@M) := \left\{(x, y, f) = (x, y, f_P)_{P \in \@L_s} \in \prod_{X \times X} (\Hom_X(P^\@M))_{P \in \@L_s} \relmiddle| f \text{ an $\@L$-homom.\ } \@M_x -> \@M_y\right\}
\end{align*}
where $\prod_{X \times X}$ denotes the fiber product over $X \times X$ with respect to the maps $\sigma, \tau : \Hom_X(P^\@M) -> X$, and we are identifying an element $((x, y, f_P))_P$ of the fiber product with $(x, y, f)$ where $f := (f_P)_P$.
The \defn{isomorphism groupoid} of $\@M$ is the core
\begin{align*}
\Iso_X(\@M) := \{(x, y, f) \in \Hom_X(\@M) \mid f \text{ an $\@L$-isomorphism } \@M_x \cong \@M_y\}.
\end{align*}
The category (resp., groupoid) operations of $\Hom_X(\@M), \Iso_X(\@M)$ are given sortwise.  As before, we will usually refer to morphisms in $\Hom_X(\@M), \Iso_X(\@M)$ as $f : \@M_x -> \@M_y$, or simply $f$, instead of $(x, y, f)$.

We equip $\Hom_X(\@M), \Iso_X(\@M)$ with the subspace topology inherited from $\prod_{X \times X} (\Hom_X(P^\@M))_{P \in \@L_s}$.  For a subbasic open set $\sqsqbr{S |-> T} \subseteq \Hom_X(P^\@M)$ (with $S, T \subseteq P^\@M$ open), we denote the corresponding subbasic open set of the product restricted to $\Hom_X(\@M)$ by
\begin{align*}
\sqsqbr{S |-> T}_P &= \{(x, y, f) \in \Hom_X(\@M) \mid f_P(S_x) \cap T_y \ne \emptyset\} \\
&= \{(x, y, f) \in \Hom_X(\@M) \mid x \in p(S) \AND f_P(S_x) = T_y\} \qquad\text{for open sections $S, T$}.
\end{align*}
Thus, the topology of $\Hom_X(\@M)$ is generated by $\sigma, \tau$ together with these sets (it is enough to consider $S, T$ in a basis).  We use the same notation for $\Iso_X(\@M)$.

\begin{proposition}
\label{thm:etalestr-isogpd-qpol}
If $\@M$ is second-countable, then $\Hom_X(\@M)$ (resp., $\Iso_X(\@M)$) is a quasi-Polish category (resp., groupoid).
\end{proposition}
\begin{proof}
By \cref{thm:symgpd-qpol}, $\prod_{X \times X} (\Hom_X(P^\@M))_{P \in \@L_s}$ is quasi-Polish.
Let $(x, y, f) = (x, y, f_P)_P \in \prod_{X \times X} (\Hom_X(P^\@M))_{P \in \@L_s}$; we must show that $f$ being a homomorphism $\@M_x -> \@M_y$ is a $\*\Pi^0_2$ condition.  For each sort $P \in \@L_s$, let $\@S_P$ be a countable basis of open sections in $P^\@M$.  For a $(P_0, \dotsc, P_{n-1})$-ary relation symbol $R \in \@L_r$, $f$ preserves $R$ iff
\begin{align*}
\forall (a_0, \dotsc, a_{n-1}) \in (P_0^\@M)_x \times \dotsb \times (P_{n-1}^\@M)_x\; (R^\@M_x(a_0, \dotsc, a_{n-1}) \implies R^\@M_y(f_{P_0}(a_0), \dotsc, f_{P_{n-1}}(a_{n-1}))),
\end{align*}
which, by considering open sections $S_i \ni a_i$ and $T_i \ni f_{P_i}(a_i)$, can be expressed as
\begin{align*}
\forall S_0, T_0 \in \@S_{P_0},\; \dotsc,\; S_{n-1}, T_{n-1} \in \@S_{P_{n-1}} \left(\begin{aligned}
&x \in p((S_0 \times_X \dotsb \times_X S_{n-1}) \cap R^\@M)
\AND \forall i\, (f_i \in \sqsqbr{S_i |-> T_i}) \\
&\implies y \in p((T_0 \times_X \dotsb \times_X T_{n-1}) \cap R^\@M)
\end{aligned}\right)
\end{align*}
(where $p$ denotes the projections from the various étale spaces to $X$).  Preservation of functions is similar (or can be reduced to preservation of their graphs).

Similarly, $\Iso_X(\@M) \subseteq \prod_{X \times X} (\Iso_X(P^\@M))_{P \in \@L_s}$ is $\*\Pi^0_2$, by taking the ``converse'' of the above conditions.
\end{proof}

\subsection{Fiberwise countable Borel spaces and structures}
\label{sec:cbstr}

We have the following analogs of the above topological notions in the Borel context.

Let $p : A -> X$ be a \defn{fiberwise countable Borel space} over a standard Borel space $X$, i.e., $A$ is standard Borel, $p$ is Borel, and $p$ is countable-to-1.  Equivalently, by the Lusin--Novikov uniformization theorem \cite[18.10]{Kcdst}, $A$ has a countable cover by \defn{Borel sections} $S \subseteq A$, meaning Borel sets such that $p|S : S -> X$ is injective.  For example, every second-countable étale space $p : A -> X$ over quasi-Polish $X$ has an underlying fiberwise countable Borel space, with a countable cover of $A$ by open sections yielding a countable cover by Borel sections.

\begin{remark}
\label{rmk:cbsp-etale}
Conversely, for every fiberwise countable Borel space $p : A -> X$, there are compatible (quasi-)Polish topologies on $A, X$ turning $p : A -> X$ into an étale space.  Indeed, let $\@S$ be a countable cover of $A$ by Borel sections.  Take any compatible Polish topologies on $A, X$ making $p$ continuous, and making each $S \in \@S$ open.  Let $\@T$ be a countable basis of open sets in $A$, each contained in some $S \in \@S$.  Let $X'$ be $X$ with a finer Polish topology making $p(T) \subseteq X$ open for each $T \in \@T$, and let $A'$ be $A$ with $p^{-1}(U)$ adjoined to its topology for each open $U \subseteq X'$.  Then $A'$ is Polish (since $A' \cong A \times_X X'$), $p : A' -> X'$ is continuous and open (since $p(T \cap p^{-1}(U)) = p(T) \cap U$ for basic open $T \cap p^{-1}(U) \subseteq A'$ with $T \in \@T$ and open $U \subseteq X'$), and $p : A' -> X'$ is étale since $\@S$ is a countable cover of $A'$ by open sections.
\end{remark}

The \defn{transformation category} $\Hom_X(A)$ and \defn{isomorphism groupoid} $\Iso_X(A)$ of $p : A -> X$ are defined exactly as in \cref{sec:symgpd}, with Borel structure generated by the basic Borel sets $\sqsqbr{S |-> T} \subseteq \Hom_X(A)$ (defined as in \cref{sec:symgpd}) for Borel sections $S, T \subseteq A$.  By \cref{rmk:cbsp-etale} and \cref{thm:symgpd-qpol}, or by a direct coding argument similar to the proof of \cref{thm:symgpd-qpol}, $\Hom_X(A), \Iso_X(A)$ are standard Borel, with Borel structure compatible with the quasi-Polish topologies when $p : A -> X$ is a second-countable étale space over quasi-Polish $X$.

For a countable language $\@L$ and standard Borel space $X$, a \defn{fiberwise countable Borel $\@L$-structure} $\@M = (P^\@M, R^\@M, f^\@M)_{P, R, f \in \@L}$ over $X$ consists of:
\begin{itemize}
\item  for each sort $P \in \@L_s$, an underlying fiberwise countable Borel space $p : P^\@M -> X$ over $X$ ($P^\@M$ denoted $M$ for one-sorted $\@L$);
\item  for each $(P_0, \dotsc, P_{n-1})$-ary $R \in \@L_r$, a Borel subset $R^\@M \subseteq P_0^\@M \times_X \dots \times_X P_{n-1}^\@M$;
\item  for each $(P_0, \dotsc, P_{n-1})$-ary $Q$-valued $f \in \@L_f$, a Borel map $f^\@M : P_0^\@M \times_X \dots \times_X P_{n-1}^\@M -> Q^\@M$ over $X$.
\end{itemize}
We think of this as a ``Borel family'' of countable $\@L$-structures $\@M_x$ for each $x \in X$.  Every second-countable étale $\@L$-structure $\@M$ over quasi-Polish $X$ has an underlying fiberwise countable Borel structure.

The \defn{homomorphism category} $\Hom_X(\@M)$ and \defn{isomorphism groupoid} $\Iso_X(\@M)$ of a fiberwise countable Borel structure $\@M$ are defined exactly as in \cref{sec:etalestr}.  By the proof of \cref{thm:etalestr-isogpd-qpol}, $\Hom_X(\@M), \Iso_X(\@M)$ are standard Borel, compatible with the quasi-Polish topologies when $\@M$ is a second-countable étale structure over quasi-Polish $X$.

\subsection{Uniformizing fiberwise structures}
\label{sec:cbstr-unif}


We now have the main result of this section:

\begin{theorem}
\label{thm:cbstr-isogpd-ff}
For any countable (multi-sorted) language $\@L$, standard Borel space $X$, and fiberwise countable Borel structure $\@M$ over $X$, there is a countable single-sorted relational language $\@L'$ and a full and faithful Borel functor $\Iso_X(\@M) -> \#S_\infty \ltimes \Mod_\#N(\@L')$.
\end{theorem}

We split the proof into two parts:

\begin{proposition}
\label{thm:cbstr-ext}
For any countable (multi-sorted) language $\@L$, standard Borel space $X$, and fiberwise countable Borel $\@L$-structure $\@M$ over $X$, there is a countable single-sorted relational language $\@L'$, a fiberwise countable Borel $\@L'$-structure $\@M'$ over $X$ which is fiberwise infinite, and a Borel identity-on-objects isomorphism of groupoids $\Iso_X(\@M) \cong \Iso_X(\@M')$.
\end{proposition}
\begin{proof}
By replacing functions with their graphs, we may assume $\@L$ is relational.  Let
\begin{align*}
\@L'_r := \@L_r \sqcup \@L_s \sqcup \{C_0, C_1, \dotsc\}
\end{align*}
where each $P \in \@L_s$ is treated as a unary relation symbol, as is each $C_i$.  The underlying fiberwise countable Borel  space of $\@M'$ is
\begin{align*}
M' := \bigsqcup_{P \in \@L_s} P^\@M \sqcup (X \times \#N)
\end{align*}
where $X \times \#N$ is equipped with the first coordinate projection to $X$ (with $(x, i) \in X \times \#N$ thought of as a newly added constant in $\@M'_x$).  Each $(P_0, \dotsc, P_{n-1})$-ary $R \in \@L_r$ is interpreted as
\begin{align*}
R^{\@M'} := R^\@M \subseteq P_0^\@M \times_X \dotsb \times_X P_{n-1}^\@M \subseteq M^{\prime n}_X,
\end{align*}
each $P \in \@L_s$ is interpreted as
\begin{align*}
P^{\@M'} := P^\@M \subseteq M',
\end{align*}
and each $C_i$ is interpreted as
\begin{align*}
C_i^{\@M'} := X \times \{i\} \subseteq M'.
\end{align*}
Clearly $\@L$-isomorphisms $f = (f_P)_{P \in \@L_s} : \@M_x \cong \@M_y$ are in canonical bijection with $\@L'$-isomorphisms $\@M'_x \cong \@M'_y$ (given by $\bigsqcup_P f_P : \bigsqcup_P P^\@M_x \cong \bigsqcup_P P^\@M_y$ together with the identity on the newly added constants).
\end{proof}

\begin{proposition}
\label{thm:cbstr-unif}
For any countable single-sorted relational language $\@L$, standard Borel space $X$, and fiberwise countable Borel $\@L$-structure $\@M$ over $X$ which is fiberwise infinite, there is a full and faithful Borel functor $F : \Iso_X(\@M) -> \#S_\infty \ltimes \Mod_\#N(\@L)$.
\end{proposition}
\begin{proof}
Using a countable cover of $M$ by Borel sections, we may find a Borel map $f : M -> \#N$ which is a bijection $f_x := f|M_x : M_x \cong \#N$ for each $x$ (see e.g., \cite[18.15]{Kcdst}).  Use these bijections to pushforward each $\@M_x$ to a structure on $\#N$.  That is, define $F$ on objects by
\begin{align*}
F : X &--> \Mod_\#N(\@L) \\
x &|--> (f_x)_*(\@M_x),
\end{align*}
i.e., for each $n$-ary $R \in \@L$ and $\vec{a} \in \#N$,
\begin{align*}
R^{F(x)}(\vec{a}) \coloniff R^\@M(f_x^{-1}(\vec{a}));
\end{align*}
clearly this is Borel (in $x$).  Define $F$ on morphisms by
\begin{align*}
F : \Iso_X(\@M) &--> \#S_\infty \ltimes \Mod_\#N(\@L) \\
(g : \@M_x -> \@M_y) &|--> (f_y \circ g \circ f_x^{-1}, F(x)).
\end{align*}
This is Borel, since letting $\@S$ be a countable basis of open sections in $M$,
\begin{align*}
(f_y \circ g \circ f_x^{-1})(i) = j
&\iff \exists S, T \in \@S\; (f_x^{-1}(i) \in S \AND f_y^{-1}(j) \in T \AND g \in \sqsqbr{S |-> T}).
\end{align*}
It is straightforward to check that this works.
\end{proof}

\begin{proof}[Proof of \cref{thm:cbstr-isogpd-ff}]
Apply \cref{thm:cbstr-ext} to get $\@L', \@M'$ and $G : \Iso_X(\@M) -> \Iso_X(\@M')$, then apply \cref{thm:cbstr-unif} (to $\@L', \@M'$) to get $F : \Iso_X(\@M') -> \#S_\infty \ltimes \Mod_\#N(\@L')$; the desired functor is $F \circ G$.
\end{proof}

\begin{remark}
\label{rmk:cbstr-isogpd-ff-io}
We may arrange for the functor $\Iso_X(\@M) -> \#S_\infty \ltimes \Mod_\#N(\@L')$ produced by \cref{thm:cbstr-isogpd-ff} to be not only full and faithful, but also injective on objects (so that it induces not only a reduction but an embedding of orbit equivalence relations $\#E_{\Iso_X(\@M)} -> \#E_{\#S_\infty}^{\Mod_\#N(\@L')}$).  Indeed, let $\@L', \@M', G, F$ be as above, and let $f_x$ be as in the proof of \cref{thm:cbstr-unif}.  Note that for each $x \in X$, $D_x := \{f_x(x, i)\}_{i \in \#N}$ (where $(x, i) \in M'_x$ are the newly added constants) is an infinite subset of $\#N$; moreover, the map $x |-> D_x$ is Borel.  Let $u : X -> \#S_\infty$ be an injective Borel map, and for each infinite subset $D \subseteq \#N$, let $v_D : \#N -> D$ be a bijection, with $D |-> v_D$ Borel.  Define for each $x \in X$
\begin{align*}
f'_x : M_x &\cong \#N \\
(x, i) &|-> v_{D_x}(u(x)(i)) \\
a &|-> f_x(a) \qquad\text{for all other $a \in M_x$}.
\end{align*}
Let $F'$ be obtained from \cref{thm:cbstr-unif} by replacing $f_x$ with $f'_x$.  Then $F'$ is injective on objects, since from $F'(x)$ we may recover $f'_x(x, i) \in \#N$ as the unique element of $C_i^{F'(x)}$ for each $i$, hence also $D_x = \{f'_x(x, i)\}_i$ and so $u(x) = v_{D_x}^{-1} \circ f'_x(x, -)$.
\end{remark}

\section{Non-Archimedean groupoids}
\label{sec:nonarchgpd}

In this section, we show that every non-Archimedean open quasi-Polish groupoid $G$ is canonically isomorphic to the isomorphism groupoid of an étale structure over $G^0$ (whose elements are left cosets of open subgroupoids).  This can be seen as a ``Yoneda embedding'' for such groupoids (see \cref{rmk:gpd-yoneda}), which together with the results of the previous section immediately yields \cref{thm:intro-nonarchgpd-rep}.

\subsection{Étale spaces of left cosets}

Let $G$ be a non-Archimedean open topological groupoid.  For an open subgroupoid $U \subseteq G$, a \defn{left coset} of $U$ is a subset of the form $g \cdot U$ for some $g \in G$ with $\sigma(g) \in U$ (i.e., $g \cdot U \ne \emptyset$); thus, $g \cdot U$ is an open subset of $G_{\tau(g)}$ (recall from \cref{sec:gpd} that by default, we always regard a groupoid $G$ as a set over $G^0$ via $\tau$).  We have the \defn{left coset equivalence relation}
\begin{align*}
g \sim_U h \coloniff g \cdot U = h \cdot U \iff h \in g \cdot U \iff g^{-1} \cdot h \in U
\end{align*}
between elements of $\sigma^{-1}(U)$ in the same $\tau$-fiber.  Since $U \subseteq G$ is open, so is $({\sim_U}) \subseteq \sigma^{-1}(U) \times_{G^0} \sigma^{-1}(U)$, whence by \cref{thm:etale-quotient} the quotient \defn{space of left cosets}
\begin{align*}
G/U := \sigma^{-1}(U)/{\sim_U}
\end{align*}
is étale over $G^0$ via the descendant along the open projection $\pi = \pi_U : \sigma^{-1}(U) ->> G/U$ of $\tau : \sigma^{-1}(U) -> G^0$ which we continue to denote by $\tau : G/U -> G^0$.  By the remarks after \cref{thm:etale-quotient}, open sections $\pi(S) \subseteq G/U$ are in bijection with $\sim_U$-invariant $\sim_U$-small open $S \subseteq \sigma^{-1}(U)$.  We say \defn{$U$-invariant} instead of $\sim_U$-invariant, i.e.,
\begin{align*}
S \cdot U = S,
\end{align*}
and \defn{$U$-small} instead of $\sim_U$-small, i.e.,
\begin{align*}
S^{-1} \cdot S \subseteq U.
\end{align*}
Note that we always have a canonical open section
\begin{align*}
\pi(U) = \{U_x \mid x \in U^0\} \subseteq G/U,
\end{align*}
with $\tau(\pi(U)) = \tau(U) = U^0$; we call it the \defn{unit section}.

The left multiplication action of $G$ on itself descends to a continuous action $G \curvearrowright G/U$
\begin{align*}
{\cdot} : G \times_{G^0} G/U &--> G/U \\
(g, h \cdot U) &|--> g \cdot h \cdot U
\end{align*}
(fiber product with respect to $\sigma$ on the left).  Hence for an open section $\pi(S) \subseteq G/U$ (corresponding to $U$-invariant $U$-small open $S \subseteq \sigma^{-1}(U)$), the right multiplication map
\begin{align*}
(-) \cdot S : \sigma^{-1}(\tau(S)) &--> G/U \\
g &|--> g \cdot S = g \cdot \pi(S)_{\sigma(g)}
\end{align*}
is continuous.

Now consider two open subgroupoids $U, V \subseteq G$ and $V$-invariant $V$-small open $S \subseteq \sigma^{-1}(V)$.
Then $(-) \cdot S$ descends to the quotient $G/U = \sigma^{-1}(U)/{\sim_U}$ iff $\sigma^{-1}(U) \subseteq \sigma^{-1}(\tau(S))$, i.e., $U^0 \subseteq \tau(S)$, and whenever $g, h \in \sigma^{-1}(U)_x$ with $g^{-1} \cdot h \in U$, we have $g \cdot S = h \cdot S$, i.e., $g^{-1} \cdot h \cdot S \subseteq S$, i.e., $g^{-1} \cdot h \in S \cdot S^{-1}$; in other words, iff
\begin{align*}
U \subseteq S \cdot S^{-1}
\end{align*}
(which implies $U^0 \subseteq \tau(S)$).
To summarize:

\begin{lemma}
\label{thm:nonarchgpd-rightmult}
For any open subgroupoids $U, V \subseteq G$ and $V$-invariant $V$-small open $S \subseteq \sigma^{-1}(V)$ such that
\begin{gather*}
U \subseteq S \cdot S^{-1},
\end{gather*}
we have a well-defined continuous right multiplication map over $G^0$
\begin{align*}
f_{U,V,S} := (-) \cdot S : G/U &--> G/V.
\qed
\end{align*}
\end{lemma}

In particular, when $U \subseteq V$ and $S := V$, we get the canonical projection map $G/U -> G/V$ induced by the inclusions $\sigma^{-1}(U) \subseteq \sigma^{-1}(V)$ and $({\sim_U}) \subseteq ({\sim_V})$, which we denote by
\begin{align*}
\pi_{U,V} := f_{U,V,V} = (-) \cdot V : G/U &--> G/V.
\end{align*}

\subsection{The canonical structure}
\label{sec:nonarchgpd-str}

Fix a basis $\@U$ of open subgroupoids in $G$, together with for each $U \in \@U$ a family $\@S_U$ of $U$-invariant $U$-small open subsets of $\sigma^{-1}(U)$ descending to a basis of open sections in $G/U$, with $U \in \@S_U$.
Note that when $G$ is second-countable, using \cref{thm:nonarchgpd-basis-ctble}, we may find $\@U$ and each $\@S_U$ countable.

\begin{lemma}
\label{thm:nonarchgpd-basis-section}
$\bigcup_{U \in \@U} \@S_U$ is a basis of open sets in $G$.
\end{lemma}
\begin{proof}
Let $V \subseteq G$ be open and $g : x -> y \in V$.  By continuity of multiplication, there are open $g \in W \subseteq G$ and $x \in U \in \@U$ such that $W \cdot U \subseteq V$.  Then $W \cdot U \subseteq G$ is $U$-invariant open, whence $\pi_U(W \cdot U) \subseteq G/U$ is a union of $\pi_U(S)$ for $S \in \@U$, whence $W \cdot U = \pi_U^{-1}(\pi_U(W \cdot U)$ is a union of $S = \pi_U^{-1}(S)$ for $S \in \@U$, one of which contains $g$.
\end{proof}

Let $\@M$ be the étale structure over $G^0$ in the language $\@L$ consisting of a sort for each $U \in \@U$, interpreted in $\@M$ as $G/U$, and a function symbol for each $U, V \in \@U$ and $S \in \@S_V$ satisfying $U \subseteq S \cdot S^{-1}$ as in \cref{thm:nonarchgpd-rightmult}, interpreted in $\@M$ as the map $f_{U,V,S} = (-) \cdot S : G/U -> G/V$ defined there.  We call $\@M$ the \defn{canonical structure} of $G$ (with respect to the chosen $\@U, \@S_U$).  When $G$ is second-countable and $\@U, \@S_U$ are countable, $\@M$ is second-countable.

\subsection{The homomorphism category}
\label{sec:nonarchgpd-hom}

Following \cite[\S3.5]{Mo2}, we now give an explicit description of the homomorphism category $\Hom_{G^0}(\@M)$ of the canonical structure $\@M$.  This ``topological Yoneda lemma'' argument forms the heart of the proof of \cref{thm:intro-nonarchgpd-yoneda-iso}.

Fix objects $x, y \in G^0$.  Given a homomorphism $h : \@M_x -> \@M_y$, i.e., a morphism $x -> y$ in $\Hom_{G^0}(\@M)$, we have for each sort $U \in \@U$ a map $h_U : (G/U)_x -> (G/U)_y$, which when $x \in U$ we may evaluate at the unit coset $U_x = \pi_U(x) \in (G/U)_x$, yielding an element $h_U(U_x) \in (G/U)_y$.

As $U$ varies, the sets $(G/U)_y$ are related as follows: for $x \in U \subseteq V \in \@U$, we have the projection map $\pi_{U,V} = (-) \cdot V : (G/U)_y -> (G/V)_y$, which is functorial in $(U, V)$ (i.e., $\pi_{U,U} = 1$ and $\pi_{V,W} \circ \pi_{U,V} = \pi_{U,W}$).  We may thus form the inverse limit
\begin{align*}
\projlim_{x \in U \in \@U} (G/U)_y = \{\vec{a} = (a_U)_U \in \prod_{x \in U \in \@U} (G/U)_y \mid \forall U \subseteq V\, (\pi_{U,V}(a_U) = a_V)\}
\end{align*}
consisting of all coherent families of elements.

\begin{lemma}
\label{thm:nonarchgpd-yoneda}
For all $x, y \in G^0$, the above-described map
\begin{align*}
\Phi : \Hom_{G^0}(\@M)(x, y) &--> \prod_{x \in U \in \@U} (G/U)_y \\
(h : \@M_x -> \@M_y) &|--> (h_U(U_x))_U
\end{align*}
restricts to a bijection
\begin{align*}
\Phi : \Hom_{G^0}(\@M)(x, y) &\cong \projlim_{x \in U \in \@U} (G/U)_y
\end{align*}
with inverse
\begin{align*}
\Psi : \projlim_{x \in U \in \@U} (G/U)_y &--> \Hom_{G^0}(\@M)(x, y) \\
\vec{a} = (a_U)_U &|--> \left(\begin{aligned}
(G/U)_x &-> (G/U)_y \\
b &|-> a_V \cdot S \quad\text{for any $b \subseteq S \in \@S_U$, $x \in V \in \@U$, $V \subseteq S \cdot S^{-1}$}
\end{aligned}\right)_U.
\tag{$*$}
\end{align*}
\end{lemma}
\begin{proof}
%
First, we check that $\Phi(h) \in \projlim_U (G/U)_y$.  Let $x \in U \subseteq V \in \@U$.  Since $V \in \@S_V$ and $V \cdot V^{-1} \supseteq U \cdot U^{-1} = U$, the map $f_{U,V,V}$ is part of $\@M$.  So since $h$ is a homomorphism, we have $\pi_{U,V}(h_U(U_x)) = f_{U,V,V}(h_U(U_x)) = h_V(f_{U,V,V}(U_x)) = h_V(U_x \cdot V) = h_V(V_x)$, as desired.

Next, we check that $\Psi$ is well-defined.  Let $\vec{a} = (a_U)_U \in \projlim_{x \in U \in \@U} (G/U)_y$.  For each $b \in (G/U)_x$, note that the set of $(V, S)$ allowed in the right-hand side of ($*$) is down-directed.  For any such $(V, S)$, we have $a_V \cdot S \in (G/U)_y$ by \cref{thm:nonarchgpd-rightmult}; and $a_V \cdot S$ is monotone in $(V, S)$ (since $\vec{a} \in \projlim_{x \in U \in \@U} (G/U)_y$), hence (being a left coset) is constant for all $(V, S)$.  So ($*$) defines a map $\Psi(\vec{a})_U : (G/U)_x -> (G/U)_y$ for each $U$.  That these maps form a homomorphism $\Psi(\vec{a}) : \@M_x -> \@M_y$ follows from left-equivariance of the right multiplication functions $f_{U,V,S}$: if $\Psi(\vec{a})_U(b) = a_W \cdot T$ for $b \subseteq T \in \@S_U$, $x \in W \in \@U$, and $W \subseteq T \cdot T^{-1}$ as in ($*$),
then $f_{U,V,S}(\Psi(\vec{a})_U(b)) = a_W \cdot T \cdot S = a_W \cdot f_{U,V,S}(T) = \Psi(\vec{a})_V(f_{U,V,S}(b))$ since $f_{U,V,S}(b) \subseteq f_{U,V,S}(T) \in \@S_V$ and
$f_{U,V,S}(T) \cdot f_{U,V,S}(T)^{-1}
= (T \cdot S) \cdot (S^{-1} \cdot T^{-1})
\supseteq T \cdot U \cdot T^{-1}
\supseteq T \cdot T^{-1}
\supseteq W$ (using that $T \in \@S_U$ so $\sigma(T) \subseteq U$).

To check that $\Phi \circ \Psi = 1$: taking $b := U_x$, $S := U$, and $V := U$ in ($*$) yields $\Phi(\Psi(\vec{a}))_U = \Psi(\vec{a})_U(U_x) = a_U \cdot U = a_U$.

Finally, to check that $\Psi \circ \Phi = 1$: let $h : \@M_x -> \@M_y$, $U \in \@U$, and $a \in (G/U)_x$; we must check that $h_U(a) = \Psi(\Phi(h))_U(a)$.  Let $a \subseteq S \in \@S_U$, $x \in V \in \@U$, and $V \subseteq S \cdot S^{-1}$ as in ($*$), so that $\Psi(\Phi(h))_U(a) = \Psi((h_U(U_x))_U)_U(a) = h_V(V_x) \cdot S$.  Since $V \subseteq S \cdot S^{-1}$, the map $f_{V,U,S}$ is part of $\@M$, whence $h_V(V_x) \cdot S = f_{V,U,S}(h_V(V_x)) = h_U(f_{V,U,S}(V_x)) = h_U(V_x \cdot S) = h_U(S_x) = h_U(a)$.
\end{proof}

\begin{remark}
\label{rmk:gpd-yoneda}
When $G$ is a discrete groupoid, we may pick $\@U := \{\{x\} \mid x \in G^0\}$ and $\@S_{\{x\}} := \{\{g\} \mid g \in \sigma^{-1}(x)\}$.  Then each fiber of the canonical structure $\@M_x$ consists of the sets $G(y, x)$ for all $y \in G^0$, equipped with the right multiplication maps $(-) \cdot g : G(y, x) -> G(z, x)$ for $g : z -> y$; in other words, $\@M_x$ is the right multiplication $G$-action on morphisms with target $x$, or the representable presheaf at $x$.  \Cref{thm:nonarchgpd-yoneda} says that right $G$-equivariant maps $\@M_x -> \@M_y$ are in bijection (via evaluation at the unit morphism $x$) with morphisms $x -> y$, which is the usual Yoneda lemma for groupoids.
\end{remark}

\begin{remark}
When $G$ is a non-Archimedean topological group, $\@U$ is a neighborhood basis of its identity consisting of open subgroups, and $\@S_U$ for $U \in \@U$ is simply the set of left cosets of $U$.  Then (the unique fiber of) $\@M$ consists of, for each $U \in \@U$, the set $G/U$ of left cosets of $U$, together with the right multiplication maps $(-) \cdot S : G/U -> G/V$ for left cosets $S \in \@S_V$ with $U \subseteq S \cdot S^{-1}$; we recover in this case the proof of \cref{thm:intro-nonarchgrp-rep} sketched in the Introduction.
\end{remark}

We next transport the topology on $\Hom_{G^0}(\@M)$ (as defined in \cref{sec:etalestr}) across the bijection in \cref{thm:nonarchgpd-yoneda}:

\begin{lemma}
\label{thm:nonarchgpd-yoneda-top}
The topology on $\Hom_{G^0}(\@M)$ is generated by the maps $\sigma, \tau : \Hom_{G^0}(\@M) -> G^0$ together with the sets
\begin{align*}
\sqsqbr{\pi_U(U) |-> \pi_U(T)}_U = \{h : \@M_x -> \@M_y \in \Hom_{G^0}(\@M) \mid x \in U \AND h_U(U_x) = T_y\}
\end{align*}
for $U \in \@U$ and $T \in \@S_U$.
\end{lemma}
\begin{proof}
We must show that any subbasic open set $\sqsqbr{\pi_U(S) |-> \pi_U(T)}_U$ for $U \in \@U$ and $S, T \in \@S_U$ can be generated by $\sigma, \tau$ and sets of the above form.  Let $h : \@M_x -> \@M_y \in \sqsqbr{\pi_U(S) |-> \pi_U(T)}_U$, i.e., $x \in \tau(S)$ and $h_U(S_x) = T_y$.  Let $V \in \@U$ with $x \in V \subseteq S \cdot S^{-1}$, so that $h_U(S_x) = h_V(V_x) \cdot S$ by \cref{thm:nonarchgpd-yoneda}.  Pick any $g \in h_V(V_x)$ and $k \in S_{\sigma(g)}$ (possible since $g \in h_V(V_x) \in G/V$ whence $g \in \sigma^{-1}(V) \subseteq \sigma^{-1}(S \cdot S^{-1}) = \sigma^{-1}(\tau(S))$), so that $g \cdot k \in h_V(V_x) \cdot S = h_U(S_x) = T_y$.  By continuity of multiplication, there are open $k \in S' \subseteq S$ and $g \in T' \subseteq \sigma^{-1}(\tau(S))$ with $T' \in \@S_V$ such that $T' \cdot S' \subseteq T$.  We claim that
\begin{align*}
h \in \sqsqbr{\pi_V(V) |-> \pi_V(T')}_V \subseteq \sqsqbr{\pi_U(S) |-> \pi_U(T)}_U.
\end{align*}
Since $h_V(V_x) \ni g \in T'$, $h \in \sqsqbr{\pi_V(V) |-> \pi_V(T')}_V$.  For any $h' : \@M_{x'} -> \@M_{y'} \in \sqsqbr{\pi_V(V) |-> \pi_V(T')}_V$, i.e., $x' \in V$ and $h'_V(V_{x'}) = T'_{y'}$, we have $x' \in V \subseteq S \cdot S^{-1}$, whence by \cref{thm:nonarchgpd-yoneda}, $h'_U(S_{x'}) = T'_{y'} \cdot S \supseteq T'_{y'} \cdot S' \subseteq T_{y'}$, and so $h'_U(S_{x'}) = T_{y'}$ since both are left cosets of $U$ and $T'_{y'} \cdot S' \ne \emptyset$ since $T' \subseteq \sigma^{-1}(\tau(S))$.
\end{proof}

Since the right multiplication maps $f_{U,V,S} : G/U -> G/V$ in $\@M$ are equivariant with respect to left multiplication, each morphism $g : x -> y \in G$ yields a homomorphism $g \cdot (-) : \@M_x -> \@M_y$; in other words, we have a canonical identity-on-objects functor
\begin{align*}
\eta : G &--> \Hom_{G^0}(\@M) \\
(g : x -> y) &|--> (g \cdot (-) : \@M_x -> \@M_y).
\end{align*}
Composed with the bijection in \cref{thm:nonarchgpd-yoneda}, this becomes for each $x, y \in G^0$
\begin{align*}
\Phi \circ \eta : G(x, y) &--> \projlim_{x \in U \in \@U} (G/U)_y \\
(g : x -> y) &|--> (g \cdot U)_U.
\end{align*}

\begin{proposition}
\label{thm:nonarchgpd-yoneda-emb-dense}
If $G$ is $T_0$, then $\eta : G -> \Hom_{G^0}(\@M)$ is a topological embedding with $\tau$-fiberwise dense image.
\end{proposition}
\begin{proof}
We use the description of the topology of $\Hom_{G^0}(\@M)$ from \cref{thm:nonarchgpd-yoneda-top}.  For $U \in \@U$ and $T \in \@S_U$, we have
\begin{align*}
\eta^{-1}(\sqsqbr{\pi_U(U) |-> \pi_U(T)}_U) = \{g : x -> y \in G \mid x \in U \AND g \cdot U = T_y\} = T.
\end{align*}
It follows from \cref{thm:nonarchgpd-basis-section} and $G$ being $T_0$ that $\eta$ is an embedding.

For density, consider a basic open set $A = \sigma^{-1}(V) \cap \tau^{-1}(W) \cap \bigcap_{i < n} \sqsqbr{\pi_{U_i}(U_i) |-> \pi_{U_i}(T_i)}_{U_i} \subseteq \Hom_{G^0}(\@M)$, where $V, W \subseteq G^0$ are open, $U_i \in \@U$, and $T_i \in \@S_{U_i}$.  Let $y \in G^0$; we must show that if $A_y \ne \emptyset$, then there is a $g \in G$ such that $\eta(g) \in A_y$.  Let $h : \@M_x -> \@M_y \in A_y$.  Then $y \in W$, and $x \in V \cap \bigcap_i U_i$, whence there is some $x \in U \in \@U$ with $U \subseteq \bigcap_i U_i$ and $G^0 \cap U \subseteq V$.
Pick any $g : \@M_z -> \@M_y \in h_U(U_x)$.  Then $\eta(g) \in \tau^{-1}(W)$ since $y \in W$, $\eta(g) \in \sigma^{-1}(U)$ since $h_U(U_x) \in G/U$, and for each $i$, $\eta(g)((U_i)_z) = g \cdot U_i \subseteq h_U(U_x) \cdot U_i = f_{U,U_i,U_i}(h_U(U_x)) = h_{U_i}(f_{U,U_i,U_i}(U_x)) = h_{U_i}((U_i)_x) = (T_i)_y$, i.e., $\eta(g) \in \sqsqbr{\pi_{U_i}(U_i) |-> \pi_{U_i}(T_i)}_{U_i}$.
\end{proof}

\subsection{The isomorphism groupoid}
\label{sec:nonarchgpd-iso}

Restricting the above to the case where $G$ is a non-Archimedean quasi-Polish groupoid and $\@U, \@S_U$ are chosen to be countable, we have the main results of this section:

\begin{theorem}
\label{thm:nonarchgpd-yoneda-iso}
For open non-Archimedean quasi-Polish $G$, and countable $\@U, \@S_U$, $\eta : G -> \Hom_{G^0}(\@M)$ is an isomorphism of topological groupoids $G \cong \Iso_{G^0}(\@M)$.
\end{theorem}
\begin{proof}
Since $G$ is a groupoid, $\eta$ lands in $\Iso_{G^0}(\@M) \subseteq \Hom_{G^0}(\@M)$.  By \cref{thm:nonarchgpd-yoneda-emb-dense}, $\eta : G -> \Iso_{G^0}(\@M)$ is a $\tau$-fiberwise dense embedding, hence by \cref{thm:subgpd-dense} a homeomorphism.
\end{proof}

\begin{theorem}
\label{thm:nonarchgpd-rep}
For any open non-Archimedean quasi-Polish groupoid $G$, there is a countable single-sorted relational language $\@L$, an $\@L_{\omega_1\omega}$-sentence $\phi$, and a Borel equivalence of groupoids $G -> \#S_\infty \ltimes \Mod_\#N(\@L, \phi)$.
\end{theorem}
\begin{proof}
Combining \cref{thm:nonarchgpd-yoneda-iso} with \cref{thm:cbstr-isogpd-ff} yields $\@L$ as above and a full and faithful Borel functor $F : G -> \#S_\infty \ltimes \Mod_\#N(\@L)$.  By \cref{thm:functor-borel-equiv}, the essential image $[F(G^0)]_{\#S_\infty} \subseteq \Mod_\#N(\@L)$ is Borel, hence by \cref{thm:lopez-escobar} equal to $\Mod_\#N(\@L, \phi)$ for some $\@L_{\omega_1\omega}$-sentence $\phi$, whence $F : G -> \#S_\infty \ltimes \Mod_\#N(\@L, \phi)$ is a Borel equivalence.
\end{proof}

\begin{corollary}
Up to Borel equivalence, the following classes of standard Borel groupoids coincide:
\begin{enumerate}
\item[(i)]  open non-Archimedean quasi-Polish groupoids;
\item[(ii)]  action groupoids of Borel $\#S_\infty$-actions;
\item[(iii)]  groupoids of models of $\@L_{\omega_1\omega}$-sentences on $\#N$.
\end{enumerate}
Hence, so do their classes of orbit equivalence relations, up to Borel bireducibility.
\end{corollary}
\begin{proof}
(i)$\implies$(iii) is by \cref{thm:nonarchgpd-rep}, (iii)$\implies$(ii) is immediate, and (ii)$\implies$(i) is by the Becker--Kechris theorem \cite[5.2.1]{BK}.
\end{proof}


\begin{remark}
\label{rmk:nonarchgpd-rep-io}
As in \cref{rmk:cbstr-isogpd-ff-io}, we may arrange for the equivalence of groupoids $G -> \#S_\infty \ltimes \Mod_\#N(\@L, \phi)$ in \cref{thm:nonarchgpd-rep} to be injective on objects (hence for the reduction between the corresponding orbit equivalence relations to be an embedding).
\end{remark}

\begin{remark}
\label{rmk:nonarchgpd-rep-bo}
We cannot, however, further require the equivalence $G -> \#S_\infty \ltimes \Mod_\#N(\@L, \phi)$ to be \emph{bijective} on objects, i.e., an isomorphism.  Indeed, if $G$ is the disjoint union of two non-isomorphic non-Archimedean Polish groups (regarded as a groupoid with two objects), then $G$ cannot fully faithfully embed with invariant image into the action groupoid of a single group action.
\end{remark}

\begin{remark}
Moerdijk \cite[\S6]{Mo1} also gives a computation which, when adapted to our setting,%
\footnote{In topos-theoretic terms, Moerdijk \cite[\S6.4--5]{Mo1} computes a site for the classifying topos $\!BG$ of $G$; this is the classifying topos for a geometric theory such that $G^0$ parametrizes its models via the canonical structure $\@M$, meaning that the geometric morphism $\@M : G^0 -> \!BG$ associated to $\@M$ is an open surjection (whose pullback with itself is $G$).
By starting from a countable posite for $G$ (see \cite[\S8]{Cqpol}), we may take this site for $\!BG$ to be countable, corresponding to a countable $\omega_1$-coherent theory (as in \cite{Cscc}).

We also have an open geometric surjection $\@N : Y ->> \!BG$ from a quasi-Polish space $Y$ given by a standard parametrization of countable models of $\!BG$, e.g., $Y$ is the space of models of $\!BG$ on quotients of partial equivalence relations on $\#N$ (see \cite[C5.2.8(c)]{Jeleph}).  The pullback of $\@M$ along $\@N$ is a continuous open surjection $\Iso_{G^0,Y}(\@M,\@N) ->> Y$, the projection from the quasi-Polish space $\Iso_{G^0,Y}(\@M,\@N)$ of isomorphisms between a fiber of $\@M$ and a fiber of $\@N$.  In particular, every countable model of $\!BG$ is (isomorphic to a model parametrized by $Y$, hence) isomorphic to some $\@M_x$.  Hence, the site for $\!BG$ axiomatizes the fibers of $\@M$ up to isomorphism.  It is now straightforward to tweak this to an axiomatization of the fibers of $\@M'$ ($\@M$ + $\aleph_0$-many constants) from \cref{thm:cbstr-unif}.}
shows that the $\@L_{\omega_1\omega}$-sentence $\phi$ in \cref{thm:nonarchgpd-rep} can be computed ``explicity'' from a $\*\Pi^0_2$-definition of $G$ (witnessing quasi-Polishness), thus avoiding the use of Lopez-Escobar (\cref{thm:lopez-escobar}) in the proof of \cref{thm:nonarchgpd-rep}.  We will not give the details, which are long and involve tedious syntactic manipulations (especially when further adapted to work in the metric setting, \cref{thm:locpolgpd-rep}).
\end{remark}

By combining the above results with the results in \cite{Cscc}, we obtain a complete correspondence between countable discrete $\@L_{\omega_1\omega}$-theories and open non-Archimedean Polishable standard Borel groupoids.  This result is best stated in the language of 2-categories; we refer to \cite{Cscc} and any standard reference on 2-category theory (e.g., \cite[B1.1]{Jeleph}) for the definitions used below.  Briefly, to each countable $\@L_{\omega_1\omega}$-theory $\@T$, we associate its \defn{syntactic Boolean $\omega_1$-pretopos} $\-{\ang{\@L \mid \@T}}^B_{\omega_1}$, which is the category of \defn{$\@L_{\omega_1\omega}$-imaginary sorts} (definable quotients of countable disjoint unions of formulas) and definable functions for $\@T$, whose finite limits and countable colimits encode the syntax of $\@T$ modulo logically irrelevant details.  An \defn{$\@L_{\omega_1\omega}$-interpretation} $(\@L, \@T) -> (\@L', \@T')$ between two theories (over possibly different languages) is defined as usual for first-order theories, except the domain is allowed to be an $\@L'_{\omega_1\omega}$-imaginary sort; equivalently, an interpretation is a functor $\-{\ang{\@L \mid \@T}}^B_{\omega_1} -> \-{\ang{\@L' \mid \@T'}}^B_{\omega_1}$ preserving finite limits and countable colimits.

\begin{corollary}
\label{thm:2interp-equiv}
There is a contravariant equivalence between the 2-categories
\begin{itemize}
\item  $\&{\omega_1\omega Thy}_{\omega_1}$, of countable $\@L_{\omega_1\omega}$-theories (over varying languages), interpretations between them, and definable isomorphisms between interpretations, and
\item  $\&{ONAPolGpd}$, of standard Borel groupoids admitting a compatible open non-Archimedean Polish groupoid topology, Borel functors between them, and Borel natural isomorphisms,
\end{itemize}
taking a theory $(\@L, \@T) \in \&{\omega_1\omega Thy}_{\omega_1}$ to its groupoid $\!{Mod}(\@L, \@T)$ of countable models (on initial segments of $\#N$) and an interpretation $(\@L, \@T) -> (\@L', \@T')$ to the induced Borel functor $\!{Mod}(\@L', \@T') -> \!{Mod}(\@L, \@T)$.
\end{corollary}
\begin{proof}
That the 2-functor is well-defined and fully faithful (i.e., restricts to an equivalence on each hom-category) is by \cref{rmk:isogpd-nonarchopen} and \cite[10.4 (11.4 in arXiv version)]{Cscc}.  That the 2-functor is essentially surjective (i.e., every object in the codomain is equivalent to an object in the image) is by \cref{thm:nonarchgpd-rep}.
\end{proof}

\section{Grey sets and topometric spaces}
\label{sec:greytopmet}

In this section, we review the basic theory of grey sets \cite{BYM} and topometric spaces \cite{BY08,BY10,BBM}.  Our treatment will be somewhat more systematic and general than in previous literature; notably, we will define a notion of possibly non-Hausdorff topometric spaces.

\subsection{Grey sets}
\label{sec:grey}

Let $\#I := [0, 1]$.  As usual in metric contexts, it is convenient to think of elements of $\#I$ as generalized ``truth values'', with $0 =$ absolute truth and $1 =$ absolute falsehood.  Thus, let $\sqle$ be the opposite of the usual ordering $\le$ on $\#I$:
\begin{align*}
r \sqle s  \coloniff  r \ge s.
\end{align*}
Let $\sqcup, \sqcap$ denote join and meet respectively in the $\sqle$ ordering, i.e., $\sqcup := \wedge$ and $\sqcap := \vee$.  Let
\begin{align*}
r \dotplus s := 1 \wedge (r + s), &&
r \dotminus s := 0 \vee (r - s)
\end{align*}
denote truncated addition and subtraction on $\#I$.

Let $X$ be a set.  A \defn{grey subset} $A \sqle X$ of $X$ is a function $A : X -> \#I$, thought of as a generalized ``indicator function''.  For an ordinary subset $A \subseteq X$, its \defn{zero-indicator} is the grey subset
\begin{align*}
\*0_A : X &--> \#I \\
x &|--> (0 \text{ if $x \in A$, else } 1).
\end{align*}
We often write $A$ for $\*0_A$ when no confusion can arise.
For a grey subset $A \sqle X$ and $r \in \#I$, define
\begin{align*}
A_{<r} := \{x \in X \mid A(x) < r\},
\end{align*}
and similarly for $A_{\ge r}, A_{=r}$, etc.

For an operation on $\#I$, we use the same symbol to denote the pointwise operation on grey subsets.  For example, we have the \defn{union} $A \sqcup B$, \defn{intersection} $A \sqcap B$, \defn{(truncated) sum} $A \dotplus B$, \defn{multiples} $r \cdot A$, and \defn{shifts} $A \dotplus r, A \dotminus r$ for $A, B \sqle X$ and $r \in \#I$.  Note that typically, the natural grey analog of intersection $\cap$ is $\dotplus$ rather than $\sqcap$.

Note that the poset embedding $(A |-> \*0_A) : (2^X, \subseteq) -> (\#I^X, \sqle)$ preserves arbitrary meets and joins.  Note also that
\begin{align*}
A_{<r} = \*0_{[0, r)} \circ A, &&
A_{\le r} = \*0_{[0, r]} \circ A
\end{align*}
(where we are writing $A_{<r}$ for $\*0_{A_{<r}}$); since $\*0_{[0, r)} : (\#I, \sqle) -> (\#I, \sqle)$ (resp., $\*0_{[0, r]}$) preserves arbitrary joins and finite meets (resp., finite joins and arbitrary meets), so does $(A |-> A_{<r}) : (\#I^X, \sqle) -> (2^X, \subseteq)$ (resp., $A |-> A_{\le r}$).  Finally, note that
\begin{align*}
(A \dotplus B)_{<r} = \bigcup_{r \ge s + t} (A_{<s} \cap B_{<t}), &&
(A \dotplus B)_{\le r} = \bigcup_{r \ge s + t} (A_{\le s} \cap B_{\le t});
\end{align*}
in the former union it is enough to consider $s, t$ rational (thus making the union countable).

Let $f : X -> Y$ be a function.  For $B \sqle Y$, its \defn{preimage} under $f$ is
\begin{align*}
f^{-1}(B) := B \circ f : X &--> \#I.
\end{align*}
Clearly, $f^{-1}(-)$ preserves pointwise operations like $\sqcup, \dotplus, (-)_{<r}$.  For $A \sqle X$, its \defn{image} under $f$ is
\begin{align*}
f(A) : Y &--> \#I \\
y &|--> \bigsqcup_{x \in f^{-1}(y)} A(x) = \bigwedge_{x \in f^{-1}(y)} A(x).
\end{align*}
We have the usual relations like $f(A) \sqle B \iff A \sqle f^{-1}(B)$, $f(-)$ preserves unions, etc.  Note also
\begin{align*}
f(A)_{<r} &= \{y \in Y \mid \exists x \in f^{-1}(y)\, (A(x) < r)\} = f(A_{<r}).
\end{align*}
For sets $X \subseteq Y$ and a grey subset $A \sqle X$, we identify $A$ with the grey subset of $Y$ given by its image under the inclusion $X `-> Y$; that is, $A(y) := 1$ for $y \in Y \setminus X$.

\subsection{Grey relations}
\label{sec:greyrel}

Let $X, Y$ be sets.  Given $A \sqle X$, define
\begin{align*}
A \times Y := p^{-1}(A) = A \circ p \sqle X \times Y
\end{align*}
where $p : X \times Y -> X$ is the projection.  Similarly, given $B \sqle Y$, define $X \times B \sqle X \times Y$.  For $A \sqle X$ and $B \sqle Y$, put
\begin{align*}
A \oplus B := (A \times Y) \dotplus (X \times B) \sqle X \times Y.
\end{align*}
Note that $A \times Y = A \oplus \*0_Y$; we write $\oplus$ instead of $\times$ in the general case to avoid confusion with any notion of product of $A, B$ as functions to $\#I$.

Let $R \sqle X \times Y$ be a \defn{grey binary relation}.  For $A \sqle X$, its \defn{relational image} under $R$ is
\begin{align*}
R[A] := q(R \dotplus (A \times Y)) \sqle Y
\end{align*}
where $q : X \times Y -> Y$ is the second projection.  Thus
\begin{align*}
R[A](y) = \bigwedge_{x \in X} (R(x, y) \dotplus A(x)).
\end{align*}
Note that for $A \subseteq X$, we have
\begin{align*}
R[A](y) &= \bigwedge_{x \in A} R(x, y), \\
R[A]_{<r} &= R_{<r}[A]
\end{align*}
(where $R[A]$ means $R[\*0_A]$).  For a general grey $A \sqle X$, we have
\begin{align*}
R[A]_{<r} &= \bigcup_{r \ge s+t} R_{<s}[A_{<t}].
\end{align*}

The \defn{inverse} of $R \sqle X \times Y$ is
\begin{align*}
R^\ominus := (\text{swap} : X \times Y --->[{(x, y) |-> (y, x)}]{\sim} Y \times X)(R) \sqle Y \times X.
\end{align*}
For another grey relation $S \sqle Y \times Z$, the \defn{composition} of $S$ with $R$ is
\begin{align*}
S \ocirc R := m((R \times Z) \dotplus (X \times S)) \sqle X \times Z
\end{align*}
where $m : X \times Y \times Z -> X \times Z$ is the projection.  (The symbols $\ominus, \ocirc$ are to avoid confusion with the usual inverse and composition of $R$ as a function $X \times Y -> \#I$.)  Thus
\begin{align*}
R^\ominus(y, x) &= R(x, y), \\
(S \ocirc R)(x, z) &= \bigwedge_{y \in Y} (R(x, y) \dotplus S(y, z)).
\end{align*}
Note that (the zero-indicator of) the diagonal $\Delta_X \subseteq X \times X$ is the identity for $\ocirc$.

\subsection{Metric spaces}
\label{sec:metsp}

By a \defn{(pseudo)metric}, we always mean of diameter $\le 1$.
A \defn{grey equivalence relation} $d \sqle X \times X$, i.e., a grey relation such that $\Delta_X \sqle d$, $d^\ominus \sqle d$, and $d \ocirc d \sqle d$, is the same thing as a pseudometric.  Given a pseudometric space $X = (X, d)$, we denote the quotient metric space by
\begin{align*}
X/d := (X/d_{=0}, d/d_{=0}).
\end{align*}
For a (grey) subset $A \sqle X$, its \defn{($d$-)saturation}\footnote{In \cite{BYM} this is called the \defn{$d$-thickening} of $A$.} is
\begin{align*}
[A] := d[A] \sqle X;
\end{align*}
for an ordinary subset $A$, this is simply
\begin{align*}
[A] = d(A, -) : X -> \#I.
\end{align*}
For $x \in X$, we write $[x] := [\{x\}]$; note that $[x]_{<r}$ is the open $r$-ball around $x$.  We say that $A \sqle X$ is \defn{($d$-)invariant} if $A = [A]$.

Even for a metric space $X = (X, d)$, it is often helpful to think of $X$ as analogous to a set $X$ equipped with a (nontrivial) equivalence relation $E$ representing the quotient space $X/E$.  In such a situation, when dealing with e.g., functions $Y -> X$ which represent functions between quotient spaces $Y/F -> X/E$, one would need to take saturations of sets at various places.  For example, the below definition of ``topometrically open'' is analogous to that of a map $Y -> X$ between topological spaces such that the composite map $Y -> X ->> X/E$ is open.

By a \defn{Lipschitz map} $f : X -> Y$ between (pseudo)metric spaces, we always mean a $1$-Lipschitz map, i.e., $f(d_X) \sqle d_Y$ (here $f$ refers to the induced product map $X^2 -> Y^2$).  In other words, $f$ is a homomorphism of grey equivalence relations.

On a product of metric spaces $X \times Y$, we always take the \defn{(truncated) sum metric}
\begin{align*}
d_{X \times Y}((x, y), (x', y')) := d_X(x, x') \dotplus d_Y(y, y');
\end{align*}
this ensures that a metric $d_X \sqle X \times X$ is itself an invariant grey set.  Note that under the canonical bijection $(X \times Y)^2 \cong X^2 \times Y^2$ swapping the middle two factors, we have
\begin{align*}
d_{X \times Y} \cong d_X \oplus d_Y.
\end{align*}

\subsection{Grey topology}
\label{sec:greytop}

Let $X$ be a topological space.  We say that $A \sqle X$ is \defn{open} if each $A_{<r}$ is open (i.e., $A : X -> \#I$ is upper semicontinuous), and \defn{closed} if each $A_{\le r}$ is closed, or equivalently $1-A \sqle X$ is open.

For a continuous map $f : X -> Y$ and open $B \sqle Y$, $f^{-1}(B) \sqle X$ is open (since $f^{-1}(B)_{<r} = f^{-1}(B_{<r})$).  If moreover $f$ is open, then for open $A \sqle X$, so is $f(A) \sqle Y$ (since $f(A)_{<r} = f(A_{<r})$).

We now consider the interaction between a topology and a (pseudo)metric on a set $X$, in particular the metric analog of quotient topologies.  Following the usual convention \cite{BY08}, in the presence of both a topology and a (\emph{a priori} unrelated) (pseudo)metric, we use general topological terminology (e.g., ``open'') to refer by default to the topology, and metric terminology (e.g., ``ball'') to refer by default to the metric.

\begin{proposition}
\label{thm:topomet-strong}
Let $(X_i)_{i \in I}$ be a family of topological spaces, $Y$ be a pseudometric space, and $f_i : X_i -> Y$ be a function for each $i \in I$.  Then
\begin{align*}
\@T &:= \{U_{<1} \mid U \sqle Y \text{ invariant} \AND \forall i \in I\, (f_i^{-1}(U) \sqle X_i \text{ open})\} \\
&= \{U_{<r} \mid U \sqle Y \text{ invariant} \AND \forall i \in I\, (f_i^{-1}(U) \sqle X_i \text{ open}) \AND r \in \#R\},
\end{align*}
and $\@T$ forms a topology on $Y$, such that an invariant grey set $U \sqle Y$ is $\@T$-open iff $f_i^{-1}(U) \sqle X_i$ is open for all $i$.
\end{proposition}
In this situation, we call $\@T$ the \defn{topometric strong topology} on $Y$ induced by the maps $f_i$.
\begin{proof}
That the two definitions of $\@T$ agree follows from the formula $U_{<r} = (U \dotplus (1-r))_{<1}$ for $r \in \#I$, and the fact that the operation $U |-> U \dotplus (1-r)$ preserves invariance and is preserved by taking $f_i$-preimage (hence preserves $f_i^{-1}(U) \sqle X_i$ being open).

Similarly, that $\@T$ forms a topology on $Y$ follows from $U_{<1} \cap V_{<1} = (U \sqcap V)_{<1}$ and $\bigcup_i (U_i)_{<1} = (\bigsqcup_i U_i)_{<1}$ (and that these operations preserve invariance and are preserved by taking $f_i$-preimage).

Let $U \sqle Y$ be invariant.  If $f_i^{-1}(U) \sqle X_i$ is open for all $i$, then $U$ is $\@T$-open by the second definition of $\@T$.  Conversely, if $U$ is $\@T$-open, then for each $r$, we have $U_{<r} \in \@T$, whence $U_{<r} = V_{<1}$ for some invariant $V \sqle Y$ such that $f_i^{-1}(V) \sqle X_i$ is open for all $i$, whence for all $i$, we have $f_i^{-1}(U)_{<r} = f_i^{-1}(U_{<r}) = f_i^{-1}(V_{<1}) = f_i^{-1}(V)_{<1}$ which is open in $X_i$, i.e., $f_i^{-1}(U) \sqle X_i$ is open.
\end{proof}

For a pseudometric space $X = (X, d)$ equipped with a (\emph{a priori} unrelated) topology $\@T$, we say that $\@T$ is \defn{($d$-)invariant}, or that $X = (X, \@T, d)$ is a \defn{pseudotopometric space}, if $\@T$ is the topometric strong topology induced by the identity map $1_X : (X, \@T) -> (X, d)$.  In other words, every open set $U \subseteq X$ must be $V_{<1}$ for some invariant open grey $V \sqle X$.  If $d$ is a metric, we call $X = (X, \@T, d)$ a \defn{topometric space}.\footnote{%
Strictly speaking, our definition of \defn{topometric space} is incomparable with those in the literature \cite{BY08,BY10,BBM}.  Unlike those definitions, we do not require the \defn{Hausdorff} condition that $d \sqle X^2$ be closed.  We do, however, additionally require the topology to be determined by the invariant open grey sets.  Note that this latter condition automatically holds for \defn{completely regular topometric spaces} in the sense of \cite{BY10}, including the important cases of compact topometric spaces as well as automorphism groups of metric structures.
}

Note that in a pseudotopometric space, the topology is refined by that induced by the pseudometric, since for invariant open $U \sqle X$, we have $U_{<1} = [U]_{<1} = \bigcup_{1 \ge r+s} [U_{<r}]_{<s}$ where each $[U_{<r}]_{<s}$ is a union of metrically open balls of radius $s$.

A \defn{maximal topometric space} is a topometric space $X$ whose metric is discrete, i.e., $d = \Delta_X$.  Maximal topometric spaces can be identified with their underlying topological spaces.  A \defn{minimal topometric space} is a topometric space $X$ whose metric induces its topology, or equivalently $d \sqle X^2$ is open (in analogy with discrete topological spaces, whose diagonals are open).

Continuous maps to a pseudotopometric space can be characterized in terms of invariant open grey sets:

\begin{proposition}
\label{thm:topomet-cts}
A map $f : X -> Y$ from a topological space $X$ to a pseudotopometric space $Y$ is continuous iff for every invariant open $U \sqle Y$, $f^{-1}(U) \sqle X$ is open.
\end{proposition}
\begin{proof}
Follows from $f^{-1}(U)_{<r} = f^{-1}(U_{<r})$ for $U \sqle Y$.
\end{proof}

Clearly, any topometric strong topology is invariant.  We have the analog of the usual universal property of the strong topology:

\begin{proposition}
\label{thm:topomet-strong-univ}
In the situation of \cref{thm:topomet-strong}, a Lipschitz map $g : Y -> Z$ to some other pseudotopometric space $Z$ is $\@T$-continuous iff $g \circ f_i : X_i -> Z$ is continuous for each $i \in I$.
\end{proposition}
\begin{proof}
Since $g$ is Lipschitz, $g^{-1}(V) \sqle Y$ is invariant for invariant $V \sqle Z$.  Now apply \cref{thm:topomet-strong,thm:topomet-cts}.
\end{proof}

We record the following simple fact, which we will use freely:

\begin{lemma}
\label{lm:topomet-reg}
Let $X$ be a pseudotopometric space.  For any open $U \subseteq X$ and $x \in U$, there is an open $V \ni x$ and $r > 0$ with $[V]_{<r} \subseteq U$.
\end{lemma}
\begin{proof}
Let $U = U'_{<1}$ for invariant open $U' \sqle X$.  Then $U = U'_{<1} = \bigcup_{s < 1} U'_{<s}$, so $x \in U'_{<s}$ for some $s < 1$; and $[U'_{<s}]_{<1-s} \subseteq \bigcup_{1 \ge s+t} [U'_{<s}]_{<t} = [U']_{<1} = U'_{<1}$.  Put $V := U'_{<s}$ and $r := 1-s$.
\end{proof}

Consider now a topological space $X$ equipped with a pseudometric $d \sqle X^2$, its metric quotient $X/d$ with the quotient map $\pi : X ->> X/d$, as well as the completion $X `-> \^{X/d}$; we continue to call the composition $\pi : X ->> X/d `-> \^{X/d}$ the \defn{quotient map}.  We call the topometric strong topology on $X/d$ (resp., $\^{X/d}$) induced by the original topology on $X$ the \defn{(complete) topometric quotient topology}.  Open sets in $X/d$ are quotients of $U_{<1} \subseteq X$ for $d$-invariant open $U \sqle X$.

More generally, for any isometry $f : X -> Y$ with dense image between pseudometric spaces with a topology on $X$, by abuse of terminology, we call the topometric strong topology on $Y$ induced by $f$ the \defn{topometric quotient topology} on $Y$.  For such a topology, we have a bijection
\begin{align*}
\{\text{invariant open } U \sqle X\} &\cong \{\text{invariant open } V \sqle Y\} \\
f &|-> d_Y[f(U)] \\
f^{-1}(V) &<-| V.
\end{align*}

As in general topology, the case of open quotient maps is of particular interest.  A map $f : X -> Y$ from a topological space $X$ to a topometric space $Y$ is \defn{topometrically open} if for any open $U \subseteq X$, $[f(U)] \sqle Y$ is open.
This implies that the same holds for any open $U \sqle X$, since $[f(U)]_{<r} = \bigcup_{r \ge s+t} [f(U)_{<s}]_{<t} = \bigcup_{r \ge s+t} [f(U_{<s})]_{<t}$.

\begin{lemma}
\label{thm:greyrelopen}
Let $X, Y$ be topological spaces and $R \sqle X \times Y$.  The following are equivalent:
\begin{enumerate}[label=(\roman*)]
\item  For any open $U \subseteq X$ and $r > 0$, $R_{<r}[U] \subseteq Y$ is open (i.e., $R_{<r}$ is relationally open in the sense of \cref{thm:relopen}).
\item  For any open $U \sqle X$, $R[U] \sqle Y$ is open.
\item  For any open $W \sqle X \times Y$, $q(R \dotplus W) \sqle Y$ is open, where $q : X \times Y -> Y$ is the second projection.
\end{enumerate}
\end{lemma}
If these hold, we say that $R$ is \defn{relationally open}.  Note that this definition is compatible with \cref{thm:relopen} when $R \subseteq X \times Y$: in that case, $R_{<r} = R$ (i.e., $(\*0_R)_{<r} = R$) for any $0 < r \le 1$.
\begin{proof}
For $U \subseteq X$, $R[U]_{<r} = R_{<r}[U]$, so (i) says that $R[U] \sqle Y$ is open for all open $U \subseteq X$; clearly this is a special case of (ii).
Similarly, for $W \subseteq X \times Y$, $q(R \dotplus W)_{<r} = q((R \dotplus W)_{<r}) = q(\bigcup_{r \ge s + t} (R_{<s} \cap W_{<t})) = \bigcup_{r \ge s + t} q(R_{<s} \cap W_{<t})$, so (i)$\implies$(iii).  Finally, taking $W := U \times Y$ in (iii) gives (ii).
\end{proof}

When $R$ in \cref{thm:greyrelopen} is a pseudometric $d \sqle X^2$, we have $d[U] = \pi^{-1}([\pi(U)])$ where $\pi : X ->> X/d$ is the quotient map; thus $d$ is relationally open iff $\pi : X ->> X/d$ is topometrically open.
When $d$ is already a metric, we get that $d \sqle X^2$ is relationally open iff the identity $1_X : X -> (X, d)$ is topometrically open.

Following \cite{BYM}, we call a topometric space $X$ \defn{adequate} if $1_X : X -> (X, d)$ is topometrically open, i.e., $d \sqle X^2$ is relationally open, i.e., for every open $U \subseteq X$, $[U] \sqle X$ is open.
This condition may be combined with invariance of the topology as follows: an \defn{adequate topometric space} is, equivalently, a topological space $X$ equipped with a metric such that
\begin{align*}
\{[U]_{<r} \mid U \subseteq X \text{ open} \AND r > 0\}
\end{align*}
forms a basis for the topology (it is enough to take basic open $U$ and rational $r$).  For a topological space $X$ and relationally open pseudometric $d \sqle X^2$, $X/d$ is an adequate topometric space.

Adequacy also interacts well with completion.  For an adequate topometric space $X = (X, d)$, the embedding $X `-> \^X = (\^X, \^d)$ is topometrically open, since for open $U \subseteq X$, we have $\^d[U]|X = d[U] \sqle X$ (by metric density of $X \subseteq \^X$) which is open by adequacy of $X$, whence $\^d[U] \sqle \^X$ is open.  This easily implies that $\^X$ is adequate, with a basis of open sets
\begin{align*}
\{\^d[U]_{<r} \mid U \subseteq X \text{ open} \AND r > 0\}.
\end{align*}
Hence for a topological space $X$ and relationally open pseudometric $d \sqle X^2$, $\^{X/d}$ is an adequate complete topometric space, with a basis of open sets
\begin{align*}
\{[\pi(U)]_{<r} \mid U \subseteq X \text{ open} \AND r > 0\}
\end{align*}
where $\pi : X -> \^{X/d}$ is the quotient map.

Finally, we note that for adequate topometric spaces, topometric quotient respects products:

\begin{lemma}
\label{thm:topomet-product}
Let $X, Y$ be topological spaces equipped with relationally open pseudometrics $d_X \sqle X^2$ and $d_Y \sqle Y^2$.  Then the sum pseudometric $d_{X \times Y}$ on $X \times Y$ is relationally open, and the canonical isometry
\begin{align*}
(X \times Y)/d_{X \times Y} &\cong X/d_X \times Y/d_Y
\end{align*}
(with the sum metric on both sides)
is a homeomorphism between the topometric quotient topology on the left and the product of the topometric quotient topologies on the right.

Thus for adequate topometric spaces $X, Y$, $X \times Y$ (with the sum metric and the product topology) is an adequate topometric space; furthermore, the canonical isometry
\begin{align*}
\^{X \times Y} \cong \^X \times \^Y
\end{align*}
is a homeomorphism between the induced topologies.
\end{lemma}
\begin{proof}
We have $(d_{X \times Y})_{<r} \cong \bigcup_{r \ge s+t} ((d_X)_{<s} \times (d_Y)_{<t})$ (identifying $(X \times Y)^2 \cong X^2 \times Y^2$); since in general, $R \subseteq X \times X'$ and $S \subseteq Y \times Y'$ relationally open implies that $R \times S \subseteq (X \times Y) \times (X' \times Y')$ (identifying $(X \times X') \times (Y \times Y') \cong (X \times Y) \times (X' \times Y')$) is relationally open (because for basic open $U \times V \subseteq X \times Y$, we have $(R \times S)[U \times V] = R[U] \times S[V]$), it follows that $d_{X \times Y}$ is relationally open.

For invariant grey open $U \sqle X$ and $V \sqle Y$, we have $U_{<1} \times V_{<1} = (U \otimes V)_{<1}$, where $U \otimes V := (U \times Y) \sqcap (X \times V)$ (so $(U \otimes V)(x, y) = U(x) \vee V(y)$) which is $d_{X \times Y}$-invariant; thus the map $(X \times Y)/d_{X \times Y} -> X/d_X \times Y/d_Y$ is continuous.  Conversely, for a basic open rectangle $U \times V \subseteq X \times Y$, so that $d_{X \times Y}[U \times V]_{<r}$ descends to a basic open set in $(X \times Y)/d_{X \times Y}$, from above, we have $d_{X \times Y}[U \times V]_{<r} = \bigcup_{r \ge s+t} (d_X[U]_{<s} \times d_Y[V]_{<t})$ which descends to an open set in $X/d_X \times Y/d_Y$.

For adequate topometric spaces $X, Y$, the canonical isometry $X \times Y \cong (X \times Y)/d_{X \times Y}$ factors into homeomorphisms $X \times Y \cong X/d_X \times Y/d_Y \cong (X \times Y)_{d_{X \times Y}}$ where the first homeomorphism is by invariance of the topologies of $X, Y$ and the second homeomorphism is by the above; thus $X \times Y$ is already adequate topometric.  Similar calculations as above yield the homeomorphism $\^{X \times Y} \cong \^X \times \^Y$.
\end{proof}

\subsection{Grey subsets of groupoids}
\label{sec:greygpd}

Let $G$ be a groupoid.  For $A, B \sqle G$, define\footnote{%
In \cite{BYM} $A \odot B$ is denoted $A * B$.}
\begin{align*}
A^\ominus &:= \nu(A) \sqle G, \\
A \odot B &:= \mu(A \oplus_{G^0} B)
\end{align*}
where $\nu, \mu$ are the inverse and multiplication maps of $G$ and $A \oplus_{G^0} B$ is the restriction of $A \oplus B \sqle G \times G$ to $G \times_{G^0} G \subseteq G \times G$.  Thus
\begin{align*}
A^\ominus(g) &= A(g^{-1}), \\
(A \odot B)(g) &= \bigwedge_{g = h \cdot k} (A(h) \dotplus B(k)).
\end{align*}
For a topological groupoid $G$ and open $A \sqle G$, $A^\ominus$ is open; and if $G$ is an open groupoid and $A, B \sqle G$ are open, then $A \odot B \sqle G$ is open.

A grey subset $A \sqle G$ is \defn{symmetric} if $A = A^\ominus$, \defn{strictly unital} if $A|G^0$ is $\{0, 1\}$-valued (i.e., $A|G^0$ is (the zero-indicator of) a genuine subset of $G^0$) and $\sigma(A), \tau(A) \sqle A$ (i.e.,
\begin{align*}
A(g) < 1 \implies A(\sigma(g)) = A(\tau(g)) = 0
\end{align*}
for all $g \in G$), and a \defn{strict grey subgroupoid} if $A$ is symmetric, strictly unital, and satisfies $A \odot A \sqle A$, i.e.,
\begin{align*}
A(g \cdot h) &\le A(g) \dotplus A(h)
\end{align*}
for all composable $g, h \in G$.  For $G$ a group, such $A$ is usually called a \defn{norm} or \defn{length function}, and corresponds to a left- (or right-) invariant metric on $G$; see \cref{sec:locpolgpd-metriz} for the analogous fact for groupoids.  Every strictly unital $A \sqle G$ generates a smallest strict grey subgroupoid, namely
\begin{align*}
\ang{A} := \bigsqcup_{n \ge 1} (A \sqcup A^\ominus)^{\odot n}
\end{align*}
(where $B^{\odot n} := B \odot B \odot \dotsb \odot B$).  If $G$ is an open topological groupoid and $A \sqle G$ is open, then so is $\ang{A} \sqle G$.

\begin{remark}
We do not know of a good notion of ``non-strict grey subgroupoid'' $A \sqle G$, i.e., where $A$ is allowed to be grey on unit morphisms.  Such a notion would perhaps be needed in order to generalize \cref{thm:intro-locpolgpd-rep} to non-$\sigma$-locally Polish $G$.
\end{remark}

\section{Continuous infinitary logic}
\label{sec:ctslog}

In this section, we review some conventions regarding continuous infinitary logic.

For general background on (finitary) continuous logic, see \cite{BBHU}.  For background on continuous $\@L_{\omega_1\omega}$, see \cite{BDNT} or \cite{CL}; we will mostly be following conventions of the latter paper.  We will only be concerned with metric structures of \emph{bounded diameter $\le 1$} over \emph{relational} languages, all of whose relations are \emph{(1-)Lipschitz}; according to the usual conventions (see the aforementioned references), this means each relation symbol has associated modulus of continuity given by the identity function.%
\footnote{As before, for $n$-ary relations we consider the sum (not max) metric on the product space.}
Hence, our definition of \defn{(multi-sorted) relational first-order language $\@L$} for continuous logic is the same as in the discrete case (\cref{sec:disclog}).

Given a (multi-sorted relational) language $\@L$, an \defn{$\@L$-structure} $\@M$ consists of:
\begin{itemize}
\item  for each sort $P \in \@L_s$, an underlying complete metric space $P^\@M$ (of diameter $\le 1$; denoted $M$ when $\@L$ is one-sorted);
\item  for each $(P_0, \dotsc, P_{n-1})$-ary relation symbol $R \in \@L_r$, an invariant grey subset $R^\@M \sqle P_0^\@M \times \dotsb \times P_{n-1}^\@M$ (i.e., a 1-Lipschitz map $R^\@M : P_0^\@M \times \dotsb \times P_{n-1}^\@M -> \#I$).
\end{itemize}

Given a single-sorted relational language $\@L$-structure $\@M$ and an isometric bijection $f : M \cong N$ with another complete metric space $N$, we have the \defn{pushforward structure} $f_*(\@M) := (N, f(R^\@M))_{R \in \@L_r}$, exactly as in the discrete case.  For $\@L$ countable and for any fixed Polish metric space $X$, we have the \defn{Polish space of $\@L$-structures on $X$}
\begin{align*}
\Mod_X(\@L) &\cong \prod_{n \in \#N; R \in \@L_r(n)} \Hom(X^n, \#I) \\
\@M = (X, R^\@M)_{R \in \@L_r} &|-> (R^\@M)_{R \in \@L_R}
\end{align*}
where here $\Hom(X^n, \#I)$ denotes the space of Lipschitz maps $X^n -> \#I$ with the pointwise convergence topology.  Pushforward of structures gives the \defn{logic action} $\Iso(X) \curvearrowright \Mod_X(\@L)$ of the isometry group $\Iso(X)$ of $X$ (with the pointwise convergence topology) on $\Mod_X(\@L)$, a continuous action of a Polish group.  The action groupoid
\begin{align*}
\Iso(X) \ltimes \Mod_X(\@L)
\end{align*}
is the (open) \defn{Polish groupoid of $\@L$-structures on $X$}.  We will only be concerned with the case where $X = \#U$ is the \defn{Urysohn sphere} (see \cref{sec:urysohn} below for the definition).

For the definition of $\@L_{\omega_1\omega}$ in the metric setting, see \cite{CL}.  Briefly, formulas are formed from atomic formulas of the form $R(\vec{x})$ (for $R \in \@L_r$) and $d(x, y)$ ($d$ is a symbol denoting the metric) from continuous $n$-ary functions $\#I^n -> \#I$ (representing ``logical connectives''), sup and inf over variables (representing ``quantifiers''), and sup and inf over countable families of formulas which admit a single modulus of uniform continuity.  In particular, an \defn{$\@L_{\omega_1\omega}$-sentence} $\phi$ is such an expression without free variables, which evaluates to a number $\phi^\@M \in \#I$ for each $\@L$-structure $\@M$; we say that $\@M$ \defn{satisfies} $\phi$, denoted $\@M \models \phi$, if $\phi^\@M = 0$.

For an $\@L_{\omega_1\omega}$-sentence $\phi$ and Polish metric space $X$, it is easy to verify that
\begin{align*}
\Mod_X(\@L) &--> \#I \\
\@M &|--> \phi^\@M
\end{align*}
is a Borel $\Iso(X)$-invariant grey subset of $\Mod_X(\@L)$.  The converse when $X = \#U$ is the version of the Lopez-Escobar theorem for continuous logic proved by Coskey--Lupini \cite{CL}:

\begin{theorem}[Coskey--Lupini]
\label{thm:coskey-lupini}
For any $\Iso(\#U)$-invariant Borel $B \sqle \Mod_\#U(\@L)$, there is an $\@L_{\omega_1\omega}$-sentence $\phi$ such that $B(\@M) = \phi^\@M$ for all $\@M \in \Mod_\#U(\@L)$.
\end{theorem}

In particular, for any $\Iso(\#U)$-invariant Borel (genuine) subset $B \subseteq \Mod_\#U(\@L)$, there is a \defn{$\{0, 1\}$-valued} $\@L_{\omega_1\omega}$-sentence $\phi$, meaning $\phi^\@M \in \{0, 1\}$ for all $\@M$, such that
\begin{align*}
B = \Mod_\#U(\@L, \phi) := \{\@M \in \Mod_\#U(\@L) \mid \@M \models \phi\}.
\end{align*}
We will only need this special case of \cref{thm:coskey-lupini}.  For such $\phi$, the action groupoid
\begin{align*}
\Iso(\#U) \ltimes \Mod_\#U(\@L, \phi)
\end{align*}
is the \defn{standard Borel groupoid of models of $\phi$ on $\#U$}.

\section{Metric-étale spaces and structures}
\label{sec:metale}

In this section, we study the metric analogs of étale spaces and structures.%
\footnote{A \emph{metric-étale space} over $X$ corresponds to a complete metric \emph{locale} in the sheaf topos $\Sh(X)$; see \cite{Hen}.  The lack of enough continuous sections as explained in \cref{sec:metalesp} can be understood in terms of the fact that spatiality of complete metric locales (see \cite[X~2.2]{PP}) requires the axiom of dependent choice, which need not hold in $\Sh(X)$.}
Informally, a \emph{metric-étale space} $A$ over a topological space $X$ is a ``continuous'' assignment of a complete metric space $A_x$ to each $x \in X$; similarly, a \emph{metric-étale structure} is a ``continuous'' family of metric structures.  We define the \emph{isomorphism groupoid} of a metric-étale structure in analogy with the étale setting, and prove that (when the structure is quasi-Polish) it admits a full and faithful Borel functor to a groupoid of metric structures on the Urysohn sphere $\#U$ (\cref{thm:sbmstr-isogpd-ff}), using Katětov's construction of $\#U$.  This will form half of the proof of \cref{thm:intro-locpolgpd-rep}.

\subsection{Metric-étale spaces}
\label{sec:metalesp}

Let $X$ be a topological space.  A \defn{metric-étale space over $X$} is a complete topometric space $A$ equipped with a (topologically) open continuous projection map $p : A -> X$, such that the metric $d \sqle A \times A$ satisfies $d \sqle A \times_X A$ (meaning $d$ is discrete between points in different fibers), and $d \sqle A \times_X A$ is open (meaning that for all $r > 0$, the $r$-neighborhood of the diagonal
\begin{align*}
d_{<r} = \{(a, b) \in A \times_X A \mid d(a, b) < r\} \subseteq A \times_X A
\end{align*}
is open).  This last condition is analogous to the characterization of étale spaces via \cref{thm:etale-diagopen}, and can be thought of as uniformly witnessing that the subspace topology on each fiber $A_x$ is induced by the metric (since $d_{A_x} = d|A_x \sqle A_x \times A_x$ is open).

\begin{lemma}
\label{lm:metale-open}
Let $p : A -> X$ be an open topological space over a topological space $X$, and $R \sqle A \times_X A$ be a grey relation on $A$ which is discrete between fibers (i.e., $R(a, b) = 1$ for all $p(a) \ne p(b)$).  If $R \sqle A \times_X A$ is open, then $R \sqle A \times A$ is relationally open.
\end{lemma}
\begin{proof}
Since $p : A -> X$ is open, the second projection $q := A \times_X p : A \times_X A -> A$ is open, whence for any $0 < r < 1$, $q|R_{<r} : d_{<r} -> A$ is open, i.e., $R_{<r} \subseteq A \times A$ is relationally open; thus $R \sqle A \times A$ is relationally open.
\end{proof}

\begin{corollary}
\label{thm:metale-adequate}
A metric-étale space $A$ over $X$ is adequate.  \qed
\end{corollary}

In general, given a metric-étale space $p : A -> X$, it is unreasonable to expect there to be a plentiful supply of (partial) continuous sections $s : U -> A$ of $P$ defined on open $U \subseteq X$.  Thus, we must make do with approximate sections.  We call open $S \subseteq A$ \defn{$d_{<r}$-small} or \defn{$r$-small} over $X$ if $S \times_X S \subseteq d_{<r}$.  For each $r > 0$, since $d_{<r} \subseteq A \times_X A$ is open, $A$ has a cover by $r$-small open $S \subseteq A$.  It follows that for any basis $\@S$ of open sets in $A$,
\begin{align*}
\@S_{<r} := \{S \in \@S \mid S \text{ $r$-small}\}
\end{align*}
is still a basis.

Give metric-étale spaces $p : A -> X$ and $q : B -> X$, we always equip the fiber product $A \times_X B -> X$ with the (truncated) sum metric.

\begin{proposition}
\label{thm:metale-sum}
For metric-étale spaces $p : A -> X$ and $q : B -> X$ over $X$, $A \times_X B$ is metric-étale over $X$.
\end{proposition}
\begin{proof}
Since $p, q$ are topologically open continuous maps, so is $p \times_X q : A \times_X B -> X$.  By \cref{thm:metale-adequate,thm:topomet-product} (and the fact that all metrics involved are discrete between fibers), $A \times_X B$ is a topometric space.  It is standard that the sum of complete metrics is complete.  Since $d_{A \times_X B} \cong d_A \oplus_X d_B \sqle A^2_X \times_X B^2_X$ (identifying $(A \times_X B)^2_X \cong A^2_X \times_X B^2_X$), $d_{A \times_X B} \sqle (A \times_X B)^2_X$ is open.
\end{proof}

The following analog to \cref{thm:etale-quotient} describes the construction of metric-étale quotient spaces:

\begin{proposition}
\label{thm:metale-quotient}
Let $p : A -> X$ be an open topological space over $X$, and let $d \sqle A \times_X A$ be an open (in $A \times_X A$) pseudometric.  Then $p$ descends to metric-étale $p' : \^{A/d} -> X$, and the (complete) quotient map $\pi : A -> \^{A/d}$ is topometrically open.
\end{proposition}
\begin{proof}
By \cref{lm:metale-open}, $d \sqle A \times A$ is relationally open, whence $\^{A/d}$ is an adequate complete topometric space and $\pi$ is topometrically open.  The complete topometric quotient topology on $\^{A/d}$ has basic open sets $[\pi(S)]_{<r}$ for open $S \subseteq A$ and $r > 0$.  For open $U \subseteq X$, we have $p^{\prime-1}(U) = [\pi(p^{-1}(U))]_{<1}$; thus $p'$ is continuous.  For open $S \subseteq A$ and $0 < r \le 1$, we have $p'([\pi(S)]_{<r}) = p(S)$; thus $p'$ is open.  Finally, the completed metric $\^d \sqle \^{A/d} \times_X \^{A/d}$ is $[\pi(d)]$, where $\pi : A \times_X A -> \reallywidehat{(A \times_X A)/d_{A \times_X A}} \cong \^{A/d} \times_X \^{A/d}$ (where $d_{A \times_X A}$ is the sum metric of $d$ and the $\cong$ is by \cref{thm:topomet-product}) is topometrically open; thus $\^d$ is open.
\end{proof}

In the situation of \cref{thm:metale-quotient}, we call open $S \subseteq A$ \defn{$d_{<r}$-small over $X$} if $S \times_X S \subseteq d_{<r}$.  If $T \subseteq \^{A/d}$ is $\^d_{<r}$-small, then $\pi^{-1}(T) \subseteq A$ is $d_{<r}$-small.  If $S \subseteq A$ is $d_{<r}$-small open, then in general $\pi(S) \subseteq \^{A/d}$ might not be open; however, for any $s > 0$, $[\pi(S)]_{<s}$ will be $\^d_{<r+2s}$-small open, and sets of this latter form (for all $S$ belonging to any basis for $A$ and all $s > 0$) are a basis for $\^{A/d}$.

As in the étale case (\cref{thm:etale-ctble}), we will primarily be interested in metric-étale spaces over quasi-Polish spaces, which are furthermore ``uniformly fiberwise separable'':

\begin{proposition}
\label{thm:metale-ctble}
Let $p : A -> X$ be a metric-étale space over a second-countable space $X$.  The following are equivalent:
\begin{enumerate}
\item[(i)]  $A$ is second-countable;
\item[(ii)]  for any $r > 0$, $A$ has a countable cover by $r$-small open $S \subseteq A$;
\item[(iii)]  for any $r > 0$, $A$ has a countable basis consisting of $r$-small open $S \subseteq A$.
\end{enumerate}
\end{proposition}
\begin{proof}
(i)$\implies$(iii):  For a countable basis $\@S$, as remarked above, $\@S_{<r}$ will be a countable basis of $r$-small open sets.

(iii)$\implies$(ii):  obvious.

(ii)$\implies$(i):  Let $\@U$ be a countable basis for $X$.  For each $n \in \#N^+$, let $\@S_n$ be a countable cover of $A$ by $1/n$-small open sets.  We claim that
\begin{align*}
\{p^{-1}(U) \cap S \mid U \in \@U \AND S \in \bigcup_n \@S_n\}
\end{align*}
forms a countable basis for $A$.  Indeed, let $T \subseteq A$ be open and $a \in T$, let $T' \ni a$ be open and $n \in \#N^+$ with $[T']_{<1/n} \subseteq T$, and let $a \in S \in \@S_n$.  Then $a \in S \cap T'$, whence $p(a) \in U \subseteq p(S \cap T')$ for some $U \in \@U$.  Then $a \in p^{-1}(U) \cap S \subseteq T$, since for any $b \in p^{-1}(U) \cap S$, we have $p(b) \in U \cap p(S) \subseteq p(S \cap T')$, whence there is $c \in (S \cap T')_{p(b)}$, whence $d(b, c) < 1/n$ because $b, c \in S \in \@S_n$, whence $b \in [T']_{<1/n} \subseteq T$.
\end{proof}


\begin{proposition}
\label{thm:metale-qpol}
Let $p : A -> X$ be a second-countable metric-étale space over a quasi-Polish space $X$.  Then $A$ is quasi-Polish.
\end{proposition}
\begin{proof}
The proof is essentially a ``uniform'' version of the proof that Polish spaces are quasi-Polish; see \cite[5.1, Proof~2]{Cqpol}.
Let $\@U, \@S$ be countable bases of open sets in $X, A$ respectively, with $\@S$ closed under binary intersections and $p^{-1}(U) \in \@S$ for all $U \in \@U$.  Let
\begin{align*}
e : A &--> X \times \#S^\@S \\
a &|--> (p(a), \{S \in \@S \mid a \in S\})
\end{align*}
(where we are identifying $\#S^\@S$ with the powerset of $\@S$).  Clearly $e$ is a continuous embedding.  Hence, it suffices to show that for $(x, \@A) \in X \times \#S^\@S$,
\begin{align*}
(x, \@A) \in \im(e) \iff \left(\begin{aligned}
&\forall U \in \@U\, (x \in U \iff p^{-1}(U) \in \@A) \AND \\
&\forall S, T \in \@S\, (S \cap T \in \@A \iff S, T \in \@A) \AND \\
&\forall n \in \#N^+\, \exists S \in \@S_{<1/n}\, (S \in \@A) \AND \\
&\forall S \in \@A\, \exists n \in \#N^+,\, T \in \@A\, ([T]_{<1/n} \subseteq S)
\end{aligned}\right).
\end{align*}
$\Longrightarrow$ is straightforward (for the fourth condition, use that $S$ is a union of $[S']_{<r}$ for open $S' \subseteq A$ and $r > 0$).  For $\Longleftarrow$, let $(x, \@A)$ obey the right-hand side; we must find $a \in A_x$ such that $a \in S \iff S \in \@A$ for all $S \in \@S$.  By the first three conditions on the right-hand side,
$
\@A_x := \{S_x \mid S \in \@S\}
$
is a Cauchy filter base in the complete metric space $A_x$.  Let $a$ be its limit, i.e.,
\begin{align*}
\{a\} = \bigcap_{S \in \@A} \-{S_x}^{A_x}.
\end{align*}
For $S \in \@A$, by the fourth condition on the right-hand side, there are $n \in \#N^+$ and $T \in \@A$ with $[T]_{<1/n} \subseteq S$, whence $a \in \-{T_x} \subseteq [T]_{<1/n} \subseteq S$.  Conversely, for $a \in T \in \@S$, letting $T' \ni a$ be open and $n \in \#N^+$ with $[T']_{<2/n} \subseteq T$, by the third condition, there is some $1/n$-small $S \in \@A$, whence $a \in \-{S_x}$.  Then there is some $a' \in S_x$ with $d(a, a') < 1/n$, whence $a' \in S \cap [T']_{<1/n}$, whence $x = p(a') \in U \subseteq p(S \cap [T']_{<1/n})$ for some $U \in \@U$.  Then $p^{-1}(U) \cap S \in \@A$ by the first and second conditions, and $p^{-1}(U) \cap S \subseteq T$ as in the proof of \cref{thm:metale-ctble}, whence $T \in \@A$ by the second condition.
\end{proof}

\subsection{Isometry groupoids}
\label{sec:istgpd}

Let $p : A -> X$ be a metric-étale space over a topological space $X$.  Define
\begin{align*}
\Hom_X(A) &:= \{(x, y, f) \mid x, y \in X \AND f : A_x -> A_y \text{ (1-)Lipschitz}\}, \\
\Iso_X(A) &:= \{(x, y, f) \in \Hom_X(A) \mid f \text{ an isometry } A_x \cong A_y\}.
\end{align*}
$\Hom_X(A)$ is naturally equipped with the structure of a category, the \defn{Lipschitz category} of $A$, with objects $x \in X$ and morphisms consisting of Lipschitz maps $f : A_x -> A_y$; the category structure maps $\sigma, \tau, \iota, \nu, \mu$ are defined exactly as in the étale case (\cref{sec:symgpd}).  As before, we usually refer to morphisms in $\Hom_X(A)$ as $f : A_x -> A_y$, or simply $f$, instead of $(x, y, f)$.  The core $\Iso_X(A) \subseteq \Hom_X(A)$ is the \defn{isometry groupoid} of $A$.

Equip $\Hom_X(A)$ with the topology generated by $\sigma, \tau : \Hom_X(A) -> X$ as well as the sets $\sqsqbr{S |-> T}$ defined as in \cref{sec:symgpd} for open $S, T \subseteq A$.  As before, it suffices to consider $S, T$ in some basis of open sets in $A$, although we cannot expect them to be sections.  It is easily seen that when $X = 1$, the topology on $\Hom_X(A) = \{\text{Lipschitz } f : A -> A\}$ is just the pointwise convergence topology.  Equip $\Iso_X(A) \subseteq \Hom_X(A)$ with the subspace topology.

\begin{proposition}
\label{thm:istgpd-cts}
$\Hom_X(A)$ (resp., $\Iso_X(A)$) is a topological category (resp., groupoid).
\end{proposition}
\begin{proof}
$\sigma, \tau, \iota, \nu$ are continuous exactly as in \cref{thm:symgpd-cts}.

Continuity of $\mu$: as in \cref{thm:symgpd-cts}, $\mu^{-1}(\sigma^{-1}(U)), \mu^{-1}(\tau^{-1}(U))$ are open for open $U \subseteq X$.  Let $x --->{f} y --->{g} z \in \Hom_X(A)$ (i.e., $f : A_x -> A_y$ and $g : A_y -> A_z$ are Lipschitz maps) and $\sqsqbr{R |-> T} \subseteq \Hom_X(A)$ be a subbasic open set, with $R, T \subseteq A$ open, such that $g \cdot f \in \sqsqbr{R |-> T}$.  So there is $a \in R_x$ with $g(f(a)) \in T_z$.  Let $T' \ni g(f(a))$ be open and $r > 0$ with $[T']_{<r} \subseteq T$, and let $S \ni f(a)$ be open and $r$-small.  Then $(g, f) \in \sqsqbr{S |-> T'} \times_X \sqsqbr{R |-> S} \subseteq \mu^{-1}(\sqsqbr{R |-> T})$, since for any $g' : y' -> z' \in \sqsqbr{S |-> T'}$ and $f' : x' -> y' \in \sqsqbr{R |-> S}$, we have some $a' \in R_{x'}$ (since $x' \in p(R)$ by definition of $\sqsqbr{R |-> S}$) with $f'(a') \in S$, as well as some $b' \in S_{y'}$ with $g'(b') \in T'$, whence $d(g'(f'(a')), g'(b')) \le d(f'(a'), b') < r$, whence $g'(f'(a')) \in [T']_{<r} \subseteq T$, whence $g' \cdot f' \in \sqsqbr{R |-> T}$.
\end{proof}

\begin{proposition}
\label{thm:istgpd-qpol}
If $X, A$ are quasi-Polish, then so are $\Hom_X(A), \Iso_X(A)$.
\end{proposition}
\begin{proof}
As in \cref{thm:symgpd-qpol}, the embedding into the fiberwise lower powerspace
\begin{align*}
e : \Hom_X(A) &--> \@F_{X \times X}(A \times A) \\
(x, y, f) &|--> (x, y, \graph(f))
\end{align*}
is a continuous embedding; so to show that $\Hom_X(A)$ is quasi-Polish, it suffices to show that $\im(e)$ is $\*\Pi^0_2$.  Let $\@S$ be a countable basis of open sets in $A$.  We claim that for $(x, y, F) \in \@F_{X \times X}(A \times A)$,
\begin{align*}
(x, y, F) \in \im(e) \iff \left(\begin{aligned}
&\forall S \in \@S\, (x \in p(S) \implies \exists T \in \@S\, ((x, y, F) \in \Dia_{X^2}(S \times T))) \AND \\
&\forall S_1, S_2, T_1, T_2 \in \@S,\, r \in \#Q^+ \\
&\left(\begin{aligned}
&S_1 \cup S_2 \text{ $r$-small} \AND (x, y, F) \in \Dia_{X^2}(S_1 \times T_1) \cap \Dia_{X^2}(S_2 \times T_2) \\
&\implies \exists T_1 \supseteq T_1' \in \@S,\, T_2 \supseteq T_2' \in \@S \\
&\qquad\quad (T_1' \cup T_2' \text{ $r$-small} \AND (x, y, F) \in \Dia_{X^2}(S_1 \times T_1') \cap \Dia_{X^2}(S_2 \times T_2'))
\end{aligned}\right)
\end{aligned}\right).
\end{align*}

$\Longrightarrow$: Let $f : x -> y \in \Hom_X(A)$; we must check that $(x, y, \graph(f))$ obeys the right-hand side.  The first condition is straightforward.  To prove the second condition, let $S_1, S_2, T_1, T_2 \in \@S$ and $r \in \#Q^+$ such that $S_1 \cup S_2$ is $r$-small and $(x, y, \graph(f)) \in \Dia_{X^2}(S_1 \times T_1) \cap \Dia_{X^2}(S_2 \times T_2)$, i.e., there are $a_1 \in (S_1)_x$ and $a_2 \in (S_2)_x$ such that $f(a_1) \in T_1$ and $f(a_2) \in T_2$.  Since $S_1 \cup S_2$ is $r$-small, $d(f(a_1), f(a_2)) \le d(a_1, a_2) < r$.  Let $s > 0$ such that $d(f(a_1), f(a_2)) + 2s \le r$, let $T_1' \in \@S$ be $s$-small with $f(a_1) \in T_1' \subseteq T_1$, whence $f(a_2) \in [T_1']_{<d(f(a_1), f(a_2))+s}$, and let $T_2' \in \@S$ be $s$-small with $f(a_2) \in T_2' \subseteq T_2 \cap [T_1']_{<d(f(a_1), f(a_2))+s}$.  Then $T_1' \cup T_2'$ is $r$-small, since two points in the same fiber of $T_1'$ or $T_2'$ are within distance $s < r$, while for $b \in (T_1')_z$ and $c \in (T_2')_z$, we have $d(c, b') < d(f(a_1), f(a_2))+s$ for some $b' \in (T_1')_z$ whence $d(c, b) \le d(c, b') + d(b', b) < d(f(a_1), f(a_2))+s+s \le r$.  So $T_1', T_2'$ witness the second condition.

$\Longleftarrow$: Let $(x, y, F)$ satisfy the right-hand side; we must show that $F = \graph(f)$ for some Lipschitz $f : A_x -> A_y$.  We claim that the second condition on the right-hand side implies
\begin{align*}
\tag{$*$}  \forall (a_1, b_1), (a_2, b_2) \in F \cap (A_x \times A_y)\, (d(a_1, a_2) \ge d(b_1, b_2)).
\end{align*}
Let $(a_1, b_1), (a_2, b_2) \in F \cap (A_x \times A_y)$ and $d(a_1, a_2) < r \in \#Q^+$.  Let $s > 0$ such that $d(a_1, a_2) + 2s \le r$, let $S_1 \in \@S$ be $s$-small with $a_1 \in S_1$, and let $S_2 \in \@S$ be $s$-small with $a_2 \in S_2 \subseteq [T_1]_{<d(a_1, a_2)+s}$.  Then by similar reasoning as in the proof of $\Longrightarrow$ above, $S_1 \cup S_2$ is $r$-small.  Let $T_1, T_2 \in \@S$ be $s$-small with $b_1 \in T_1$ and $b_2 \in T_2$.  Then $(x, y, F) \in \Dia_{X^2}(S_1 \times T_1) \cap \Dia_{X^2}(S_2 \times T_2)$ as witnessed by $(a_1, b_1), (a_2, b_2)$, whence by the second condition, there are $T_1 \supseteq T_1' \in \@S$ and $T_2 \supseteq T_2' \in \@S$ such that $T_1' \cup T_2'$ is $r$-small and $(x, y, F) \in \Dia_{X^2}(S_1 \times T_1') \cap \Dia_{X^2}(S_2 \times T_2')$, i.e., there are $a_1' \in (S_1)_x$, $a_2' \in (S_2)_x$, $b_1' \in (T_1')_x$, and $b_2' \in (T_2')_x$ such that $(a_1', b_1'), (a_2', b_2') \in F$.  Then
\begin{align*}
d(b_1, b_2)
&\le d(b_1, b_1') + d(b_1', b_2') + d(b_2', b_2) \\
&< s + r + s \qquad\text{since $T_1$ $s$-small, $T_1' \cup T_2'$ $r$-small, and $T_2$ $s$-small}.
\end{align*}
Taking $s \searrow 0$ gives $d(b_1, b_2) \le r$, and then taking $r \searrow d(a_1, a_2)$ gives ($*$).

In particular, taking $a_1 = a_2$ in ($*$) shows that $F$ is the graph of a partial function $f : A_x \rightharpoonup A_y$.  To see that $f$ is total, let $a \in A_x$.  For each $n \in \#N^+$, let $S_n \in \@S$ be $1/n$-small with $a \in (S_n)_x$; then by the first condition on the right-hand side, there is some $(a_n, b_n) \in F \cap ((S_n)_x \times A_y)$.  Then $\lim_n a_n = a$, whence $(a_n)_n \subseteq A_x$ is Cauchy, whence so is $(b_n)_n$ by ($*$); let $b = \lim_n b_n \in A_y$.  Since each $(a_n, b_n) \in F$ and $F$ is closed, it follows that $(a, b) \in F$, i.e., $f(a) = b$.  Thus $f$ is a total function $A_x -> A_y$, and is Lipschitz by ($*$).  This completes the proof that $\im(e)$ is $\*\Pi^0_2$.

To show that $\Iso_X(A) \subseteq \Hom_X(A)$ is $\*\Pi^0_2$, add the ``converses'' of the above conditions.
\end{proof}

\subsection{Metric-étale structures}
\label{sec:metalestr}

Let $\@L$ be a (multi-sorted) relational first-order language and $X$ be a topological space.  A \defn{metric-étale $\@L$-structure $\@M$ over $X$} consists of:
\begin{itemize}
\item  for each sort $P \in \@L_s$, an underlying metric-étale space $p : P^\@M -> X$ over $X$ (when $\@L$ is one-sorted, we denote $P^\@M$ for the unique sort $P$ by $M$);
\item  for each $(P_0, \dotsc, P_{n-1})$-ary relation symbol $R \in \@L_r$, an invariant open grey subset $R^\@M \sqle P_0^\@M \times_X \dotsb \times_X P_{n-1}^\@M$ (where $P_0^\@M \times_X \dotsb \times_X P_{n-1}^\@M$ has the sum metric).
\end{itemize}
Given a metric-étale $\@L$-structure $\@M$ over $X$, for each $x \in X$, we have a metric $\@L$-structure
\begin{align*}
\@M_x := (P^\@M_x, R^\@M_x)_{P,R \in \@L}.
\end{align*}

We say that a metric-étale $\@L$-structure $\@M$ over $X$ is \defn{second-countable} if $\@L$ is countable, $X$ is second-countable, and $P^\@M$ is second-countable for each sort $P \in \@L_s$.

The \defn{homomorphism category} $\Hom_X(\@M)$ and \defn{isomorphism groupoid} $\Iso_X(\@M)$ of a metric-étale $\@L$-structure $\@M$ over $X$ are defined exactly as in \cref{sec:etalestr}.  As there, we let
\begin{align*}
\sqsqbr{S |-> T}_P = \{(x, y, f) \in \Hom_X(\@M) \mid f_P(S_x) \cap T_y \ne \emptyset\} \qquad\text{for open $S, T \subseteq P^\@M$}
\end{align*}
denote the subbasic open sets in $\Hom_X(\@M)$, which generate the topology along with the maps $\sigma, \tau$.  We use the same notation for the restrictions of these sets to $\Iso_X(\@M)$.

\begin{proposition}
\label{thm:metalestr-isogpd-qpol}
If $\@M$ is second-countable, then $\Hom_X(\@M)$ (resp., $\Iso_X(\@M)$) is a quasi-Polish category (resp., groupoid).
\end{proposition}
\begin{proof}
Let $(x, y, f) = (x, y, f_P)_P \in \prod_{X \times X} (\Hom_X(P^\@M))_{P \in \@L_s}$; we must show that $f$ being a homomorphism $\@M_x -> \@M_y$ is a $\*\Pi^0_2$ condition.  For a $(P_0, \dotsc, P_{n-1})$-ary relation symbol $R \in \@L_r$, $f$ preserves $R$ iff for every $r \in \#Q^+$, we have
\begin{align*}
(R^\@M_{<r})_x \subseteq \-{(f_{P_0} \times \dotsb \times f_{P_{n-1}})^{-1}((R^\@M_{<r})_y)}^{(P_0^\@M)_x \times \dotsb \times (P_{n-1}^\@M)_x},
\end{align*}
i.e., every tuple $\vec{a} \in \@M_x$ with $R(\vec{a}) < r$ is a limit of tuples $\vec{b}$ with $R(f(\vec{b})) < r$.  That this condition is $\*\Pi^0_2$ follows from
\begin{lemma}
\label{lm:lipcat-homom-pi02}
Let $p : A -> X$ be a second-countable metric-étale space, and $S, T \subseteq A$ be open.  Then
\begin{align*}
\{f : A_x -> A_y \in \Hom_X(A) \mid S_x \subseteq \-{f^{-1}(T_y)}^{A_x}\}
\end{align*}
is $\*\Pi^0_2$.
\end{lemma}
\begin{proof}
Let $\@S$ be a countable basis of open sets in $A$.  For $f : A_x -> A_y \in \Hom_X(A)$, we have
\begin{align*}
S_x \subseteq \-{f^{-1}(T_y)}^{A_x}
&\iff \forall \text{ basic open } U \subseteq A_x\, (U \cap S_x \ne \emptyset \implies U \cap f^{-1}(T_y) \ne \emptyset) \\
&\iff \forall \text{ basic open } U \subseteq A_x\, (U \cap S_x \ne \emptyset \implies f(U) \cap T_y \ne \emptyset) \\
&\iff \forall R \in \@S\, (x \in p(R \cap S) \implies f \in \sqsqbr{R |-> T}).
\qedhere
\end{align*}
\end{proof}
\noindent
together with the easy fact that the map
\begin{align*}
\prod_{X \times X} (\Hom_X(P^\@M))_{P \in \@L_s} &--> \Hom_X(P_0^\@M \times_X \dotsb \times_X P_{n-1}^\@M) \\
(x, y, f_P)_P &|--> f_{P_0} \times_X \dotsb \times_X f_{P_{n-1}} : (P_0^\@M)_x \times \dotsb \times (P_{n-1}^\@M)_x -> (P_0^\@M)_y \times \dotsb \times (P_{n-1}^\@M)_y
\end{align*}
is continuous.
\end{proof}

\subsection{Fiberwise separable Borel metric spaces and structures}

As in the étale case, we next develop the Borel analogs of the above notions.

Let $X$ be a standard Borel space.  A \defn{fiberwise separable Borel pseudometric space over $X$} is a standard Borel space $p : A -> X$ over $X$ equipped with a Borel pseudometric $d \sqle A \times_X A$, discrete between fibers, such that there exists a fiberwise countable $d$-dense Borel subset $S \subseteq A$.  Equivalently, by Lusin--Novikov, there are countably many \defn{Borel sections} $S \subseteq A$ (i.e., $p|S$ is injective), which we can always assume to be \defn{full}, meaning $p(S) = p(A)$, which enumerate a countable dense set in each fiber $A_x$.  Note that these conditions imply that $p(A) \subseteq X$ is Borel.  We call $A$ a \defn{fiberwise separable Borel complete metric space over $X$} if $d$ is a complete metric (equivalently, each fiber is a complete metric space).

\begin{proposition}
Let $p : A -> X$ be a quasi-Polish metric-étale space over a quasi-Polish space $X$.  Then $A$ is a fiberwise separable Borel complete metric space over $X$.
\end{proposition}
\begin{proof}[Proof]
Let $\@S$ be a countable basis of open sets in $A$.  For each $S \in \@S$, $p|S : S -> X$ is a continuous open map between quasi-Polish spaces, hence by a standard large section uniformization argument, has a Borel section $s_S : p(S) -> S$ (see \cite[7.9]{Cqpol}).  Then $\{\im(s_S) \mid S \in \@S\}$ is a countable family of fiberwise dense Borel sections in $A$.
\end{proof}

\begin{remark}
Here is an alternative, more explicit, construction of the Borel sections $s_S$ in the preceding proof, which has the advantage that the resulting countable family of Borel sections has known Borel complexity.

Let $\@S = \{S_0, S_1, \dotsc\}$.  Fix an $S = S_{n_0} \in \@S$.  For each $x \in p(S)$, let $n_0(x) := n_0$, and inductively for each $k \in \#N$ let
\begin{align*}
n_{k+1}(x) := \text{least $n$ s.t.\ } x \in p(S_n) \AND S_n \text{ is $1/(k+1)$-small} \AND \exists r \in \#Q^+\, ([S_n]_{<r} \subseteq S_{n_k(x)}).
\end{align*}
Then for each $x \in p(S)$, $((S_{n_k(x)})_x)_{k \in \#N}$ is a sequence of nonempty open sets of vanishing diameter in the complete metric space $A_x$, each containing the closure of the next, whence their intersection is a point $s_S(x) \in (S_{n_0(x)})_x = S_x$.

By induction, each $n_k : p(S) -> \#N$ (with the discrete topology on $\#N$) is Baire class $k$.  Hence
\begin{align*}
\im(s_S) &= \{a \in p^{-1}(p(S)) \mid a = s_S(p(a))\} \\
&= \{a \in p^{-1}(p(S)) \mid \forall k\, \exists n\, (n_k(p(a)) = n \AND a \in S_n)\}
\end{align*}
is $\*\Pi^0_\omega$.
\end{remark}

As in the étale case (\cref{rmk:cbsp-etale}), there is a converse to the preceding result:

\begin{proposition}
\label{thm:sbmsp-metale}
Let $p : A -> X$ be a fiberwise separable Borel complete metric space over $X$.  Then there are compatible Polish topologies on $A, X$ turning $p : A -> X$ into a metric-étale space.
\end{proposition}

The proof uses the following generalized version of the Kunugui--Novikov uniformization theorem for fiberwise open Borel sets \cite[28.7]{Kcdst}, with essentially the same proof:

\begin{lemma}
\label{thm:sbmsp-kunugui-novikov}
Let $p : A -> X$ be a standard Borel space over $X$ and $\@S$ be a countable family of Borel subsets of $A$.  Then any Borel $R \subseteq A$ which is fiberwise a union of fibers of $S \in \@S$ can be written as a countable union of sets of the form $p^{-1}(B) \cap S$ with $B \subseteq X$ Borel and $S \in \@S$.
\end{lemma}
\begin{proof}
For every $a \in R_x$, we have $a \in S_x \subseteq R$ for some $S \in \@S$.  So putting $X_S := \{x \in X \mid S_x \subseteq R\}$, we have $R = \bigcup_{S \in \@S} (p^{-1}(X_S) \cap S)$.  Each $X_S$ is $\*\Pi^1_1$, whence so is $p^{-1}(X_S) \cap S$, so by the Novikov separation theorem \cite[28.5]{Kcdst}, there are Borel sets $R_S \subseteq p^{-1}(X_S) \cap S$ with $R = \bigcup_S R_S$.  We have $p(R_S) \subseteq X_S$, so by the Lusin separation theorem, there are Borel sets $B_S \subseteq X$ with $p(R_S) \subseteq B_S \subseteq X_S$.  Then $R = \bigcup_S R_S \subseteq \bigcup_S (p^{-1}(B_S) \cap S) \subseteq \bigcup_S (p^{-1}(X_S) \cap S) = R$.
\end{proof}

\begin{proof}[Proof of \cref{thm:sbmsp-metale}]
By an $\omega$-iteration, we may find a countable family $\@S$ of fiberwise dense Borel sections in $A$ together with a countable Boolean algebra $\@U$ of Borel sets in $X$ such that
\begin{enumerate}[label=(\roman*)]
\item  $\@U$ generates a (zero-dimensional) Polish topology on $X$;
\item  $p(S) \in \@U$ for all $S \in \@S$;
\item  $p^{-1}(U) \cap S \in \@S$ for all $U \in \@U$ and $S \in \@S$;
\item  for all $S, S' \in \@S$ and rational $r, r' > 0$, $[S]_{<r} \cap [S']_{<r'}$ is a countable union of $[S'']_{<r''}$ for $S'' \in \@S$ and rational $r'' > 0$ (using (iii) and \cref{thm:sbmsp-kunugui-novikov}); and $p([S]_{<r} \cap [S']_{<r'}) \in \@U$.
\end{enumerate}
Equip $X$ with the topology generated by $\@U$ and $A$ with the topology generated by $[S]_{<r}$ for $S \in \@S$ and rational $r > 0$; the latter sets form a basis by (iv).  Then $p$ is continuous by (iii) and density of $\@S$, and open by (ii).
To check that $d \sqle A \times_X A$ is open: if $d(a, b) < r$, then there are rational $s, t > 0$ with $d(a, b) \le s < s+2t \le r$ and $S \in \@S$ with $a \in [S]_{<t}$; then it is easily seen that $(a, b) \in [S]_{<t} \times_X [S]_{<s+t} \subseteq d_{<r}$.
To check that the topology on $A$ is $d$-invariant: for a basic open set $[S]_{<r}$, we have $[S]_{<r} = \bigcup_{r \ge s+t} [[S]_{<s}]_{<t}$.
Thus $p : A -> X$ is metric-étale, hence $A$ is quasi-Polish by \cref{thm:metale-qpol}.

Finally, to check that $A$ is Polish: note that the topological closure $\-{[S]_{<r}}$ in $A$ of a basic open set $[S]_{<r}$ is contained in $[S]_{\le r}$.  Indeed, for $a \not\in [S]_{\le r}$, either $x := p(a) \not\in p(S)$ in which case $p^{-1}(X \setminus p(S))$ is an open neighborhood of $a$ disjoint from $[S]_{<r}$, or else letting $b \in S_x$ be the unique element, we have $d(a, b) > r$, so letting $r' > 0$ be rational with $d(a, b) \ge r+2r'$, there is some $S' \in \@S$ with $a \in [S']_{<r'}$, whence $x \not\in p([S]_{<r} \cap [S']_{<r'})$ (since for $c \in ([S]_{<r} \cap [S']_{<r'})_x$ we would have $d(a, b) \le d(a, S'_x) + d(S'_x, c) + d(c, b) < r' + r' + r$), whence by (iv) and closure of $\@U$ under complements, $p^{-1}(X \setminus p([S]_{<r} \cap [S']_{<r'})) \cap [S']_{<r'}$ is an open neighborhood of $a$ disjoint from $[S]_{<r}$.  It now follows from the formula $[S]_{<r} = \bigcup_{s<r} [S]_{<s}$ that $A$ is regular, hence Polish.
\end{proof}

For a fiberwise separable Borel complete metric space $p : A -> X$ over $X$, we define the \defn{Lipschitz category} $\Hom_X(A)$ and \defn{isometry groupoid} $\Iso_X(A)$ to be the underlying standard Borel category (groupoid) of the quasi-Polish category (groupoid) defined in \cref{sec:istgpd}, for any compatible (quasi-)Polish topologies on $A, X$ turning $A$ into a metric-étale space over $X$.  That the Borel structure on $\Hom_X(A), \Iso_X(A)$ does not depend on the chosen topology follows from

\begin{lemma}
Let $p : A -> X$ be a quasi-Polish metric-étale space.  The standard Borel structure on $\Hom_X(A)$ is generated by the maps $\sigma, \tau : \Hom_X(A) -> X$ and either of:
\begin{enumerate}[label=(\roman*)]
\item  the sets $\sqsqbr{S |-> T}$ (as defined in \cref{sec:symgpd}) for $S, T \in \@S$, for any countable family $\@S$ of Borel fiberwise open sets in $A$ restricting to a basis in each fiber; or
\item  the sets
\begin{align*}
\sqsqbr{S |-> [T]_{<r}} = \{f : A_x -> A_y \in \Hom_X(A) \mid x \in p(S) \AND d(f(S_x), T_y) < r\}
\end{align*}
for $S, T \in \@S$ and $r \in \#Q^+$, for any countable family $\@S$ of fiberwise dense Borel sections in $A$.
\end{enumerate}
\end{lemma}
\begin{proof}
First we prove (i).  Note that the standard Borel structure on $\Hom_X(A)$ as defined in \cref{sec:istgpd} is also of the form (i), by taking $\@S$ to be a countable basis of open sets in $A$.  So it is enough to prove that any two families $\@S, \@S'$ as in (i) yield the same Borel structure.  This follows from \cref{thm:sbmsp-kunugui-novikov} and the identities
\begin{align*}
\sqsqbr{\bigcup_i S_i |-> T} &= \bigcup_i \sqsqbr{S_i |-> T}, &
\sqsqbr{p^{-1}(U) \cap S |-> T} &= \sigma^{-1}(U) \cap \sqsqbr{S |-> T}, \\
\sqsqbr{S |-> \bigcup_i T_i} &= \bigcup_i \sqsqbr{S |-> T_i}, &
\sqsqbr{S |-> p^{-1}(U) \cap T} &= \tau^{-1}(U) \cap \sqsqbr{S |-> T}
\end{align*}
(for Borel $S, T \subseteq A$ and $U \subseteq X$).

Now let $\@S, R$ be as in (ii).  It is straightforward to check that for $S, T \in \@S$ and $r, s, t \in \#Q^+$,
\begin{align*}
\sqsqbr{S |-> [T]_{<r}} &= \bigcup \{\sqsqbr{[S]_{<s} |-> [T]_{<t}} \mid s, t \in \#Q^+ \AND s + t \le r\}, \\
\sqsqbr{[S]_{<s} |-> [T]_{<t}} &= \bigcup \{\sqsqbr{[S]_{<s} \cap S' |-> [T]_{<t}} \mid S' \in \@S\} \\
&= \bigcup \{\sqsqbr{p^{-1}(p([S]_{<s} \cap S')) \cap S' |-> [T]_{<t}} \mid S' \in \@S\} \\
&= \bigcup \{\sigma^{-1}(p([S]_{<s} \cap S')) \cap \sqsqbr{S' |-> [T]_{<t}} \mid S' \in \@S\}.
\end{align*}
Thus, the Borel structure in (ii) coincides with the one in (i) where $\@S$ (in (i)) is taken to consist of $[S]_{<r}$ for $S \in \@S$ and $r \in \#Q^+$.
\end{proof}

For a countable relational language $\@L$ and standard Borel space $X$, a \defn{fiberwise separable Borel metric $\@L$-structure} $\@M = (P^\@M, R^\@M)_{P,R \in \@L}$ over $X$ consists of
\begin{itemize}
\item  for each sort $P \in \@L_s$, an underlying fiberwise separable Borel complete metric space $p : P^\@M -> X$ over $X$ ($P^\@M$ denoted $M$ for one-sorted $\@L$);
\item  for each $(P_0, \dotsc, P_{n-1})$-ary $R \in \@L_r$, a $d$-invariant Borel grey subset $R^\@M \sqle P_0^\@M \times_X \dotsb \times_X P_{n-1}^\@M$.
\end{itemize}
Every second-countable metric-étale $\@L$-structure $\@M$ has an underlying fiberwise separable Borel metric $\@L$-structure.

The \defn{homomorphism category} $\Hom_X(\@M)$ and \defn{isomorphism groupoid} $\Iso_X(\@M)$ of a fiberwise separable Borel metric structure $\@M$ are defined exactly as in \cref{sec:metalestr} (that is, \cref{sec:etalestr}).  By the proof of \cref{thm:metalestr-isogpd-qpol}, with the following analog of \cref{lm:lipcat-homom-pi02}, $\Hom_X(\@M), \Iso_X(\@M)$ are standard Borel:

\begin{lemma}
\label{lm:lipcat-homom-borel}
Let $p : A -> X$ be a fiberwise separable Borel complete metric space, and $S, T \subseteq A$ be Borel and fiberwise open.  Then
\begin{align*}
\{f : A_x -> A_y \in \Hom_X(A) \mid S_x \subseteq \-{f^{-1}(T_y)}^{A_x}\}
\end{align*}
is Borel.
\end{lemma}
\begin{proof}
Repeat the proof of \cref{lm:lipcat-homom-pi02}, with $\@S$ taken to be a countable family of Borel fiberwise open sets restricting to a basis in each fiber.
\end{proof}

\subsection{Katětov functions}
\label{sec:katetov}

In the discrete setting (\cref{sec:cbstr-unif}), we turned any fiberwise countable Borel structure into a fiberwise countably infinite Borel structure by adjoining $\aleph_0$-many constants to each fiber.  The metric analog of this procedure will be given by Katětov's construction of the Urysohn sphere \cite{Kat}.  In this and the next section, we prove that the construction may be performed in a uniform Borel fashion.  This technique is well-known, and there are several similar applications in the literature (see \cite[2.2]{GK}, \cite[3.2]{Usp}, \cite[2.3]{EFPRTT}); however, there does not appear to be a precise statement and proof with the level of uniformity we need.  For the sake of completeness, we will give the details.

For a (pseudo)metric space $X$ (of diameter $\le 1$), a \defn{Katětov function} on $X$ is a (1-)Lipschitz function $u : X -> \#I$ satisfying
\begin{align*}
u(x) + u(y) \ge d(x, y) \qquad\forall x, y \in X,
\end{align*}
and can be thought of as specifying the distances between points in $X$ and another point in a larger metric space containing $X$ (with the above inequality expressing the triangle inequality).  In the language of grey sets (recall \cref{sec:greyrel}), the above inequality says simply
\begin{align*}
u \oplus u \sqle d.
\end{align*}
Let
\begin{align*}
\@E(X) := \{u \in \#I^X \mid u \text{ is a Katětov function}\},
\end{align*}
equipped with the sup metric.  Note that $\@E(X) \subseteq \ell^\infty(X)$ is closed, hence complete.  We have a canonical isometric map
\begin{align*}
\delta = \delta_X : X &--> \@E(X) \\
x &|--> d(x, -),
\end{align*}
whose image closure is thus (isometric to) the completion of $X$ (or $X/d$, if pseudometric).  The following simple fact (an instance of the \emph{enriched Yoneda lemma}) is key:
\begin{align*}
d_{\@E(X)}(\delta(x), u) = u(x) \qquad\forall x \in X,\, u \in \@E(X).
\end{align*}

A Lipschitz map $f : X -> Y$ extends along $\delta$ to a map
\begin{align*}
\@E(f) : \@E(X) &--> \@E(Y) \\
u &|--> [f(u)]
\end{align*}
(here treating $u \sqle X$ as a grey subset), which is easily seen to be Lipschitz (since if $u \le v \dotplus r$ then $[f(u)] \le [f(v \dotplus r)] = [f(v) \dotplus r] = [f(v)] \dotplus r$, using that taking image of grey subsets preserves $\dotplus r$).
$\@E$ is functorial: $\@E(g \circ f) = \@E(g) \circ \@E(f)$, and $\@E(1_X) = 1_{\@E(X)}$.
If $f$ is isometric, then so is $\@E(f)$, and we then have $\@E(f)(u) \circ f = u$ for all $u \in \@E(X)$; if $f$ furthermore has dense image, then $\@E(f)$ is bijective, and given by $\@E(f)(u) = u \circ f^{-1}$ when $f$ is bijective.  For $f : X \subseteq Y$ an (isometric) inclusion, we denote $\@E(f)$ by $\@E(X \subseteq Y)$ (which is simply $u |-> d_Y[u]$).

\begin{lemma}[{see \cite[1.6]{Kat}, \cite[3.3]{GK}}]
\label{lm:katetov-ff}
Let $X, Y$ be metric spaces, $f : X \cong Y$ be an isometric bijection, and $g : \@E(X) \rightharpoonup \@E(Y)$ be a partial isometry whose domain $\dom(g)$ contains $\im(\delta_X)$.  Then $g = \@E(f)|\dom(g)$ iff $g \circ \delta_X = \delta_Y \circ f$, i.e., the following commutes:
\begin{equation*}
\begin{tikzcd}
X \dar["f"'] \rar[hook, "\delta_X"] & \dom(g) \mathrlap{{}\subseteq \@E(X)} \dar["g"] \\
Y \rar[hook, "\delta_Y"'] & \@E(Y)
\end{tikzcd}
\end{equation*}
\end{lemma}
\begin{proof}
For all $u \in \dom(g)$ and $y \in Y$, we have
\begin{align*}
g(u)(y)
&= d_{\@E(Y)}(\delta_Y(y), g(u)), \\
\@E(f)(u)(y)
&= u(f^{-1}(y)) \\
&= d_{\@E(X)}(\delta_X(f^{-1}(y)), u) \\
&= d_{\@E(Y)}(g(\delta_X(f^{-1}(y))), g(u)).
\end{align*}
If $\delta_Y \circ f = g \circ \delta_X$ then clearly these are equal.  Conversely, if these are equal for all $u, y$, then taking $u := \delta_X(f^{-1}(y'))$ for $y' \in Y$ gives
$g(\delta_X(f^{-1}(y')))(y)
= d_{\@E(Y)}(\delta_Y(y), g(\delta_X(f^{-1}(y'))))
= d_{\@E(Y)}(g(\delta_X(f^{-1}(y))), g(\delta_X(f^{-1}(y'))))
= d_Y(y, y')$,
whence $g \circ \delta_X = \delta_Y \circ f$.
\end{proof}


\begin{corollary}
\label{thm:katetov-ff}
Let $X, Y$ be metric spaces and $g : \@E(X) \cong \@E(Y)$ be an isometric bijection such that $g(\im(\delta_X)) = \im(\delta_Y)$.  Then there is a unique isometric bijection $f : X \cong Y$ such that $g = \@E(f)$.
\end{corollary}
\begin{proof}
By \cref{lm:katetov-ff}, the unique $f$ is given by $f := \delta_Y^{-1} \circ g \circ \delta_X$.
\end{proof}

\begin{lemma}
\label{thm:katetov-finsupp-dense}
Let $X, Y \subseteq Z$ be pseudometric spaces.  Then $d_H(\im(\@E(X \subseteq Z)), \im(\@E(Y \subseteq Z))) \le 2d_H(X, Y)$, where $d_H$ is Hausdorff distance.
\end{lemma}
\begin{proof}
Let $u \in \@E(X)$; we show that $d(d_Z[u], \im(\@E(Y \subseteq Z))) \le 2d_H(X, Y)$.  Let $v := d_Z[u]|Y \in \@E(Y)$, $z \in Z$, and $r > 0$.
Then for any $x \in X$ and $y \in Y$, we have
\begin{align*}
d_Z[u](z)
&= \bigwedge_{x' \in X} (d(z, x') \dotplus u(x')) \\
&\le d(z, x) + u(x) \\
&\le d(z, y) + d(y, x) + u(x);
\end{align*}
thus
\begin{align*}
d_Z[u](z)
&\le \bigwedge_{y \in Y} (d(z, y) \dotplus \bigwedge_{x \in X} (d(y, x) \dotplus u(x))) \\
&= \bigwedge_{y \in Y} (d(z, y) \dotplus v(y)) \\
&= d_Z[v](z).
\end{align*}
Conversely, for any $x \in X$, there is $y \in Y$ with $d(x, y) < d_H(X, Y) + r$, whence
\begin{align*}
d_Z[v](z)
&= \bigwedge_{y' \in Y} (d(z, y') \dotplus v(y')) \\
&\le d(z, y) + v(y) \\
&= d(z, y) + \bigwedge_{x' \in X} (d(y, x') \dotplus u(x')) \\
&\le d(z, y) + d(y, x) + u(x) \\
&\le d(z, x) + 2d(x, y) + u(x) \\
&< d(z, x) + u(x) + 2(d_H(X, Y) + r);
\end{align*}
thus
\begin{align*}
d_Z[v](z) &\le \bigwedge_{x \in X} (d(z, x) \dotplus u(x)) + 2(d_H(X, Y) + r) \\
&= d_Z[u](z) + 2(d_H(X, Y) + r).
\end{align*}
So $\abs{d_Z[u](z) - d_Z[v](z)} \le 2d_H(X, Y)$.
\end{proof}

Let\footnote{This is more commonly denoted $\overline{E(X, \omega)}$, e.g., in \cite[\S2C]{GK}.}
\begin{align*}
\@E'(X) := \-{\bigcup_{\text{finite } F \subseteq X} \im(\@E(F \subseteq X))} \subseteq \@E(X).
\end{align*}
Clearly, the canonical isometry $\delta : X -> \@E(X)$ lands in $\@E'(X)$.
Also, clearly, for Lipschitz $f : X -> Y$, $\@E(f)$ restricts to $\@E'(X) -> \@E'(Y)$; call this restriction $\@E'(f)$.
It follows from \cref{thm:katetov-finsupp-dense} that for isometric $f : X -> Y$ with dense image, $\@E'(f)$ is an isometric bijection $\@E'(X) \cong \@E'(Y)$.  Hence for separable $X$, with countable dense $D \subseteq X$, we have $\@E'(X) \cong \@E'(D)$; since $\@E'(D)$ is the closure of the countable union of isometric copies of $\@E(F)$ for finite $F \subseteq D$, each of which is separable, it follows that $\@E'(X) \cong \@E'(D)$ is separable, with a countable dense set given by $\@E(F \subseteq X)(u) = d_X[u]$ for finite $F \subseteq D$ and rational-valued $u \in \@E(F)$.

\begin{lemma}
\label{thm:katetov'-ff}
Let $X, Y$ be metric spaces and $g : \@E'(X) \cong \@E'(Y)$ be an isometric bijection satisfying $g(\im(\delta_X)) = \im(\delta_Y)$.  Then there is a unique isometric bijection $f : X \cong Y$ such that $g = \@E'(f)$.
\end{lemma}
\begin{proof}
By \cref{lm:katetov-ff}.
\end{proof}

We now perform the above constructions in the fiberwise Borel context.  Let $p : A -> X$ be a fiberwise separable Borel pseudometric space over standard Borel $X$.  Let
\begin{align*}
\@E_X(A) := \bigsqcup_{x \in X} \@E(A_x) := \{(x, u) \mid x \in X \AND u \in \@E(A_x)\},
\end{align*}
equipped with the first projection $\@E_X(p) : \@E_X(A) -> X$; we identify the fibers $\@E_X(A)_x$ with $\@E(A_x)$ and write $u \in \@E_X(A)_x$ instead of $(x, u) \in \@E_X(A)_x$ when $x$ is understood.  We equip $\@E_X(A)$ with a standard Borel structure as follows.  Let $\@S$ be a countable family of fiberwise dense \emph{full} Borel sections in $A$.  By the above, we have a fiberwise isometry $\@E_X(A) \cong \@E_X(\bigcup \@S)$.  We equip $\@E_X(\bigcup \@S)$ with a Borel structure via the injection
\begin{align*}
\@E_X(\bigcup \@S) &`--> p(A) \times \#I^\@S \\
(x, u) &|--> (x, (S |-> u(S_x)))
\end{align*}
whose image is easily seen to be Borel.  For two countable families $\@S \subseteq \@S'$ of fiberwise dense full Borel sections in $A$, the canonical fiberwise isometry $\@E_X(\bigcup \@S') -> \@E_X(\bigcup \@S)$ is easily seen to be Borel (being equivalent, along the above injection, to the projection $\#I^{\@S'} -> \#I^\@S$); thus since the collection of countable families $\@S$ of fiberwise dense full Borel sections in $A$ is directed upwards under $\subseteq$, the spaces $\@E_X(\bigcup \@S)$ are all canonically isomorphic for any choice of $\@S$.  Equip $\@E_X(A)$ with the Borel structure of $\@E_X(\bigcup \@S)$ for any such $\@S$.  Explicitly, the Borel structure is generated by the maps $\@E_X(p) : \@E_X(A) -> X$ along with
\begin{align*}
\@E_X(A) &--> \#I \\
(x, u) &|--> u(S_x)
\end{align*}
for all full Borel sections $S \subseteq A$ (since we may always include $S$ in our choice of $\@S$ above).  The fiberwise metric $d : \@E_X(A) \times_X \@E_X(A) -> \#I$ is easily seen to be Borel, albeit not fiberwise separable in general.  Nonetheless, note that for any countable family $\@T$ of (partial) Borel sections in $\@E_X(A)$, we may define the fiberwise closure $\-{\bigcup \@T} \subseteq \@E_X(A)$ in a Borel way, using $d$.

To define $\@E'_X(A) \subseteq \@E_X(A)$, consisting fiberwise of $\@E'(A_x)$, again fix an $\@S$ as above.  For a finite $\@F \subseteq \@S$, the extension map $\@E_X(\bigcup \@F \subseteq A) : \@E_X(\bigcup \@F) -> \@E_X(A)$ (given fiberwise by $\@E((\bigcup \@F)_x \subseteq A_x)$ as defined above) is an isometric embedding, and is easily seen to be Borel.  For each $\@F$, we may easily find a countable family $\@T_\@F$ of full Borel sections in $\@E_X(\bigcup \@F)$ enumerating in each fiber $\@E((\bigcup \@F)_x)$ all the rational-valued $u \in \@E((\bigcup \@F)_x)$, which is thus fiberwise dense in $\@E_X(\bigcup \@F)$.  It follows that
\begin{align*}
\@E'_X(A) = \-{\bigcup_{\text{finite } \@F \subseteq A} \im(\@E_X(\bigcup \@F \subseteq A))} = \-{\bigcup_{\text{finite } \@F \subseteq \@S} \@E_X(\bigcup \@F \subseteq A)(\bigcup \@T_\@F)} \subseteq \@E_X(A)
\end{align*}
is Borel, and has a countable family of fiberwise dense (full) Borel sections consisting of $\@E_X(\bigcup \@F \subseteq A)(T)$ for all finite $\@F \subseteq \@S$ and $T \in \@T_\@F$, hence is a fiberwise separable Borel complete metric space over $X$.

Finally, it is easily seen that the canonical fiberwise isometry $\delta : A -> \@E_X(A)$ is Borel; and it lands in $\@E'_X(A) \subseteq \@E_X(A)$.  To summarize, we have shown:

\begin{proposition}
\label{thm:sbmsp-katetov}
Let $p : A -> X$ be a fiberwise separable Borel pseudometric space over $X$.  Then
\begin{align*}
\@E'_X(A) := \bigsqcup_{x \in X} \@E'(A_x) := \{(x, u) \mid x \in X \AND u \in \@E'(A_x)\}
\end{align*}
is a standard Borel space, with Borel structure generated by the maps
\begin{align*}
\@E'_X(p) : \@E'_X(A) &--> X, &
\@E'_X(A) &--> \#I \\
(x, u) &|--> x &
(x, u) &|--> u(S_x)
\end{align*}
for any full Borel section $S \subseteq A$; and $\@E'_X(p) : \@E'_X(A) -> X$ with the fiberwise metric $d : \@E'_X(A) \times_X \@E'_X(A) -> \#I$ is a fiberwise separable Borel complete metric space over $X$.  Moreover, the canonical fiberwise isometry $\delta : A -> \@E'_X(A)$ is Borel.  \qed
\end{proposition}

\begin{corollary}
\label{thm:sbmsp-compl}
Let $p : A -> X$ be a fiberwise separable Borel pseudometric space over $X$.  There is a unique standard Borel structure on the (fiberwise) completion $\^{A/d}$ making the canonical quotient map $\pi : A -> \^{A/d}$, the projection $p' : \^{A/d} -> X$, and the metric $d : \^{A/d} \times_X \^{A/d} -> \#I$ Borel; and this Borel structure makes $\^{A/d}$ into a fiberwise separable Borel complete metric space over $X$.
\end{corollary}
In the above situation, we always equip $\^{A/d}$ with said Borel structure.
\begin{proof}
To define the Borel structure on $\^{A/d}$, identify it with $\-{\im(\delta)} \subseteq \@E'_X(A)$.  Let $\@S$ be a countable family of fiberwise dense full Borel sections in $A$.  Then $\pi(\@S) := \{\pi(S) \mid S \in \@S\}$ is such a family in $\^{A/d}$, whence $\^{A/d}$ is fiberwise separable.  To show uniqueness of the Borel structure on $\^{A/d}$, if $\pi' : A -> \^{A/d}'$ with projection $p'' : \^{A/d}' -> X$ is another version of the completion which is also Borel, the unique isometric bijection $f : \^{A/d} -> \^{A/d}'$ over $X$ such that $f \circ \pi = \pi'$ is defined by
\begin{align*}
f(a) = b \iff p'(a) = p''(b) \AND \forall S \in \@S\, (d_{\^{A/d}}(a, \pi(S)) = d_{\^{A/d}'}(b, \pi'(S)))
\end{align*}
which is Borel.
\end{proof}

\subsection{The Urysohn sphere}
\label{sec:urysohn}

A metric space $X$ (of diameter $\le 1$) has the \defn{Urysohn extension property} if for every finite $F \subseteq X$ and $u \in \@E(F)$, there is some $x \in X$ such that $\delta(x)|F = u$; in other words, every consistent finite description of distances between points in $X$ to another point is already realized by a point in $X$.  The \defn{approximate Urysohn extension property} is defined in the same way, except we only require that such $x$ can be found with $\delta(x)|F$ arbitrarily close to $u$.  The \defn{Urysohn sphere} $\#U$ is the unique-up-to-isometry Polish metric space of diameter $\le 1$ with the Urysohn extension property, or equivalently the approximate Urysohn extension property.  See \cite{Me1}, \cite{Usp} for more on $\#U$ (and its unbounded relative).

In \cite{Kat}, Katětov showed that starting from any separable metric space $X$ (of diameter $\le 1$), the completed direct limit
\begin{align*}
\^{\@E^{\prime\infty}}(X) := \reallywidehat{\injlim_n \@E^{\prime n}(X)}
\end{align*}
of the sequence $X --->{\delta_X} \@E'(X) --->{\delta_{\@E'(X)}} \@E'(\@E'(X)) =: \@E^{\prime2}(X) --->{\delta_{\@E^{\prime2}(X)}} \dotsb$ is an isometric copy of $\#U$.  Indeed, it is easily seen that $\^{\@E^{\prime\infty}}(X)$ satisfies the approximate Urysohn extension property; and it will follow from \cref{thm:sbmsp-urysohn-approx} below that this implies the Urysohn extension property.  Let
\begin{align*}
\delta^n = \delta^n_X : \@E^{\prime n}(X) `-> \^{\@E^{\prime\infty}}(X)
\end{align*}
be the canonical embedding into the direct limit.  Clearly, $\^{\@E^{\prime\infty}}$ extends to a functor with respect to Lipschitz maps $f : X -> Y$, where $\^{\@E^{\prime\infty}}(f) : \^{\@E^{\prime\infty}}(X) -> \^{\@E^{\prime\infty}}(Y)$ is the unique Lipschitz map such that $\^{\@E^{\prime\infty}}(f) \circ \delta^n_X = \delta^n_Y \circ \@E^{\prime n}(f)$ for each $n$, i.e., making the following diagram commute:
\begin{equation*}
\begin{tikzcd}
X \dar["f"'] \rar["\delta_X"] &
\@E'(X) \dar["\@E'(f)"'] \rar["\delta_{\@E'(X)}"] &
\@E^{\prime2}(X) \dar["\@E^{\prime2}(f)"'] \rar["\delta_{\@E^{\prime2}(X)}"] &
\dotsb \rar["\delta^n_X"] &
\^{\@E^{\prime\infty}}(X) \dar["\^{\@E^{\prime\infty}}(f)"] \\
Y \rar["\delta_Y"'] &
\@E'(Y) \rar["\delta_{\@E'(Y)}"'] &
\@E^{\prime2}(Y) \rar["\delta_{\@E^{\prime2}(Y)}"'] &
\dotsb \rar["\delta^n_Y"'] &
\^{\@E^{\prime\infty}}(Y)
\end{tikzcd}
\end{equation*}

\begin{lemma}[see {\cite[3.2]{GK}}]
\label{thm:katetov*-ff}
Let $X, Y$ be metric spaces and $g : \^{\@E^{\prime\infty}}(X) \cong \^{\@E^{\prime\infty}}(Y)$ be an isometric bijection satisfying $g(\im(\delta^n_X)) = \im(\delta^n_Y)$ for all $n$.  Then there is a unique isometric bijection $f : X \cong Y$ such that $g = \^{\@E^{\prime\infty}}(f)$.
\end{lemma}
\begin{proof}
By definition of direct limit,
$g = \^{\@E^{\prime\infty}}(f)$
is equivalent to
\begin{align*}
g \circ \delta^n_X
= \^{\@E^{\prime\infty}}(f) \circ \delta^n_X
= \delta^n_Y \circ \@E^{\prime n}(f)
\tag{$*$}
\end{align*}
for each $n$.  For each $n$, since $g(\im(\delta^n_X)) = \im(\delta^n_Y)$, there is a unique $f_n : \@E^{\prime n}(X) \cong \@E^{\prime n}(Y)$ such that
\begin{align*}
g \circ \delta^n_X = \delta^n_Y \circ f_n.
\end{align*}
By ($*$), $f$ must be given by $f := f_0$; we prove by induction that $f_n = \@E^{\prime n}(f)$, whence ($*$) holds, for all $n$.  Suppose $f_n = \@E^{\prime n}(f)$.  Then
$\delta^{n+1}_Y \circ f_{n+1} \circ \delta_{\@E^{\prime n}(X)}
= g \circ \delta^{n+1}_X \circ \delta_{\@E^{\prime n}(X)}
= g \circ \delta^n_X
= \delta^n_Y \circ \@E^{\prime n}(f)
= \delta^{n+1}_Y \circ \delta_{\@E^{\prime n}(Y)} \circ \@E^{\prime n}(f)$,
whence $f_{n+1} \circ \delta_{\@E^{\prime n}(X)} = \delta_{\@E^{\prime n}(Y)} \circ \@E^{\prime n}(f)$ since $\delta^{n+1}_Y$ is injective, whence $f_{n+1} = \@E'(\@E^{\prime n}(f)) = \@E^{\prime n+1}(f)$ by \cref{lm:katetov-ff}.
\end{proof}

For a fiberwise separable Borel metric space $p : A -> X$ over standard Borel $X$, we define
\begin{align*}
\^{\@E_X^{\prime\infty}}(A) := \reallywidehat{\injlim_n \@E_X^{\prime n}(A)}
\end{align*}
for the sequence of Borel injections $A \lhook\joinrel--->{\delta_A} \@E'_X(A) \lhook \joinrel--->{\delta_{\@E'_X(A)}} \dotsb$ (where $\delta_A, \delta_{\@E'_X(A)}, \dotsc$ are given fiberwise).  This is a fiberwise separable Borel complete metric space over $X$, which is fiberwise isometric to $\#U$.

We now prove that the uniqueness of $\#U$ holds in a uniformly Borel way.  This is based on the following uniformly Borel version of the proof that the approximate Urysohn extension property implies the Urysohn extension property in a complete metric space; see e.g., \cite[pp.~80--81]{Gro}.

\begin{lemma}
\label{thm:sbmsp-urysohn-approx}
Let $p : A -> X$ be a fiberwise separable Borel complete metric space over $X$.  Suppose that each fiber $A_x$ satisfies the approximate Urysohn extension property (i.e., is isometric to $\#U$).  Then $A$ satisfies the Urysohn extension property, uniformly across fibers: for any Borel fiberwise finite subset $F \subseteq A$ and Borel fiberwise Katětov function $u : F -> \#I$, there is a full Borel section $T \subseteq A$ such that $u|F_x = \iota(T_x)|F_x$ for each $x \in X$.
\end{lemma}
\begin{proof}
Let $\@S = \{S_0, S_1, \dotsc\}$ be a countable family of fiberwise dense full Borel sections in $A$.  We will find full Borel sections $T_0, T_1, \dotsc \subseteq A$ such that
\begin{enumerate}
\item[(i)]  $\abs{u(a) - d((T_n)_x, a)} \le 2^{-n}$ for all $a \in F_x$;
\item[(ii)]  $d((T_n)_x, (T_{n+1})_x) \le 2^{-n} + 2^{-n-1}$ for all $x \in p(A)$.
\end{enumerate}
Let $T_0 := S_0$.  Given $T_n$, define
\begin{align*}
v : F \cup T_n &--> \#I \\
a &|--> \begin{cases}
u(a) &\text{if $a \in F$}, \\
\bigvee_{b \in F_x} \abs{u(b) - d((T_n)_x, b)} &\text{if $a \in (T_n \setminus F)_x$}.
\end{cases}
\end{align*}
Then $v$ is Borel fiberwise Katětov: it is clearly fiberwise Katětov restricted to $F$; and for $a \in F_x$ and $a' \in (T_n \setminus F)_x$, we have
\begin{align*}
v(a') &= \bigvee_{b \in F_x} \abs{u(b) - d((T_n)_x, b)}
= \bigvee_{b \in F_x} \abs{u(b) - d(a', b)} \\
&\ge \abs{u(a) - d(a', a)} \ge u(a) - d(a', a) = v(a) - d(a', a), \\
v(a') &= \bigvee_{b \in F_x} \abs{u(b) - d(a', b)} \\
&\le u(a) + d(a, a') \qquad(*) \\
&\le v(a) + d(a, a'), \\
v(a) + v(a')
&= u(a) + \bigvee_{b \in F_x} \abs{u(b) - d(a', b)}
\ge u(a) + \abs{u(a) - d(a', a)} \\
&\ge u(a) - (d(a', a) - u(a)) = d(a', a),
\end{align*}
where ($*$) is because
\begin{align*}
\begin{aligned}
u(b)
&\le u(a) + d(a, b) &&\text{because $u$ is Lipschitz} \\
&\le u(a) + d(a, a') + d(a', b), \\
d(a', b)
&\le u(a') + u(b) &&\text{because $u$ is Katětov} \\
&\le u(a) + d(a, a') + u(b) &&\text{because $u$ is Lipschitz}.
\end{aligned}
\end{align*}
Thus, since each $A_x$ satisfies the approximate Urysohn extension property and $\@S$ is fiberwise dense, we may define in a Borel way for each $x \in p(A)$
\begin{align*}
k(x) := \text{the least $k$ s.t.\ } \abs{v(a) - d((S_k)_x, a)} \le 2^{-n-1} \text{ for all $a \in (F \cup T_n)_x$}
\end{align*}
and put $T_{n+1} := \bigcup_{x \in p(A)} (S_{k(x)})_x$.  Then clearly (i) holds for $T_{n+1}$.  To check (ii), we have
\begin{align*}
d((T_n)_x, (T_{n+1})_x)
&= d((T_n)_x, (S_{k(x)})_x) \\
&\le v((T_n)_x) + 2^{-n-1} \qquad\text{by definition of $k(x)$};
\end{align*}
if (the unique element of) $(T_n)_x \in F_x$, then the last line is $u((T_n)_x) + 2^{-n-1} \le 2^{-n} + 2^{-n-1}$ by (i), otherwise it is $\bigvee_{b \in F_x} \abs{u(b) - d((T_n)_x, b)} + 2^{-n-1} \le 2^{-n} + 2^{-n-1}$ again by (i).  This completes the definition of $T_n$ for each $n$.  Now by (ii), $(T_0, T_1, \dotsc)$ forms a Cauchy sequence in each fiber; let $T$ be their limits.  Then (i) ensures that $u|F_x = \iota(T_x)|F_x$ for each $x \in p(A)$.
\end{proof}

\begin{proposition}
\label{thm:sbmsp-urysohn-unique}
Let $p : A -> X$ be a fiberwise separable Borel complete metric space over $X$ which is fiberwise isometric to $\#U$.  Then there is a Borel fiberwise isometry $f : A -> \#U$ (i.e., each $f_x := f|A_x : A_x -> \#U$ is an isometry).
\end{proposition}
\begin{proof}
It is enough to find a Borel fiberwise partial isometry $A \rightharpoonup X \times \#U$ between fiberwise countable dense sets over $X$.  This may be done via the usual back-and-forth construction, using countable families $\@S = \{S_0, S_1, \dotsc\}$ and $\@T = \{T_0, T_1, \dotsc\}$ of fiberwise dense full Borel sections in $A, X \times \#U$ respectively, and using \cref{thm:sbmsp-urysohn-approx} to extend the partial isometry at each stage.
\end{proof}

\subsection{Uniformizing fiberwise metric structures}
\label{sec:sbmstr-unif}

\begin{theorem}
\label{thm:sbmstr-isogpd-ff}
For any countable (multi-sorted) relational language $\@L$, standard Borel space $X$, and fiberwise separable Borel metric $\@L$-structure $\@M$ over $X$, there is a countable single-sorted relational language $\@L'$ and a full and faithful Borel functor $\Iso_X(\@M) -> \Iso(\#U) \ltimes \Mod_\#U(\@L')$.
\end{theorem}

As in \cref{sec:cbstr-unif}, we split the proof into two parts:

\begin{proposition}
\label{thm:sbmstr-ext}
For any countable (multi-sorted) relational language $\@L$, standard Borel space $X$, and fiberwise separable Borel metric $\@L$-structure $\@M$ over $X$, there is a countable single-sorted relational language $\@L'$, a fiberwise separable Borel metric $\@L'$-structure $\@M'$ over $X$ which is fiberwise isometric to $\#U$, and a Borel identity-on-objects isomorphism of groupoids $\Iso_X(\@M) \cong \Iso_X(\@M')$.
\end{proposition}
\begin{proof}
Let
\begin{align*}
\@L'_r := \@L_r \sqcup \@L_s \sqcup \{E_0, E_1, \dotsc\}
\end{align*}
where each $P \in \@L_s$ is treated as a unary relation symbol, as is each $E_i$.  The underlying fiberwise separable Borel complete metric space of $\@M'$ is
\begin{align*}
M' &:= \^{\@E^{\prime\infty}_X}(M_0) \\
\text{where}\quad M'_0 &:= \bigsqcup_{P \in \@L_s} P^\@M.
\end{align*}
Let $\delta^i = \delta^i_{M'_0} : \@E^{\prime i}_X(M'_0) `-> M'$ be the canonical isometric embeddings defined fiberwise as in \cref{sec:urysohn}.  Each $(P_0, \dotsc, P_{n-1})$-ary $R \in \@L_r$ is interpreted as the saturated grey image (easily seen to be Borel using a countable family of fiberwise dense Borel sections)
\begin{align*}
R^{\@M'} := [(\delta^0|P_0^\@M \times_X \dotsb \times_X \delta^0|P_{n-1}^\@M)(R^\@M)] \sqle M^{\prime n}_X,
\end{align*}
each $P \in \@L_s$ is interpreted as
\begin{align*}
P^{\@M'} := [\delta^0(P^\@M)] \sqle M',
\end{align*}
and each $E_i$ is interpreted as
\begin{align*}
E_i^{\@M'} := [\im(\delta^i)] \sqle M'.
\end{align*}
Using by \cref{thm:katetov*-ff}, $\@L$-isomorphisms $f = (f_P)_{P \in \@L_s} : \@M_x \cong \@M_y$ are easily seen to be in canonical bijection with $\@L'$-isomorphisms $g := \^{\@E^{\prime\infty}}(\bigsqcup_P f_P) : \@M'_x \cong \@M'_y$ (preservation of the relations $E_i$ amounts to $g(\im(\delta_{(M'_0)_x}^i)) = \im(\delta_{(M'_0)_y}^i)$ as required by \cref{thm:katetov*-ff}).
\end{proof}

\begin{proposition}
\label{thm:sbmstr-unif}
For any countable single-sorted relational language $\@L$, standard Borel space $X$, and fiberwise separable Borel metric $\@L$-structure $\@M$ over $X$ which is fiberwise isometric to $\#U$, there is a full and faithful Borel functor $F : \Iso_X(\@M) -> \Iso(\#U) \ltimes \Mod_\#U(\@L)$.
\end{proposition}
\begin{proof}
By \cref{thm:sbmsp-urysohn-unique}, there is a Borel fiberwise isometry $f : M -> \#U$.  Define $F$ on objects by
\begin{align*}
F : X &--> \Mod_\#U(\@L) \\
x &|--> (f_x)_*(\@M_x)
\end{align*}
and on morphisms by
\begin{align*}
F : \Iso_X(\@M) &--> \Iso(\#U) \ltimes \Mod_\#U(\@L) \\
(g : \@M_x -> \@M_y) &|--> (f_y \circ g \circ f_x^{-1}, F(x)).
\end{align*}
This is Borel, since letting $\@S$ be a countable family of fiberwise dense full Borel sections in $M$, for any $u, v \in \#U$ and $r > 0$ we have
\begin{align*}
d((f_y \circ g \circ f_x^{-1})(u), v) < r
&\iff \exists S, T \in \@S,\, s, t \in \#Q^+ \left(\begin{aligned}
&s + 2t < r \AND \\
&f_x^{-1}(u) \in [S]_{<t} \AND f_y^{-1}(v) \in [T]_{<t} \AND \\
&g \in \sqsqbr{S |-> [T]_{<r}}
\end{aligned}\right).
\qedhere
\end{align*}
\end{proof}

\begin{proof}[Proof of \cref{thm:sbmstr-isogpd-ff}]
Compose $G : \Iso_X(\@M) \cong \Iso_X(\@M')$ from \cref{thm:sbmstr-ext} with $F : \Iso_X(\@M') -> \Iso(\#U) \ltimes \Mod_\#U(\@L')$ from \cref{thm:sbmstr-unif}.
\end{proof}

\begin{remark}
\label{rmk:sbmstr-isogpd-ff-io}
As in \cref{rmk:cbstr-isogpd-ff-io}, we may arrange for the functor produced by \cref{thm:sbmstr-isogpd-ff} to be not only full and faithful, but also injective on objects.
First, we modify the proof of \cref{thm:sbmstr-ext}, replacing $\@L'$ with $\@L' \sqcup \{C_0, C_1, \dotsc\}$ and $M'_0$ with $M'_0 \sqcup (X \times \#N)$ where $(x, i) \in X \times \#N$ has distance $1$ from all other points and is thought of as a newly added constant, as in \cref{thm:cbstr-ext}, with $C_i$ interpreted in $\@M'$ as $[\delta^0(X \times \{i\})] \sqle M'$.
Let $G, F$ be as above and $f_x$ be as in the proof of \cref{thm:sbmstr-unif} (applied to the modified $\@M'$).  For each $x \in X$, $D_x := \{f_x(\delta^0(x, i))\}_{i \in \#N}$ is a countably infinite, metrically discrete (i.e., distinct points have distance $1$) subset of $\#U$; and the map $x |-> D_x$ is easily seen to be a Borel map $X -> \@F(\#U)$.
Let $u : X -> \#S_\infty$ be an injective Borel map.
By the Kuratowski--Ryll-Nardzewski selection theorem \cite[12.13]{Kcdst}, for each infinite $D \in \@F(\#U)$, there is a map $v_D : \#N -> D$ enumerating a countable dense subset of $D$, with $(D, i) |-> v_D(i)$ Borel; we may easily modify $v_D$ so that it is injective, so in particular $v_D : \#N -> D$ is a bijection whenever $D \subseteq \#U$ is metrically discrete.  Define for each $x \in X$
\begin{align*}
h_x : (M'_0)_x &\cong (M'_0)_x \\
(x, i) &|-> (x, (f_x \circ \delta^0)^{-1}(v_{D_x}(u(x)(i)))) \\
a &|-> a \qquad\text{for all other $a \in (M'_0)_x$},
\end{align*}
and
\begin{align*}
f'_x := f_x \circ \^{\@E^{\prime\infty}}(h_x) : M_x \cong \#U,
\end{align*}
so that
\begin{align*}
f'_x(\delta^0(x, i))
&= f_x(\^{\@E^{\prime\infty}}(h_x)(\delta^0(x, i)) \\
&= f_x(\delta^0(h_x(x, i))) \\
&= v_{D_x}(u(x)(i)).
\end{align*}
Then $F'$ obtained from \cref{thm:sbmstr-unif} by replacing $f_x$ with $f'_x$ is injective on objects, for similar reasons as in \cref{rmk:cbstr-isogpd-ff-io}.
\end{remark}

\section{$\sigma$-locally Polish groupoids}
\label{sec:locpolgpd}

In this section, we show that every open $\sigma$-locally Polish groupoid $G$ is canonically isomorphic to the isomorphism groupoid of a metric-étale structure over $G^0$, completing the proof of \cref{thm:intro-locpolgpd-rep}.

\subsection{Metrization}
\label{sec:locpolgpd-metriz}

Let $G$ be a groupoid.  Recall the definition of \defn{strict grey subgroupoid} $U \sqle G$ from \cref{sec:greygpd}: $U|G^0$ is $\{0, 1\}$-valued, and
\begin{align*}
U(g) < 1 &\implies U(\sigma(g)) = U(\tau(g)) = 0, &
U(g^{-1}) &= U(g), &
U(g \cdot h) &\le U(g) \dotplus U(h).
\end{align*}
Such $U$ determines a fiberwise pseudometric on $\tau : \sigma^{-1}(U_{=0}) -> G^0$ given by
\begin{gather*}
d_U := (\sigma^{-1}(U_{=0}) \times_{G^0} \sigma^{-1}(U_{=0}) --->[(g, h) |-> g^{-1} \cdot h]{} G)^{-1}(U) \sqle \sigma^{-1}(U_{=0}) \times_{G^0} \sigma^{-1}(U_{=0}) \\
d_U(g : y -> x, h : z -> x) = U(g^{-1} \cdot h),
\end{gather*}
which is \defn{left-invariant}: given $g : y -> x$, $h : z -> x$, and $k : x -> x'$ with $y, z \in U_{=0}$, we have $d_U(k \cdot g, k \cdot h) = d_U(g, h)$.  Conversely, given any $V \subseteq G^0$ and left-invariant fiberwise pseudometric $d \sqle \sigma^{-1}(V) \times_{G^0} \sigma^{-1}(V)$, $U_d(g) := d(\tau(g), g)$ is easily seen to define a strict grey subgroupoid $U_d \sqle G$ with $G^0 \cap (U_d)_{=0} = V$; and the operations $U |-> d_U$ and $d |-> U_d$ give inverse bijections
\begin{align*}
\{\text{strict grey subgroupoids } U \sqle G\} \cong \{\text{left-invariant pseudometrics } d \sqle \sigma^{-1}(V) \times_{G^0} \sigma^{-1}(V)\}.
\end{align*}
Note that when $G$ is a topological groupoid and $U \sqle G$ is open, so is $d_U \sqle \sigma^{-1}(U_{=0}) \times_{G^0} \sigma^{-1}(U_{=0})$.

The classical \defn{Birkhoff--Kakutani metrization theorem} states that every topological group admits ``enough'' continuous left-invariant pseudometrics, i.e., enough open grey subgroups.  See e.g., \cite[2.1.1]{Gao}.  A crucial ingredient in the proof is the fact, immediate from continuity of multiplication, that any neighborhood of the identity contains the square of a smaller neighborhood.  Ramsay \cite[pp.~361--362]{Ram} proved a nontrivial analog of this fact for paracompact groupoids $G$, where both neighborhoods must contain all of $G^0$; this allows the groupoid analog of the proof of Birkhoff--Kakutani to be carried out \cite[3.13--14]{Bun}.  Since we are dealing with \emph{$\sigma$-locally} Polish groupoids, we need slight generalizations of these results; we give the proofs for the sake of completeness.

Recall that a Hausdorff space $X$ is \defn{paracompact} if every open cover $\@U$ of $X$ has a \defn{star-refinement}, meaning an open cover $\@V$ such that for every $x \in X$, $\bigcup \{V \in \@V \mid x \in V\}$ is contained in some $U \in \@U$; in particular, any two intersecting sets in $\@V$ are contained in a single set in $\@U$.  Every metrizable (e.g., Polish) space is paracompact.  See e.g., \cite{Kelley}.

\begin{lemma}[Ramsay]
\label{thm:ramsay}
Let $G$ be a topological groupoid.  For any open $U \subseteq G$ and unital open Hausdorff paracompact $V \subseteq G$ with $G^0 \cap V \subseteq U$, there is a symmetric unital open $W \subseteq U$ with $W \cdot W \subseteq U$ and $G^0 \cap W = G^0 \cap V$.
\end{lemma}
\begin{proof}
For any $g \in V$, we have $g^{-1} \cdot g = \sigma(g) \in G^0 \cap V \subseteq U$ since $V$ is unital, so by continuity there is an open $A \ni g$ such that $A^{-1} \cdot A \subseteq U$.  Thus $\@A := \{\text{open } A \subseteq V \mid A^{-1} \cdot A \subseteq U\}$ forms an open cover of $V$, hence has a star-refinement $\@B$.  Let
\begin{align*}
\@W := \{\{g \mid g, g^{-1}, \sigma(g), \tau(g) \in U \cap B\} \mid B \in \@B\},
\end{align*}
i.e., $\@W$ consists of the largest symmetric unital subsets of $U \cap B$ for each $B \in \@B$, and put $W := \bigcup \@W$.  Then clearly $W \subseteq U \cap V$, and any $x \in G^0 \cap V$ is in $W$ as witnessed by any $B \in \@B$ with $x \in B$.  For $W_1, W_2 \in \@W$ as witnessed by $B_1, B_2 \in \@B$, we have $W_1 \cdot W_2 \subseteq U$, since if $W_1 \cdot W_2 \ne \emptyset$, then there are $g : y -> z \in W_1$ and $h : x -> y \in W_2$, whence $y \in W_1 \cap W_2 \subseteq B_1 \cap B_2$ since $W_1, W_2$ are unital, whence there is some $A \in \@A$ such that $B_1, B_2 \subseteq A$ since $\@B$ is a star-refinement of $\@A$, whence $W_1 \cdot W_2 = W_1^{-1} \cdot W_2 \subseteq B_1^{-1} \cdot B_2 \subseteq A^{-1} \cdot A \subseteq U$.  Thus $W \cdot W \subseteq U$.
\end{proof}

\begin{proposition}[Birkhoff--Kakutani metrization theorem for locally paracompact groupoids]
\label{thm:birkhoff-kakutani}
Let $G$ be an open topological groupoid.  For any open $U \sqle G$ and unital open Hausdorff paracompact $V \subseteq G$ with $G^0 \cap V \subseteq U_{=0}$, there is an open strict grey subgroupoid $W \sqle U$ with $G^0 \cap W_{=0} = G^0 \cap V$.
\end{proposition}
\begin{proof}
We may assume $U(g) \le 1/2$ for all $g \in G$; otherwise, replace $U$ with $U/2$, find $W$, and replace $W$ with $W \dotplus W$.  Using \cref{thm:ramsay}, find recursively $V_0 \supseteq V_1 \supseteq \dotsb$ symmetric unital open subsets of $V$ with
\begin{align*}
V_{n+1} \cdot V_{n+1} \cdot V_{n+1} \subseteq V_n, &&
G^0 \cap V_n = G^0 \cap V, &&
V_n \subseteq U_{<2^{-(n+1)}}.
\end{align*}
Put (recall the notation from \cref{sec:greygpd})
\begin{align*}
V' &:= \bigsqcup_n (2^{-n} \cdot V_n) \\
V'(g) &= \bigwedge \{2^{-n} \mid V_n \ni g\}, \\
W &:= \ang{V'}
= V' \sqcup (V' \odot V') \sqcup (V' \odot V' \odot V') \sqcup \dotsb \\
W(g) &= \bigwedge_{g = g_0 \cdot \dotsb \cdot g_{m-1}; m \ge 1} (V'(g_0) + \dotsb + V'(g_{m-1})).
\end{align*}
Clearly, $W$ is an open strict grey subgroupoid; and we have $G^0 \cap W_{=0} = G^0 \cap V'_{=0} = G^0 \cap V$, since each $V_n$ is a unital subset of $V$ and $G^0 \cap V_0 = G^0 \cap V$.  It remains to check that $W \sqle U$.  Note that
\begin{align*}
V'(g \cdot h \cdot k) \le 2 \max \{V'(g), V'(h), V'(k)\}  \tag{$*$}
\end{align*}
for all composable $g, h, k \in G$ (where $V'(g) = 1$ if $g \not\in \bigcup_n V_n$), using $V_{n+1} \cdot V_{n+1} \cdot V_{n+1} \subseteq V_n$.  We now prove by induction that for all $m \ge 0$ and all composable sequences $x_m --->{g_{m-1}} \dotsb --->{g_0} x_0 \in G$ with all the $x_i \in V$, we have
\begin{align*}
V'(g_0 \dotsm g_{m-1}) \le 2(V'(g_0) + \dotsb + V'(g_{m-1})).  \tag{$\dagger$}
\end{align*}
When $m = 0$ we get $V'(x_0) = 0$ since $x_0 \in V$.  For $m > 0$, let $k < m$ be greatest such that
\begin{align*}
V'(g_0) + \dotsb + V'(g_{k-1}) \le \frac12(V'(g_0) + \dotsb + V'(g_{m-1}));  \tag{$\ddagger_0$}
\end{align*}
then
\begin{align*}
V'(g_{k+1}) + \dotsb + V'(g_{m-1}) \le \frac12(V'(g_0) + \dotsb + V'(g_{m-1})),  \tag{$\ddagger_1$}
\end{align*}
whence
\begin{align*}
\begin{aligned}
V'(g_0 \dotsm g_{m-1})
&\le 2 \max \{V'(g_0 \dotsm g_{k-1}), V'(g_k), V'(g_{k+1} \dotsm g_{k-1})\} &&\text{by ($*$)} \\
&\le 2 \max \{2(V'(g_0) + \dotsb + V'(g_{k-1})), V'(g_k), 2(V'(g_{k+1}) + \dotsb + V'(g_{k-1}))\} &&\text{by IH} \\
&\le 2 \max \{V'(g_0) + \dotsb + V'(g_{m-1}), V'(g_k), V'(g_0) + \dotsb + V'(g_{m-1})\} &&\text{by ($\ddagger_i$)} \\
&= 2(V'(g_0) + \dotsb + V'(g_{m-1})).
\end{aligned}
\end{align*}
This proves ($\dagger$) when all $x_i \in V$; clearly ($\dagger$) also holds when $m \ge 1$ and some $x_i \not\in V$ (the right-hand side is then $\ge 2$).  So for all $g \in G$,
\begin{align*}
W(g) = \bigwedge_{g = g_0 \cdot \dotsb \cdot g_{m-1}; m \ge 1} (V'(g_0) + \dotsb + V'(g_{m-1}))
\ge \frac12 V'(g).
\end{align*}
If $g \in V_0$, then for any $n \ge 0$ with $V_n \ni g$, we have $g \in V_n \subseteq U_{<2^{-(n+1)}}$, whence $U(g) < \frac12 2^{-n}$; thus $W(g) \ge \frac12 V'(g) = \frac12 \bigwedge \{2^{-n} \mid V_n \ni g\} \ge U(g)$.
Otherwise, $W(g) \ge \frac12 V'(g) = \frac12 \ge U(g)$.
\end{proof}

\subsection{Metric-étale left completions}

Let $G$ be an open topological groupoid.  For an open strict grey subgroupoid $U \sqle G$, by \cref{thm:metale-quotient} we have a metric-étale topometric quotient space
\begin{align*}
\^{G/U} := \reallywidehat{\sigma^{-1}(U_{=0})/d_U}
\end{align*}
over $G^0$, whose projection map (descended along the quotient map $\pi = \pi_U : \sigma^{-1}(U_{=0}) -> \^{G/U}$ from $\tau : \sigma^{-1}(U_{=0}) -> G^0$) we continue to denote by $\tau : \^{G/U} -> G^0$.  When $G$ is a topological group, $\^{G/U}$ is simply its left completion with respect to a left-invariant metric.

Note that the grey saturation $d_U[S]$ of $S \sqle \sigma^{-1}(U_{=0})$ is given by
\begin{align*}
d_U[S](g)
= \bigwedge_{h \in \sigma^{-1}(U_{=0})_{\tau(g)}} (S(h) \dotplus U(h^{-1} \cdot g))
= \bigwedge_{g = h \cdot k} (S(h) \dotplus U(k))
= (S \odot U)(g).
\end{align*}
We call $S$ \defn{$U$-invariant} if it is $d_U$-invariant, i.e., if $S \odot U \sqle S$.  Such $S$ are in bijection with invariant grey subsets of $\^{G/U}$.

We call open $S \subseteq \sigma^{-1}(U_{=0})$ \defn{$U_{<r}$-small} if it is $(d_U)_{<r}$-small over $G^0$, i.e., if $S \times_{G^0} S \subseteq (d_U)_{<r}$, i.e., $S^{-1} \cdot S \subseteq U_{<r}$; such $S$ can be thought of as an approximate open section, and induce a (slightly larger) approximate open section $\^d_U[\pi_U(S)]_{<s}$ in $\^{G/U}$ for each $s > 0$.  We always have at least one genuine continuous (partial) section $\pi = \pi_U : U_{=0}^0 -> \^{G/U}$ of $\tau : \^{G/U} -> G^0$, the \defn{unit section}.

The left multiplication action of $G$ on itself (restricted to $\sigma^{-1}(U_{=0}) \subseteq G$) is $d_U$-isometric, hence descends to an isometric action $\alpha : G \curvearrowright \^{G/U}$ defined on the dense subset $\im(\pi) \subseteq \^{G/U}$ by
\begin{align*}
g \cdot \pi(h) := \pi(g \cdot h).
\end{align*}
This action is continuous: for a $\^d_U$-invariant open grey $S \sqle \^{G/U}$, its preimage $\alpha^{-1}(S) \sqle G \times_{G^0} \^{G/U}$ obeys $(1_G \times_{G^0} \pi)^{-1}(\alpha^{-1}(S)) = \mu^{-1}(\pi^{-1}(S)) \sqle G \times_{G^0} \sigma^{-1}(U_{=0})$, which is open as well as invariant under the sum of the discrete metric $\Delta_G$ on $G$ and $d_U$ on $\sigma^{-1}(U_{=0})$ (by left-invariance of $d_U$), whence $\alpha^{-1}(S) \sqle G \times_{G^0} \^{G/U}$ is invariant open (since $G \times_{G^0} \^{G/U}$ can be identified with the topometric quotient of $G \times_{G^0} \sigma^{-1}(U_{=0})$ by \cref{thm:topomet-product}); by invariance of the topology on $\^{G/U}$, this is enough.

Now consider two open strict grey subgroupoids $U, V \sqle G$.  If $U \sqle V$ (so in particular $\sigma^{-1}(U_{=0}) \subseteq \sigma^{-1}(V_{=0})$), the restriction of the quotient map $\pi_V : \sigma^{-1}(V_{=0}) -> \^{G/V}$ to $\sigma^{-1}(U_{=0})$ descends along $\pi_U : \sigma^{-1}(U_{=0}) -> \^{G/U}$ to a Lipschitz \defn{projection map}
\begin{align*}
\pi_{U,V} : \^{G/U} &--> \^{G/V}
\end{align*}
defined on the dense subset $\im(\pi_U) \subseteq \^{G/U}$ by
\begin{align*}
\pi_{U,V}(\pi_U(g)) := \pi_V(g).
\end{align*}
This map $\pi_{U,V}$ is continuous, since for $\^d_V$-invariant open grey $S \sqle \^{G/V}$, we have $\pi_U^{-1}(\pi_{U,V}^{-1}(S)) = \pi_V^{-1}(S)|\sigma^{-1}(U_{=0}) \sqle \sigma^{-1}(U_{=0})$ which is open, whence $\pi_{U,V}^{-1}(S) \sqle \^{G/U}$ is open.

In general, given $U, V$ as above and open $S \sqle \sigma^{-1}(V_{=0})$, define
\begin{gather*}
R_{U,V,S} := ((g, h) |-> g^{-1} \cdot h)^{-1}(S) \sqle \sigma^{-1}(U_{=0}) \times_{G^0} \sigma^{-1}(V_{=0}) \\
R_{U,V,S}(g, h) = S(g^{-1} \cdot h),
\end{gather*}
thought of as the ``grey graph'' of right multiplication by $S$ (when $S$ is an ordinary subset, we have $(g, h) \in R_{U,V,S} \iff g^{-1} \cdot h \in S \iff h \in g \cdot S$).  $R_{U,V,S}$ is open, and descends to an invariant open $\^R_{U,V,S} \sqle \^{G/U} \times_{G^0} \^{G/V}$ (by \cref{thm:topomet-product}) iff $R_{U,V,S}$ is invariant with respect to the sum of the pseudometrics $d_U, d_V$, which is easily seen
to be equivalent to $U \odot S \odot V \sqle S$; we call such $S$ \defn{$(U,V)$-invariant}.
Note also that $R_{U,V,S}$ only depends on $S|\tau^{-1}(U_{=0})$.

We will regard the grey relations $\^R_{U,V,S} \sqle \^{G/U} \times_{G^0} \^{G/V}$ for suitable $S$ as analogous to the right multiplication maps $f_{U,V,S}$ in the étale setting (\cref{thm:nonarchgpd-rightmult}).  Namely, consider an open $V_{<r}$-small $S \subseteq \sigma^{-1}(V_{=0})$, for some $r \in (0, 1)$, such that
\begin{align*}
U \sqle S \odot V \odot S^{-1}  \tag{$*$}
\end{align*}
(writing as usual $S$ for $\*0_S$).  Note that $S \odot V \odot S^{-1}|G^0 = \tau(S)$ is $\{0,1\}$-valued, hence \cref{thm:birkhoff-kakutani} will guarantee a plentiful supply of $U$ for each given $S$.  Now $U \odot S \odot V \sqle \sigma^{-1}(V_{=0})$ is clearly $(U, V)$-invariant open, whence we have an invariant open $\^R_{U,V,U \odot S \odot V} \sqle \^{G/U} \times_{G^0} \^{G/V}$ as above.  We have
\begin{align*}
\begin{aligned}
S \odot V|\tau^{-1}(U_{=0}) \sqle U \odot S \odot V
&\sqle S \odot V \odot S^{-1} \odot S \odot V &&\text{by ($*$)} \\
&\sqle S \odot V \odot (V \dotminus r) \odot V &&\text{since $S$ is $V_{<r}$-small} \\
&\sqle S \odot V \dotminus r
\end{aligned}
\end{align*}
(the last line using that for any $A, B \sqle G$, we have $(A \odot (B \dotminus r))(g) = \bigwedge_{g = h \cdot k} (A(h) \dotplus (B(k) \dotminus r)) \ge \bigwedge_{g = h \cdot k} (A(h) \dotplus B(k)) \dotminus r = (A \odot B \dotminus r)(g)$, which follows from $a + (b \dotminus r) \ge (a \dotplus b) - r$ for $a, b, r \in \#I$), which gives
\begin{align*}
R_{U,V,S \odot V} \sqle R_{U,V,U \odot S \odot V} \sqle R_{U,V,S \odot V} \dotminus r.  \tag{$\dagger$}
\end{align*}
Now from the definition of $R_{U,V,S \odot V}$, the fact (from above) that $S \odot V = d_V[S]$, and left-invariance of $d_V$, we have
\begin{align*}
R_{U,V,S \odot V}(g, h) = d_V(S, g^{-1} \cdot h) = d_V(g \cdot S, h);
\end{align*}
in other words, ($\dagger$) says that $R_{U,V,U \odot S \odot V}$ is approximately (up to $r$) the $d_V$-distance from the graph of right multiplication by $S$.  This implies a similar estimate for $\^R_{U,V,U \odot S \odot V}$, since the latter is defined by continuously extending $R_{U,V,U \odot S \odot V}$.  To summarize:

\begin{lemma}
\label{thm:locpolgpd-rightmult}
For any open strict grey subgroupoids $U, V \sqle G$ and open $V_{<r}$-small $S \subseteq \sigma^{-1}(V_{=0})$ for some $r > 0$ such that
\begin{align*}
U \sqle S \odot V \odot S^{-1},
\end{align*}
there is an invariant open grey relation $\^R_{U,V,U \odot S \odot V} \sqle \^{G/U} \times_{G^0} \^{G/V}$ which on the dense subset $\im(\pi_U) \times_{G^0} \im(\pi_V) \subseteq \^{G/U} \times_{G^0} \^{G/V}$ satisfies
\begin{align*}
d_V(g \cdot S, h) \ge \^R_{U,V,U \odot S \odot V}(\pi_U(g), \pi_V(h)) \ge d_V(g \cdot S, h) - r
\end{align*}
for all $(g, h) \in \sigma^{-1}(U_{=0}) \times_{G^0}\sigma^{-1}(V_{=0})$.

Thus by density and continuity, we have more generally for all $(a, b) \in \^{G/U} \times_{G^0} \^{G/V}$
\begin{align*}
\olddisplaystyle \liminf_{\pi_U(g) -> a} \^d_V(\pi_V(g \cdot S), b) \ge \^R_{U,V,U \odot S \odot V}(a, b) \ge \limsup_{\pi_U(g) -> a} \^d_V(\pi_V(g \cdot S), b) - r
\end{align*}
where $\liminf_{\pi_U(g) -> a} f(g) := \sup_{r > 0} \inf_{g \in \pi_U^{-1}([a]_{<r})} f(g)$, and similarly for $\limsup$.  \qed
\end{lemma}

\subsection{The canonical structure}

From now on, we assume that $G$ is an open $\sigma$-locally Polish groupoid.  Fix a countable family $\@U$ of open strict grey subgroupoids $U \sqle G$ and a countable basis $\@S$ of open sets $S \subseteq G$, such that the following hold:
\begin{enumerate}[leftmargin=3.2em]
\item[(US1)]  For any Polish $S \in \@S$, there is a $U \in \@U$ with $U \sqle S$ and $G^0 \cap U_{=0} = G^0 \cap S$ (using \cref{thm:birkhoff-kakutani}).
\item[(US2)]  For any $V \in \@U$, $S \in \@S$ with $S \subseteq \sigma^{-1}(V_{=0})$, and Polish $T \in \@S$ with $G^0 \cap T \subseteq \tau(S) = G^0 \cap (S \odot V \odot S^{-1})$, there is a $U \in \@U$ with $U \sqle S \odot V \odot S^{-1}$ and $G^0 \cap U_{=0} = G^0 \cap T$ (using \cref{thm:birkhoff-kakutani}).
\item[(US3)]  For any $U, V \in \@U$, $U \dotplus V \in \@U$.
\end{enumerate}
For each $U \in \@U$ and $r > 0$, let
\begin{align*}
\@S_U &:= \{S \in \@S \mid S \subseteq \sigma^{-1}(U_{=0})\}, \\
\@S_{U<r} &:= \{S \in \@S_U \mid S \text{ $U_{<r}$-small}\};
\end{align*}
these are bases of open sets in $\sigma^{-1}(U_{=0})$ (the latter because $d_U \sqle \sigma^{-1}(U_{=0})^2_{G^0}$ is open).

\begin{lemma}
\label{thm:locpolgpd-basis-section}
$\{d_U[S]_{<1} \mid U \in \@U \AND S \in \@S_U\}$ is a basis of open sets in $G$.
\end{lemma}
\begin{proof}
Let $V \subseteq G$ be open and $g : x -> y \in V$.  By continuity of multiplication, there are open $g \in W \subseteq G$ and Polish $x \in T \in \@S$ such that $W \cdot T \subseteq V$.  By (US1), there is $U \in \@U$ with $U \sqle T$ and $U_{=0} \supseteq G^0 \cap T \ni x$.  Let $g \in S \in \@S_U$ with $S \subseteq W$.  Then $g \in d_U[S]_{<1} = S \cdot U_{<1} \subseteq W \cdot T \subseteq V$.
\end{proof}

Let $\@M$ be the metric-étale structure over $G^0$ in the language $\@L$ consisting of a sort for each $U \in \@U$, interpreted in $\@M$ as $\^{G/U}$, and a binary relation symbol for each $U, V \in \@U$ and $S \in \@S_V$ satisfying $U \sqle S \odot V \odot S^{-1}$ as in \cref{thm:locpolgpd-rightmult}, interpreted in $\@M$ as the grey relation $\^R_{U,V,U \odot S \odot V} \sqle \^{G/U} \times_{G^0} \^{G/V}$ from there.  We call $\@M$ the \defn{canonical structure} of $G$ (with respect to $\@U, \@S$).  Since $G$ is second-countable and $\@U, \@S$ are countable, $\@M$ is second-countable.

\subsection{The homomorphism category}
\label{sec:locpolgpd-hom}

As in the non-Archimedean case (\cref{sec:nonarchgpd-hom}), we now give an explicit description of $\Hom_{G^0}(\@M)$ as the ``left completion'' of $G$.  This is conceptually similar to the non-Archimedean case, but will involve some tedious approximation arguments.

Fix objects $x, y \in G^0$.  Given a homomorphism $h : \@M_x -> \@M_y$, i.e., $h : x -> y \in \Hom_{G^0}(\@M)$, we have for each sort $U \in \@U$ a map $h_U : (\^{G/U})_x -> (\^{G/U})_y$, which when $x \in U_{=0}$ we may evaluate at the unit section $\pi_U(x) \in (\^{G/U})_x$, yielding an element $h_U(\pi_U(x)) \in (\^{G/U})_y$.
For $U \sqle V \in \@U$ with $x \in U_{=0}$, we have the projection map $\pi_{U,V} : (\^{G/U})_y -> (\^{G/V})_y$, functorial in $(U, V)$; so as before, we may form the inverse limit
\begin{align*}
\projlim_{U \in \@U; x \in U_{=0}} (\^{G/U})_y = \{\vec{a} = (a_U)_U \in \prod_{U \in \@U; x \in U_{=0}} (\^{G/U})_y \mid \forall U \sqle V\, (\pi_{U,V}(a_U) = a_V)\}.
\end{align*}

\begin{lemma}
\label{thm:locpolgpd-yoneda}
For all $x, y \in G^0$, the above-described map
\begin{align*}
\Phi : \Hom_{G^0}(\@M)(x, y) &--> \prod_{U \in \@U; x \in U_{=0}} (\^{G/U})_y \\
(h : \@M_x -> \@M_y) &|--> (h_U(\pi_U(x)))_U
\end{align*}
restricts to a bijection
\begin{align*}
\Phi : \Hom_{G^0}(\@M)(x, y) &\cong \projlim_{U \in \@U; x \in U_{=0}} (\^{G/U})_y
\end{align*}
with inverse defined on the dense subsets $\pi_U(\sigma^{-1}(U_{=0})_x) \subseteq (\^{G/U})_x$ by
\begin{align*}
\Psi : \projlim_{U \in \@U; x \in U_{=0}} (\^{G/U})_y &--> \Hom_{G^0}(\@M)(x, y) \\
\vec{a} = (a_U)_U &|--> \left(\begin{aligned}
\Psi(\vec{a})_U : (\^{G/U})_x &-> (\^{G/U})_y \\
\pi_U(g \in \sigma^{-1}(U_{=0})_x) &|-> \olddisplaystyle\lim_{\substack{
    g \in S \in \@S_U \\
    V \in \@U, x \in V_{=0}, V \sqle S \odot U \odot S^{-1} \\
    k \in \sigma^{-1}(V_{=0})_y, \pi_V(k) -> a_V
}} \pi_U(k \cdot S)
\end{aligned}\right)_U  \tag{$*$}
\end{align*}
where
\begin{align*}
\hspace{-1em}
\olddisplaystyle\lim_{\substack{
    g \in S \in \@S_U \\
    V \in \@U, x \in V_{=0}, V \sqle S \odot U \odot S^{-1} \\
    k \in \sigma^{-1}(V_{=0})_y, \pi_V(k) -> a_V
}} \pi_U(k \cdot S)
&:= \text{unique element of}
\olddisplaystyle\bigcap_{\substack{
    g \in S \in \@S_U \\
    V \in \@U, x \in V_{=0}, V \sqle S \odot U \odot S^{-1} \\
    r > 0
}} \-{\pi_U(\pi_V^{-1}([a_V]_{<r}) \cdot S)}.  \tag{$\dagger$}
\end{align*}
\end{lemma}
\begin{proof}
First, we check that $\Phi$ is injective.  Let $h : \@M_x -> \@M_y$ and $g \in \sigma^{-1}(U_{=0})_x$ for some $U \in \@U$.  Let $r > 0$, let $g \in S \in \@S_{U<r}$, and let $V \in \@U$ with $V \sqle S \odot U \odot S^{-1}$ and $x \in V_{=0}$ (using (US2) and that $x = \tau(g) \in \tau(S)$ is in some Polish $T \in \@S$ with $G^0 \cap T \subseteq \tau(S)$).  Then
\begin{align*}
\begin{aligned}[b]
0
&= d_U(S_x, g) &&\text{since $g \in S$} \\
&\ge \^R_{V,U,V \odot S \odot U}(\pi_V(x), \pi_U(g)) &&\text{by \cref{thm:locpolgpd-rightmult}} \\
&\ge \^R_{V,U,V \odot S \odot U}(h_V(\pi_V(x)), h_U(\pi_U(g))) &&\text{since $h$ is a homomorphism} \\
&\ge \limsup_{\pi_V(k) -> h_V(\pi_V(x))} \^d_U(\pi_U(k \cdot S), h_U(\pi_U(g))) - r  &&\text{by \cref{thm:locpolgpd-rightmult}}.
\end{aligned}
\tag{$\ddagger$}
\end{align*}
Since $S$ is $U_{<r}$-small, each $\pi_U(k \cdot S)$ has diameter $\le r$, so ($\ddagger$) shows that $h_U(\pi_U(g))$ is determined up to diameter $< 4r$ by $(h_V(\pi_V(x)))_V = \Phi(h)$.  Letting $r \searrow 0$ shows that $h_U(\pi_U(g))$ is determined by $\Phi(h)$, for every $g \in \sigma^{-1}(U_{=0})_x$; by density of $\pi_U(\sigma^{-1}(U_{=0})_x) \subseteq (\^{G/U})_x$ and continuity of $h$ (being a homomorphism), this means $h$ is determined by $\Phi(h)$, i.e., $\Phi$ is injective.

Next, we check that $\Phi$ lands in $\projlim_U (\^{G/U})_y$.  Let $h : \@M_x -> \@M_y$, $g \in \sigma^{-1}(U_{=0})_x$, and $U \sqle U' \in \@U$ with $x \in U_{=0}$.  For any $r > 0$, by ($\ddagger$) with $g := x$, $g \in S \in \@S_{U<r}$, and $V \sqle S \odot U \odot S^{-1}$ as above,
\begin{align*}
r
&\ge \limsup_{\pi_V(k) -> h_V(\pi_V(x))} \^d_U(\pi_U(k \cdot S), h_U(\pi_U(x))) \\
&\ge \limsup_{\pi_V(k) -> h_V(\pi_V(x))} \^d_{U'}(\pi_{U,U'}(\pi_U(k \cdot S)), \pi_{U,U'}(h_U(\pi_U(x))))  \quad\text{since $U \sqle U'$} \\
&= \limsup_{\pi_V(k) -> h_V(\pi_V(x))} \^d_{U'}(\pi_{U'}(k \cdot S), \pi_{U,U'}(h_U(\pi_U(x)))).
\end{align*}
On the other hand, by ($\ddagger$) with $U'$ in place of $U$ (and the same $S, V$, again since $U \sqle U'$),
\begin{align*}
r \ge \limsup_{\pi_V(k) -> h_V(\pi_V(x))} \^d_{U'}(\pi_{U'}(k \cdot S), h_{U'}(\pi_{U'}(x))).
\end{align*}
Again since $\pi_{U'}(k \cdot S)$ has diameter $\le r$, taking $r \searrow 0$ gives $\pi_{U,U'}(h_U(\pi_U(x))) = h_{U'}(\pi_{U'}(x))$.

Now we check that $\Psi$ is well-defined.  First, fix $\vec{a} \in \projlim_U (\^{G/U})_y$, $U \in \@U$, and $g \in \sigma^{-1}(U_{=0})_x$; we check that $\Psi(\vec{a})_U(\pi_U(g))$ as given by ($*$) is well-defined.  For this, we need to check that the limit formula ($\dagger$) makes sense, i.e., that
\begin{align*}
\left\{\pi_U(\pi_V^{-1}([a_V]_{<r}) \cdot S) \relmiddle| \substack{
    g \in S \in \@S_U \\
    V \in \@U, x \in V_{=0}, V \sqle S \odot U \odot S^{-1} \\
    r > 0
}\right\}  \tag{$\mathsection$}
\end{align*}
forms a Cauchy filter in $(\^{G/U})_y$.  The set of allowed $(S, V, r)$ is down-directed, using (US2) and (US3).
The sets $\pi_U(\pi_V^{-1}([a_V]_{<r}) \cdot S) \subseteq (\^{G/U})_y$ are clearly monotone in $S, r$;
they are also monotone in $V$, since if $V \sqle V'$ then
$\pi_{V'}^{-1}([a_{V'}]_{<r})
= \pi_V^{-1}(\pi_{V,V'}^{-1}([\pi_{V,V'}(a_V)]_{<r}))
\supseteq \pi_V^{-1}([a_V]_{<r})$
since $\vec{a} \in \projlim_U (\^{G/U})_y$ and $\pi_{V,V'}$ is Lipschitz,
and nonempty, since $\emptyset \ne \pi_V^{-1}([a_V]_{<r}) \subseteq \sigma^{-1}(V_{=0})$ and $V_{=0} \subseteq \tau(S)$ by density of $\im(\pi_V)$ and $V \sqle S \odot U \odot S^{-1}$.
Next, note that
\begin{enumerate}
\item[($\mathparagraph$)]  If $S$ in ($\mathsection$) is $U_{<s}$-small, then $\pi_V^{-1}([a_V]_{<r}) \cdot S$ is $U_{<2r+2s}$-small.
\end{enumerate}
Indeed, for any $k \cdot l, k' \cdot l' \in \pi^{-1}_V([a_V]_{<r}) \cdot S$ with $k, k' \in \pi_V^{-1}([a_V]_{<r})$ and $l, l' \in S$, we have
\begin{align*}
\begin{aligned}
2r
&> d_V(k, k')
= V(k^{-1} \cdot k') &&\text{since $k, k' \in \pi_V^{-1}([a_V]_{<r})$} \\
&\ge \bigwedge_{l'' \in S_{\sigma(k)}, l''' \in S_{\sigma(k')}} U(l^{\prime\prime-1} \cdot k^{-1} \cdot k' \cdot l''') &&\text{since $V \sqle S \odot U \odot S^{-1}$} \\
&= \bigwedge_{l'' \in S_{\sigma(k)}, l''' \in S_{\sigma(k')}} d_U(k \cdot l'', k' \cdot l''') \\
&\ge \bigwedge_{l'' \in S_{\sigma(k)}, l''' \in S_{\sigma(k')}} (d_U(k \cdot l, k' \cdot l') - d_U(l, l'') - d_U(l', l''')) &&\text{using left-invariance of $d_U$} \\
&\ge d_U(k \cdot l, k' \cdot l') - 2s &&\text{since $S$ is $U_{<s}$-small}.
\end{aligned}
\end{align*}
This proves ($\mathparagraph$), which implies that ($\mathsection$) has sets of arbitrarily small diameter, hence is a Cauchy filter, whence $\Psi(\vec{a})_U(\pi_U(g))$ is well-defined.

In order for $\Psi(\vec{a})$ thus defined on the dense subsets $\pi_U(\sigma^{-1}(U_{=0})_x) \subseteq (\^{G/U})_x$ to extend to a homomorphism $\@M_x -> \@M_y$, we must check that $\Psi(\vec{a})$ is Lipschitz and preserves the relations $\^R_{U,V,U \odot S \odot V}$ (on the dense subset).  This will follow from the limit formula ($\dagger$) and left-invariance of the metrics $d_U$ and the relations $R_{U,V,U \odot S \odot V}$.
To check Lipschitzness, let $g, g' \in \sigma^{-1}(U_{=0})_x$ and $r > 0$.
Since $d_U \sqle \sigma^{-1}(U_{=0})^2_{G^0}$ is open, there are $S, S' \in \@S_{U<r}$ such that $(g, g') \in S \times_{G^0} S' \subseteq (d_U)_{<d_U(g, g') + r}$.
By (US2) and (US3), there is $V \in \@U$ such that $x \in V_{=0}$ and $V \sqle S \odot U \odot S^{-1}, S' \odot U \odot S^{\prime-1}$, whence $\pi_U(\pi_V^{-1}([a_V]_{<r}) \cdot S), \pi_U(\pi_V^{-1}([a_V]_{<r}) \cdot S')$ belong to the respective Cauchy filters ($\mathsection$) defining $\Psi(\vec{a})_U(\pi_U(g)), \Psi(\vec{a})_U(\pi_U(g'))$; by ($\mathparagraph$), they are $4r$-small.  Now picking any $k \in \pi_V^{-1}([a_V]_{<r})$, $l \in S_{\sigma(k)}$, and $l' \in S'_{\sigma(k)}$, so that $\pi_U(k \cdot l) \in \pi_U(\pi_V^{-1}([a_V]_{<r}) \cdot S)$ and $\pi_U(k \cdot l') \in \pi_U(\pi_V^{-1}([a_V]_{<r}) \cdot S')$, we have
\begin{align*}
\begin{aligned}
\^d_U(\Psi(\vec{a})_U(\pi_U(g)), \Psi(\vec{a})_U(\pi_U(g')))
&\le \^d_U(\pi_U(k \cdot l), \mathrlap{\pi_U(k \cdot l')) + 8r} \\
&= d_U(l, l') + 8r &&\text{by left-invariance} \\
&< d_U(g, g') + 9r &&\text{since $(l, l') \in S \times_{G^0} S' \subseteq (d_U)_{<d_U(g, g') + r}$}.
\end{aligned}
\end{align*}
Taking $r \searrow 0$ shows that $\Psi(\vec{a})_U$ is Lipschitz on $\pi_U(\sigma^{-1}(U_{=0})_x) \subseteq (\^{G/U})_x$.  A similar argument shows that $\Psi(\vec{a})$ preserves the relations $\^R_{U,V,U \odot S \odot V}$.  Thus $\Psi(\vec{a})$ extends to a homomorphism $\@M_x -> \@M_y$.

It remains only to check that $\Phi(\Psi(\vec{a})) = \vec{a}$.  For each $U \in \@U$ with $x \in U_{=0}$, to compute $\Psi(\vec{a})_U(\pi_U(x))$, we may restrict to (the cofinally many, by (US3)) $V \sqle U$ in ($*$), for which $\pi_V(k) -> a_V$ implies $\pi_U(k) = \pi_{V,U}(\pi_V(k)) -> \pi_{V,U}(a_V) = \pi_U(a_U)$ since $\pi_{V,U}$ is Lipschitz and $\vec{a} \in \projlim_U (\^{G/U})_y$, whence $\pi_U(a_U) \in \pi_U(k \cdot S)$ for eventually all $k$ in ($*$) since $x \in S$; thus $\Psi(\vec{a})_U(\pi_U(x)) = a_U$.  So $\Phi(\Psi(\vec{a})) = (\Psi(\vec{a})_U(\pi_U(x)))_U = (a_U)_U = \vec{a}$, as desired.
\end{proof}

As before, we now use \cref{thm:locpolgpd-yoneda} to simplify the description of the topology on $\Hom_{G^0}(\@M)$:

\begin{lemma}
\label{thm:locpolgpd-yoneda-top}
The topology on $\Hom_{G^0}(\@M)$ is generated by the maps $\sigma, \tau : \Hom_{G^0}(\@M) -> G^0$ together with the sets
\begin{align*}
\sqsqbr{\pi_U(U_{=0}) |-> [\pi_U(T)]_{<r}}_U = \{h : \@M_x -> \@M_y \in \Hom_{G^0}(\@M) \mid x \in U_{=0} \AND \^d_U(h_U(\pi_U(x)), \pi_U(T)) < r\}
\end{align*}
for $U \in \@U$, $T \in \@S_U$, and $r > 0$.
\end{lemma}
\begin{proof}
Recall (\cref{sec:greytop}) that since $\^{G/U}$ is an adequate topometric quotient, sets of the form $[\pi_U(T)]_{<r}$ for $T \in \@S_U$ and $r > 0$ form an open basis in $\^{G/U}$.

First, we check that the sets $\sqsqbr{\pi_U(U_{=0}) |-> [\pi_U(T)]_{<r}}_U$ are open.  Let $h : \@M_x -> \@M_y \in \sqsqbr{\pi_U(U_{=0}) |-> [\pi_U(T)]_{<r}}_U$.  Let $\^d_U(h_U(\pi_U(x)), \pi_U(T)) < t < t+2s \le r$, and let $S \subseteq \sigma^{-1}(U_{=0})$ be open $U_{<s}$-small with $x \in S$ and $\tau(S) \subseteq S$ (else replace with $S \cap \tau^{-1}(S)$).  Then it is easily seen that
\begin{align*}
h \in \sqsqbr{[\pi_U(S)]_{<s} |-> [\pi_U(T)]_{<t}}_U \subseteq \sqsqbr{\pi_U(U_{=0}) |-> [\pi_U(T)]_{<r}}_U
\end{align*}
where the middle set is a subbasic open set in $\Hom_{G^0}(\@M)$ (recall \cref{sec:metalestr}).

Now we must show that an arbitrary subbasic open set $\sqsqbr{[\pi_U(S)]_{<s} |-> [\pi_U(T)]_{<t}}_U \subseteq \Hom_{G^0}(\@M)$, where $S, T \in \@S_U$ and $s, t > 0$, can be generated by $\sigma, \tau$ and sets of the above form.  Let $h : \@M_x -> \@M_y \in \sqsqbr{[\pi_U(S)]_{<s} |-> [\pi_U(T)]_{<t}}_U$, i.e., there is $a \in [\pi_U(S)_x]_{<s}$ so that $\^d_U(h_U(a), \pi_U(T)_x) < t$; by density of $\im(\pi_U)$ and continuity of $h_U$, we may assume $a = \pi_U(g)$ for some $g \in \sigma^{-1}(U_{=0})_x$.  Let $q > 0$ with $\^d_U(h_U(a), \pi_U(T)_x) < t-4q$.  Then by the Cauchy limit formula ($\dagger$) in \cref{thm:locpolgpd-yoneda}, there are $g \in S' \in \@S_{U<q}$ with $S' \subseteq d_U[S]_{<s}$, $V \in \@U$ with $x \in V_{=0}$ and $V \sqle S' \odot U \odot S^{\prime-1}$, and $0 < r \le q$ such that
$
\pi_U(\pi_V^{-1}([h_V(\pi_V(x))]_{<r}) \cdot S') \subseteq [\pi_U(T)_x]_{<t-4q}.
$
Pick any $k \in \pi_V^{-1}([h_V(\pi_V(x))]_{<r})$ and $l \in S'_{\sigma(k)}$, so that $\pi_U(k \cdot l) \in [\pi_U(T)]_{<t-4q}$.  By continuity, there are open $l \in S'' \subseteq S'$ and $k \in T' \subseteq \sigma^{-1}(\tau(S''))$ with $T' \in \@S_{V<q}$ such that $\pi_U(T' \cdot S'') \subseteq [\pi_U(T)]_{<t-4q}$.  We claim
\begin{align*}
h \in \sqsqbr{\pi_V(V_{=0}) |-> [\pi_V(T')]_{<r}}_V \subseteq \sqsqbr{[\pi_U(S)]_{<s} |-> [\pi_U(T)]_{<t}}_U.
\end{align*}
$\in$ is because $h_V(\pi_V(x)) \in [\pi_V(k)]_{<r} \subseteq [\pi_V(T')]_{<r}$ (by definition of $k$).
For $\subseteq$, let $h' : \@M_{x'} -> \@M_{y'} \in \sqsqbr{\pi_V(V_{=0}) |-> [\pi_V(T')]_{<r}}_V$.
Then $h'_V(\pi_V(x')) \in [\pi_V(T'_{y'})]_{<r}$, so there is some $k' \in T'_{y'}$ such that $k' \in \pi_V^{-1}([h'_V(\pi_V(x'))]_{<r})$.
Since $T' \subseteq \sigma^{-1}(\tau(S''))$, there is some $l' \in S''_{\sigma(k')}$, whence
\begin{align*}
k' \cdot l' \in \pi_V^{-1}([h'_V(\pi_V(x'))]_{<r}) \cdot S'' \subseteq \pi_V^{-1}([h'_V(\pi_V(x'))]_{<r}) \cdot S'.
\end{align*}
Pick any $g' \in S'_{x'}$.  Then by ($\dagger$) in \cref{thm:locpolgpd-yoneda}, we have
\begin{align*}
h'_U(\pi_U(g')) \in \-{\pi_U(\pi_V^{-1}([h'_V(\pi_V(x'))]_{<r}) \cdot S')}.
\end{align*}
Moreover, since $S' \in \@S_{U<q}$ and $r \le q$, by the estimate ($\mathparagraph$) in the proof of \cref{thm:locpolgpd-yoneda}, the set $\-{\pi_U(\pi_V^{-1}([h'_V(\pi_V(x'))]_{<r}) \cdot S')}$ has diameter $\le 4q$, whence
\begin{align*}
\^d_U(\pi_U(k' \cdot l'), h'_U(\pi_U(g'))) \le 4q.
\end{align*}
But $\pi_U(k' \cdot l') \in \pi_U(T' \cdot S'') \subseteq [\pi_U(T)]_{<t-4q}$, so we get $h'_U(\pi_U(g')) \in [\pi_U(T)]_{<t}$.  Since $g' \in S' \subseteq d_U[S]_{<s}$, we have $\pi_U(g') \in [\pi_U(S)]_{<s}$, so $h' \in \sqsqbr{[\pi_U(S)]_{<s} |-> [\pi_U(T)]_{<t}}_U$, as desired.
\end{proof}

By left-invariance of the metrics $d_U$ and the relations $R_{U,V,U \odot S \odot V}$, left multiplication by each $g : x -> y \in G$ descends to a homomorphism $\@M_x -> \@M_y$, yielding a canonical functor
\begin{align*}
\eta : G &--> \Hom_{G^0}(\@M) \\
(g : x -> y) &|--> \left(\begin{aligned}
\@M_x &-> \@M_y \\
\pi_U(h) &|-> \pi_U(g \cdot h)
\end{aligned}\right).
\end{align*}
Composed with the bijection in \cref{thm:locpolgpd-yoneda}, this becomes for each $x, y \in G^0$
\begin{align*}
\Phi \circ \eta : G(x, y) &--> \projlim_{U \in \@U; x \in U_{=0}} (\^{G/U})_y \\
(g : x -> y) &|--> (\pi_U(g))_U.
\end{align*}

\begin{proposition}
\label{thm:locpolgpd-yoneda-emb-dense}
$\eta : G -> \Hom_{G^0}(\@M)$ is a topological embedding with $\tau$-fiberwise dense image.
\end{proposition}
\begin{proof}
For a subbasic open set in $\Hom_{G^0}(\@M)$ as in \cref{thm:locpolgpd-yoneda-top}, we have
\begin{align*}
\eta^{-1}(\sqsqbr{\pi_U(U_{=0}) |-> [\pi_U(T)]_{<r}}_U)
&= \{g : x -> y \in G \mid x \in U_{=0} \AND \^d_U(\pi_U(g), \pi_U(T)) < r\} \\
&= \{g : x -> y \in G \mid x \in U_{=0} \AND d_U(g, T) < r\} \\
&= d_U[T]_{<r},
\end{align*}
whence by \cref{thm:locpolgpd-basis-section}, $\eta$ is an embedding.

For density, consider a basic open set
$A = \sigma^{-1}(V) \cap \tau^{-1}(W) \cap \bigcap_{i < n} \sqsqbr{\pi_{U_i}((U_i)_{=0}) |-> [\pi_{U_i}(T_i)]_{<r_i}}_{U_i} \subseteq \Hom_{G^0}(\@M)$,
where $V, W \subseteq G^0$ are open, $U_i \in \@U$, $T_i \in \@S_{U_i}$, and $r_i > 0$.  Let $y \in G^0$; we must show that if $A_y \ne \emptyset$, then there is a $g \in G$ such that $\eta(g) \in A_y$.  Let
$h : \@M_x -> \@M_y \in A_y$.
Then $y \in W$, and $x \in V \cap \bigcap_i (U_i)_{=0}$, whence by (US1) and (US3), there is some $U \in \@U$ with $x \in U_{=0}$, $U \sqle U_i$ for each $i$, and $G^0 \cap U_{=0} \subseteq V$.
Then for each $i$, we have $\pi_{U,U_i}(h_U(\pi_U(x))) = h_{U_i}(\pi_{U_i}(x)) \in [\pi_{U_i}(T_i)]_{<r_i}$, and so
$h_U(\pi_U(x)) \in \bigcap_i \pi_{U,U_i}^{-1}([\pi_{U_i}(T_i)]_{<r_i}) \subseteq \^{G/U}$.
By density of $\im(\pi_U) \subseteq \^{G/U}$, there is some $g \in \sigma^{-1}(U_{=0})_y$ such that
$\pi_U(g) \in \bigcap_i \pi_{U,U_i}^{-1}([\pi_{U_i}(T_i)]_{<r_i}) \subseteq \^{G/U}$.
Then $\sigma(g) \in G^0 \cap U_{=0} \subseteq V$, $\tau(g) = y \in W$, and $\pi_{U_i}(g) = \pi_{U,U_i}(\pi_U(g)) \in [\pi_{U_i}(T_i)]_{<r}$ for each $i$, whence $g \in A_y$.
\end{proof}

\subsection{The isomorphism groupoid}
\label{sec:locpolgpd-iso}

\begin{theorem}
\label{thm:locpolgpd-yoneda-iso}
For any open $\sigma$-locally Polish groupoid $G$ and countable $\@U, \@S_U$ satisfying (US1--3) as above, $\eta : G -> \Hom_{G^0}(\@M)$ is an isomorphism of topological groupoids $G \cong \Iso_{G^0}(\@M)$.
\end{theorem}
\begin{proof}
Since $G$ is a groupoid, $\eta$ lands in $\Iso_{G^0}(\@M) \subseteq \Hom_{G^0}(\@M)$.  By \cref{thm:locpolgpd-yoneda-emb-dense}, $\eta : G -> \Iso_{G^0}(\@M)$ is a $\tau$-fiberwise dense embedding, hence by \cref{thm:subgpd-dense} a homeomorphism.
\end{proof}

\begin{theorem}
\label{thm:locpolgpd-rep}
For any open $\sigma$-locally Polish groupoid $G$, there is a countable single-sorted relational language $\@L$, a metric $\{0, 1\}$-valued $\@L_{\omega_1\omega}$-sentence $\phi$, and a Borel equivalence of groupoids $G -> \Iso(\#U) \ltimes \Mod_\#U(\@L, \phi)$.
\end{theorem}
\begin{proof}
Combining \cref{thm:locpolgpd-yoneda-iso} with \cref{thm:sbmstr-isogpd-ff} yields $\@L$ as above and a full and faithful Borel functor $F : G -> \Iso(\#U) \ltimes \Mod_\#U(\@L)$.  By \cref{thm:functor-borel-equiv}, the essential image $[F(G^0)]_{\Iso(\#U)} \subseteq \Mod_\#U(\@L)$ is Borel, hence by \cref{thm:coskey-lupini} equal to $\Mod_\#U(\@L, \phi)$ for some metric $\{0, 1\}$-valued $\@L_{\omega_1\omega}$-sentence $\phi$, whence $F : G -> \Iso(\#U) \ltimes \Mod_\#U(\@L, \phi)$ is a Borel equivalence.
\end{proof}

\begin{corollary}
Up to Borel equivalence, the following classes of standard Borel groupoids coincide:
\begin{enumerate}
\item[(i)]  open $\sigma$-locally Polish groupoids;
\item[(ii)]  action groupoids of Polish group actions;
\item[(iii)]  groupoids of models of metric $\{0, 1\}$-valued $\@L_{\omega_1\omega}$-sentences on $\#U$.
\end{enumerate}
Hence, so do their classes of orbit equivalence relations, up to Borel bireducibility.
\end{corollary}
\begin{proof}
(i)$\implies$(iii) is by \cref{thm:locpolgpd-rep}, (iii)$\implies$(ii) is immediate, and (ii)$\implies$(i) is by the Becker--Kechris theorem \cite[5.2.1]{BK}.
\end{proof}


\begin{remark}
\label{rmk:locpolgpd-rep-io}
As in \cref{rmk:sbmstr-isogpd-ff-io}, we may arrange for the equivalence of groupoids $G -> \Iso(\#U) \ltimes \Mod_\#U(\@L, \phi)$ in \cref{thm:locpolgpd-rep} to be injective on objects (hence for the reduction between the corresponding orbit equivalence relations to be an embedding).  However, as in \cref{rmk:nonarchgpd-rep-bo}, we cannot further require bijectivity on objects.
\end{remark}

%
%
%
%
%
%

\bigskip\noindent
Department of Mathematics \\
University of Illinois at Urbana--Champaign \\
Urbana, IL 61801

\medskip\noindent
\nolinkurl{ruiyuan@illinois.edu}


\begin{thebibliography}{000000}

\bibitem[AH]{AH}  J.-M.~Albert and B.~Hart, \emph{Metric logical categories and conceptual completeness for first order continuous logic}, preprint, \url{https://arxiv.org/abs/1607.03068}, 2016.

\bibitem[Bec]{Bec}  H.~Becker, \emph{Polish group actions: dichotomies and generalized elementary embeddings}, J.\ Amer.\ Math.\ Soc.\ \textbf{11(2)}~(1998), 397--449.

\bibitem[BK]{BK}  H.~Becker and A.~S.~Kechris, \emph{The Descriptive Set Theory of Polish Group Actions}, London Math.\ Soc.\ Lecture Note Series \textbf{232}, Cambridge University Press, 1996.

\bibitem[BY08]{BY08}  I.~Ben~Yaacov, \emph{Topometric spaces and perturbations of metric structures}, Logic and Analysis \textbf{1(3--4)}~(2008), 235--272.

\bibitem[BY10]{BY10}  I.~Ben~Yaacov, \emph{Lipschitz functions on topometric spaces}, J.\ Log.\ Anal.\ \textbf{5}~(2013), 1--21.

\bibitem[BBHU]{BBHU}  I.~Ben~Yaacov, A.~Berenstein, C.~W.~Henson, and A.~Usvyatsov, \emph{Model theory for metric structures}, Model theory with applications to algebra and analysis, vol.~2, 2008, 315--427.

\bibitem[BBM]{BBM}  I.~Ben~Yaacov, A.~Berenstein, and J.~Melleray, \emph{Polish topometric groups}, Trans.\ Amer.\ Math.\ Soc.\ \textbf{365(7)}~(2013), 3877--3897.

\bibitem[BDNT]{BDNT}  I.~Ben~Yaacov, M.~Doucha, A.~Nies, and T.~Tsankov, \emph{Metric Scott analysis}, Adv.\ Math.\ \textbf{318(1)}~(2017), 46--87.

\bibitem[BYM]{BYM}  I.~Ben~Yaacov and J.~Melleray, \emph{Grey subsets of Polish spaces}, J.\ Symb.\ Logic \textbf{80(4)}~(2015), 1379--1397.

\bibitem[Bun]{Bun}  M.~R.~Buneci, \emph{A Urysohn type lemma for groupoids}, Theory Appl.\ Categ.\ \textbf{32(28)}~(2017), 970--994.

\bibitem[Ch1]{Cscc}  R.~Chen, \emph{Borel functors, interpretations, and strong conceptual completeness for $\@L_{\omega_1\omega}$}, to appear in Trans.\ Amer.\ Math.\ Soc., \url{https://arxiv.org/abs/1710.02246}, 2019.

\bibitem[Ch2]{Cqpol}  R.~Chen, \emph{Notes on quasi-Polish spaces}, preprint, \url{https://arxiv.org/abs/1809.07440}, 2018.

\bibitem[Cho]{Cho}  S.~Cho, \emph{Categorical semantics of metric spaces and continuous logic}, preprint, \url{https://arxiv.org/abs/1901.09077}, 2019.

\bibitem[CL]{CL}  S.~Coskey and M.~Lupini, \emph{A López-Escobar theorem for metric structures, and the topological Vaught conjecture}, Fund.\ Math.\ \textbf{234}~(2016), 55--72.

\bibitem[deB]{deB}  M.~de~Brecht, \emph{Quasi-Polish spaces}, Ann.\ Pure Appl.\ Logic \textbf{164(3)}~(2013), 356--381.

\bibitem[dBK]{dBK}  M.~de~Brecht and T.~Kawai, \emph{On the commutativity of the powerspace constructions}, preprint, \url{https://arxiv.org/abs/1709.06226}, 2019.

\bibitem[EFP$^+$]{EFPRTT}  G.~A.~Elliott, I.~Farah, V.~Paulsen, C.~Rosendal, A.~S.~Toms, and A.~Törnquist, \emph{The isomorphism relation for separable C*-algebras}, Math.\ Res.\ Lett.\ \textbf{20(6)}~(2013), 1071--1080.

\bibitem[Gao]{Gao}  S.~Gao, \emph{Invariant descriptive set theory}, Pure and applied mathematics, vol.~293, CRC Press, 2009.

\bibitem[GK]{GK}  S.~Gao and A.~S.~Kechris, \emph{On the classification of Polish metric spaces up to isometry}, Mem.\ Amer.\ Math.\ Soc.\ \textbf{161(766)}~(2003).

\bibitem[Gro]{Gro}  M.~Gromov, \emph{Metric structures for Riemannian and non-Riemannian spaces}, Progress in Math., vol.~152, Birkhäuser, 1999.

\bibitem[Hec]{Hec}  R.~Heckmann, \emph{Spatiality of countably presentable locales (proved with the Baire category theorem)}, Math.\ Structures Comput.\ Sci.\ \textbf{25(7)}~(2015), 1607--1625.

\bibitem[Hen]{Hen}  S.~Henry, \emph{Localic metric spaces and the localic Gelfand duality}, Adv.\ Math.\ \textbf{294}~(2016), 634--688.

\bibitem[HK]{HK}  G.~Hjorth and A.~S.~Kechris, \emph{Borel equivalence relations and classifications of countable models}, Ann.\ Pure Appl.\ Logic \textbf{82}~(1996), 221--272.

\bibitem[IMI]{IMI}  A.~Ivanov and B.~Majcher-Iwanow, \emph{Polish $G$-spaces and continuous logic}, Ann.\ Pure Appl.\ Logic \textbf{168(4)}~(2017), 749--775.

\bibitem[J89]{Jclsubgpd}  P.~T.~Johnstone, \emph{A constructive ``Closed subgroup theorem'' for localic groups and groupoids}, Cah.\ Topol.\ Géom.\ Différ.\ Catég.\ \textbf{30(1)}~(1989), 3--23.

\bibitem[J02]{Jeleph}  P.~T.~Johnstone, \emph{Sketches of an elephant: a topos theory compendium}, Oxford Logic Guides, vol.~43--44, Oxford University Press, 2002.

\bibitem[Kat]{Kat}  M.~Katětov, \emph{On universal metric spaces}, in: \emph{General Topology and its Relations to Modern Analysis and Algebra VI: Proc.\ of the Sixth Prague Topological Symposium 1986}, Z.~Frolík (ed.), Heldermann Verlag, 1988, 323--330.

\bibitem[Kec]{Kcdst}  A.~S.~Kechris, \emph{Classical descriptive set theory}, Graduate Texts in Mathematics, vol.~156, Springer-Verlag, 1995.

\bibitem[Kec$'$]{Kcdst'}  A.~S.~Kechris, \emph{Classical descriptive set theory: corrections and updates}, \url{http://www.math.caltech.edu/~kechris/papers/CDST-corrections.pdf}.

\bibitem[Kel]{Kelley}  J.~L.~Kelley, \emph{General topology}, van~Nostrand, 1955.

\bibitem[LE]{LE}  E.~G.~K.~Lopez-Escobar, \emph{An interpolation theorem for denumerably long formulas}, Fund.\ Math.\ \textbf{57}~(1965), 253--272.

\bibitem[Lup]{Lup}  M.~Lupini, \emph{Polish groupoids and functorial complexity}, Trans.\ Amer.\ Math.\ Soc.\ \textbf{369(9)}~(2017), 6683--6723.

\bibitem[Mac]{Mac}  S.~Mac~Lane, \emph{Categories for the Working Mathematician}, Grad.\ Texts in Math.\ \textbf{5},
Springer, 2nd ed., 1998.

\bibitem[Me1]{Me1}  J.~Melleray, \emph{On the geometry of Urysohn's universal metric space}, Topology Appl.\ \textbf{154(2)}~(2007), 384--403.

\bibitem[Me2]{Me2}  J.~Melleray, \emph{A note on Hjorth's oscillation theorem}, J.\ Symb.\ Logic \textbf{75}~(2010), 1359--1365.

\bibitem[Mo1]{Mo1}  I.~Moerdijk, \emph{The classifying topos of a continuous groupoid. I}, Trans.\ Amer.\ Math.\ Soc.\ \textbf{310}~(1988), 629--668.

\bibitem[Mo2]{Mo2}  I.~Moerdijk, \emph{The classifying topos of a continuous groupoid. II}, Cah.\ Topol.\ Géom.\ Différ.\ Catég.\ \textbf{31}~(1990), 137--168.

\bibitem[PP]{PP}  J.~Picado and A.~Pultr, \emph{Frames and locales: topology without points}, Springer, 2012.

\bibitem[Ram]{Ram}  A.~Ramsay, \emph{The Mackey--Glimm dichotomy for foliations and other Polish groupoids}, J.\ Funct.\ Anal.\ \textbf{94}~(1990), 358--374.

\bibitem[Ten]{Ten}  B.~R.~Tennison, \emph{Sheaf theory}, London Math.\ Soc.\ Lecture Note Series \textbf{20}, Cambridge University Press, 1975.

\bibitem[Usp]{Usp}  V.~V.~Uspenskij, \emph{On subgroups of minimal topological groups}, Topology Appl.\ \textbf{155(14)}~(2008), 1580--1606.


\end{thebibliography}
\end{document}